\documentclass[a4paper,11pt]{article}

\usepackage{amsmath,amssymb,amsthm}
\usepackage{MnSymbol,mathtools,dsfont}
\usepackage{paralist,enumitem,bm,placeins,url}
\usepackage{color,graphicx,subfig} 
\usepackage[margin=30mm,top=20mm]{geometry}
\usepackage[font=small]{caption}
\usepackage{cite}
\usepackage{diagbox}
\allowdisplaybreaks
\usepackage{array}
\newcolumntype{C}[1]{>{\centering\arraybackslash}p{#1}}
\usepackage{tocbasic}

\usepackage{hyperref} 
\makeatletter
\hypersetup{%
	colorlinks	=true,
	linkcolor	=blue,%
	citecolor	=red,%
	urlcolor	=blue
}
\makeatother

\newcommand{\Om}{\Omega}
\newcommand{\Ga}{\Gamma}
\newcommand{\LL}{\mathcal{L}}

\newcommand{\V}{\mathcal{V}}
\newcommand{\W}{\mathcal{W}}
\renewcommand{\SS}{\mathcal{S}}

\newcommand{\N}{\mathbb{N}}
\newcommand{\R}{\mathbb{R}}

\newcommand{\dG}{\, d\Gamma}

\newcommand{\ds}{\, ds}
\newcommand{\dx}{\, dx}
\newcommand{\dt}{\, dt}

\newcommand{\pd}{\partial}

\newcommand{\pdnu}{\pd_{\bm{n}}}
\newcommand{\abs}[1]{\left| #1 \right|}
\newcommand{\norm}[1]{\| #1 \|}
\newcommand{\inn}[2]{ \langle #1 , #2  \rangle}
\newcommand{\scp}[2]{ \left( #1 , #2  \right)}
\newcommand{\bigscp}[2]{\big( #1 , #2 \big)}
\newcommand{\mean}[1]{\langle #1 \rangle}
\newcommand{\eps}{\varepsilon}
\newcommand{\Lx}{\Delta}

\newcommand{\LB}{\Delta_{\Gamma}}

\newcommand{\Ll}[1]{#1^{L}}

\newcommand{\Th}{\mathcal{T}_h}

\newcommand{\n}{\mathbf{n}}
\newcommand{\grad}{\nabla}
\newcommand{\gradg}{\nabla_\Gamma}
\newcommand{\mo}{m_\Omega}
\newcommand{\mg}{m_\Gamma}

\newcommand{\del}{\partial}
\newcommand{\delt}{\partial_t}
\newcommand{\deln}{\partial_\n}

\newcommand{\lu}{\bar u}

\newcommand{\lmu}{\bar \mu}
\newcommand{\ltheta}{\bar \theta}
\newcommand{\lthN}{{\bar\theta}_{N}}

\newcommand{\hu}{\hat u}

\newcommand{\hmu}{\hat \mu}
\newcommand{\htheta}{\hat \theta}

\newcommand{\D}{\mathcal D}

\newcommand{\Vk}{{\mathcal{V}^\kappa}}
\newcommand{\Vm}{{\mathcal{V}^\kappa_{m}}}

\newcommand{\Wk}{{\mathcal{W}^\kappa_\beta}}
\newcommand{\Wm}{{\mathcal{W}^\kappa_{\beta,m}}}
\newcommand{\Wo}{{\mathcal{W}^\kappa_{\beta,0}}}
\newcommand{\Wod}{{(\Wo)^{-1}}}

\renewcommand{\H}{\mathcal{H}}
\newcommand{\Hm}{\H_{\beta,m}}
\newcommand{\Ho}{\H_{\beta,0}}

\newcommand{\intO}{\int_\Omega}
\newcommand{\intG}{\int_\Gamma}

\newcommand{\Xk}{\mathcal{X}^\kappa}
\newcommand{\Yk}{\mathcal{Y}^\kappa}

\newcommand{\wto}{\rightharpoonup}

\newcommand{\emb}{\hookrightarrow}

\newcommand{\PM}{\mathcal{P}_M}

\newcommand{\mom}{m_\Om}
\newcommand{\mga}{m_\Ga}

\newcommand{\labitem}[2]{%
\def\@itemlabel{\textbf{#1}}
\item
\def\@currentlabel{#1}\label{#2}}
\makeatother

\newcommand{\bs}[1]{\boldsymbol{#1}}
\newcommand{\mb}[1]{\mathbf{#1}}
\newcommand{\rkla}[1]{{\left(#1\right)}}
\newcommand{\trkla}[1]{{(#1)}}
\newcommand{\gkla}[1]{{\left\{#1\right\}}}
\newcommand{\tgkla}[1]{{\{#1\}}}

\newcommand{\ekla}[1]{{\left[#1\right]}}
\newcommand{\tekla}[1]{{[#1]}}
\newcommand{\tabs}[1]{|{#1}|}

\newcommand{\diam}{\operatorname{diam}}

\newcommand{\muO}{\mu} 
\newcommand{\muG}{\theta} 
\newcommand{\muOh}{\mu_h} 
\newcommand{\muGh}{\theta_h} 
\newcommand{\muOvec}{P} 
\newcommand{\muGvec}{\Theta} 

\newcommand{\iOmega}{\int_\Omega}
\newcommand{\iGamma}{\int_\Gamma}

\newcommand{\UhO}{U\h^\Omega}
\newcommand{\UhG}{U\h^\Gamma}

\newcommand{\h}{_h}

\newcommand{\Ihop}{\mathcal{I}_h}
\newcommand{\Ih}[1]{\Ihop\gkla{#1}}
\newcommand{\IhGop}{\mathcal{I}_h^\Gamma}
\newcommand{\IhG}[1]{\IhGop\gkla{#1}}

\newcommand{\ids}{I}

\newcommand{\dtaum}{\partial_{\tau}^-}

\newcommand{\nn}{^{n}}
\newcommand{\no}{^{n-1}}

\newcommand{\MOmega}{\mathbf{M}_\Omega} 
\newcommand{\MGamma}{\mathbf{M}_\Gamma} 
\newcommand{\LOmega}{\mathbf{L}_\Omega} 
\newcommand{\LGamma}{\mathbf{L}_\Gamma} 

\newcommand{\restr}[2]{\ensuremath{
  \left.\kern-\nulldelimiterspace 
  #1 
  \vphantom{\big|} 
  \right|_{{#2},h} 
  }}

\newcommand{\trestr}[2]{\ensuremath{
		#1 
		|_{{#2},h} 
}}
  
\newcommand{\extend}[2]{\ensuremath{
  \left.\kern-\nulldelimiterspace 
  #1 
  \vphantom{\big|} 
  \right|^{{#2},h} 
  }}

\newcommand{\textend}[2]{\ensuremath{
		#1 
		|^{{#2},h} 
}}

\newcommand{\GammaSymb}{\Gamma} 
\newcommand{\OmegaSymb}{\overline{\Omega}} 

\newcommand{\Overset}[3][0pt]{ %
			\ensuremath{\overset{\raise#1
			\hbox{\scriptsize\ensuremath{#2}}}{#3}}
}

\newcommand{\InnerSymb}{ \Overset[-3pt]{\circ}{\Omega}} 

\newcommand{\SubInnerSymb}{\raisebox{-2pt}{\scriptsize\ensuremath{\InnerSymb}}}
\newcommand{\SubInnerSymbH}{\raisebox{-2pt}{\scriptsize\ensuremath{\InnerSymb,h}}}
\newcommand\RInn{R\SubInnerSymb}

\newcommand{\Sub}[2][0pt]{%
	\raisebox{#1}{\scriptsize\ensuremath{#2}}
}
  
\newcommand{\eye}{\mathds{1}}

\newcommand{\tl}{^{\tau}}
\newcommand{\tp}{^{\tau,+}}
\newcommand{\tm}{^{\tau,-}}
\newcommand{\tpm}{^{\tau,(\pm)}}

\theoremstyle{plain}
\newtheorem{thm}{Theorem}[section]
\newtheorem{prop}[thm]{Proposition}
\newtheorem{lemma}[thm]{Lemma}
\newtheorem{cor}[thm]{Corollary}

\theoremstyle{definition}
\newtheorem{remark}{Remark}[section]

\numberwithin{equation}{section}

\makeatletter
\renewenvironment{proof}[1][\proofname]{%
	\par\pushQED{\qed}\normalfont%
	\topsep6\p@\@plus6\p@\relax
	\trivlist\item[\hskip\labelsep\bfseries#1\@addpunct{.}]%
	\ignorespaces
}{%
	\popQED\endtrivlist\@endpefalse
}
\makeatother

\makeatletter
\renewcommand\paragraph{\@startsection{paragraph}{4}{\z@}%
	{1ex \@plus1ex \@minus.2ex}%
	{-1em}%
	{\normalfont\normalsize\bfseries}}
\renewcommand\subparagraph{\@startsection{paragraph}{4}{\z@}%
	{1ex \@plus1ex \@minus.2ex}%
	{-1em}%
	{\normalfont\normalsize\itshape\quad}}
\makeatother

\usepackage{setspace}

\begin{document}
\title{ \bfseries Phase-field dynamics with transfer of materials:\\ The Cahn--Hilliard equation with reaction rate dependent dynamic boundary conditions}

\author{Patrik Knopf \footnotemark[1] \and Kei Fong Lam \footnotemark[2] \and Chun Liu \footnotemark[3] \and Stefan Metzger \footnotemark[4]}

\date{\today}

\renewcommand{\thefootnote}{\fnsymbol{footnote}}

\footnotetext[1]{Fakult\"at f\"ur Mathematik, Universit\"at Regensburg, 93053 Regensburg, Germany, \\ \tt(\href{mailto:Patrik.Knopf@ur.de}{patrik.knopf@ur.de}) }
\footnotetext[2]{Department of Mathematics, Hong Kong Baptist University, Kowloon Tong, Hong Kong, \\
	\tt(\href{mailto:akflam@hkbu.edu.hk}{akflam@hkbu.edu.hk}) }
\footnotetext[3]{Department of Applied Mathematics, Illinois Institute of Technology, Chicago, IL 60616, \\
	\tt(\href{mailto:cliu124@iit.edu}{cliu124@iit.edu}) }
\footnotetext[4]{Department Mathematik, Friedrich-Alexander-Universit\"at Erlangen-N\"urnberg, 91058 Erlangen, \\ Germany, 
	\tt(\href{mailto:stefan.metzger@fau.de}{stefan.metzger@fau.de}) }

\maketitle

\begin{center}
	\textit{This is a preprint version of the paper. Please cite as:} \\  
	P. Knopf, K.F. Lam, C. Liu, S. Metzger,\\ 
	ESAIM: Mathematical Modelling and Numerical Analysis 55(1): 229-282, 2021  \\ 
	\url{https://doi.org/10.1051/m2an/2020090}
\end{center}

\bigskip

\begin{abstract}
	The Cahn--Hilliard equation is one of the most common models to describe phase separation
processes of a mixture of two materials. For a better description of short-range interactions between the
material and the boundary, various dynamic boundary conditions for the Cahn--Hilliard equation have been proposed and investigated in recent
times. Of particular interests are the model by Goldstein, Miranville and Schimperna (Physica D, 2011) and the model by Liu and Wu (Arch.~Ration.~Mech.~Anal., 2019). Both of these models satisfy similar physical properties but differ greatly in their mass conservation behaviour. In this paper we introduce a new model which interpolates between these previous models, and investigate analytical properties such as the existence of unique solutions and convergence to the previous models mentioned above in both the weak and the strong sense. For the strong convergences we also establish rates in terms of the interpolation parameter, which are supported by numerical simulations obtained from a fully discrete, unconditionally stable and convergent finite element scheme for the new interpolation model. 
\end{abstract}



\newpage

\begin{small}
	\setcounter{tocdepth}{2}
	\hypersetup{linkcolor=black}
	\tableofcontents
\end{small}

\setlength\parskip{1ex}

\section{Introduction}
The Cahn--Hilliard equation was originally introduced in \cite{cahn-hilliard} to model phase separation and de-mixing processes in binary alloys, while later applications have been found in mathematical models of phenomena arising in material sciences, life sciences and image processing. 
In certain applications (e.g., in hydrodynamic applications such as contact line problems), it turned out to be essential to model short-range interactions of the binary mixture with the solid wall of the container more accurately. To this end, several dynamic boundary conditions have recently been proposed and investigated in the literature. Below we review two such models in more detail and introduce a new system with dynamic boundary conditions which can be regarded as an interpolation between these two previous models.

The standard Cahn--Hilliard equation as introduced in \cite{cahn-hilliard} reads as follows:
\begin{subequations}
\label{CH}
\begin{alignat}{2}
	\label{CH:1}
	&\delt u = \mom \Lx \mu &&\quad\text{in } Q_T := \Omega\times (0,T),\\
	\label{CH:2}
	&\mu = -\eps \Lx u + \eps^{-1} F'(u) &&\quad\text{in } Q_T,\\
	\label{CH:3}
	&u\vert_{t=0}=u_0 &&\quad\text{in } \Omega.
\end{alignat}
\end{subequations}
Here, $\Omega\subset\R^d$ (where $d\in\{2,3\}$) denotes a bounded domain with boundary $\Gamma:=\del\Omega$ whose unit outer normal vector field is denoted by $\n$. 
The functions $u=u(x,t)$ and $\mu=\mu(x,t)$ depend on time $t\in [0,T]$ (with fixed but arbitrary $T>0$) and position $x\in\Omega$. 
The symbol $\Lx$ denotes the Laplace operator in $\Omega$. 
The symbol $\mo$ denotes a mobility parameter which is assumed to be a positive constant. This is a typical assumption, although non-constant mobilities find a use in some situations (see e.g.~\cite{elliotgarcke}). 
In order to describe a mixture of two materials, the phase-field variable $u$ represents the difference of two local
relative concentrations. 
After a short period of time, the solution $u$ will attain values
close to $\pm 1$ in large regions of the domain $\Omega$. These regions, which correspond to the pure phases of the materials, are 
separated by a small interfacial region whose thickness is proportional to a small parameter $\eps>0$.
As the time evolution of the phase-field variable $u$ is governed by chemical reactions, 
the function $\mu$ stands for the chemical potential in the bulk (i.e., in $\Omega$). 
It can be expressed as the Fr\'echet derivative of the following free energy of Ginzburg--Landau type:
\begin{align}
E_\text{bulk}(u) = \int_\Omega \frac \eps 2|\grad u|^2 + \frac 1 \eps F(u) \dx.
\end{align}
In this context, the function $F$ represents the bulk potential which usually has a double-well shape, i.e., it attains its minima at $-1$ and $1$ and has a local maximum at $0$. A typical choice is the smooth double-well potential $F(s) = \frac{1}{4}(s^2-1)^2$ (see Remark~\ref{rem:pot}). As the time-evolution of $u$ is considered in a bounded
domain, suitable boundary conditions have to be imposed. The homogeneous Neumann conditions
\begin{align}
\label{HNC}
\del_\n \mu = 0, \quad
\del_\n u = 0 \quad\text{on}\;\; \Sigma_T:=\Gamma\times(0,T).
\end{align}
are the classical choice. The no-flux condition $\eqref{HNC}_1$ leads to mass conservation in the bulk
\begin{align}
\int_\Omega u(t) \dx = \int_\Omega u(0) \dx, \quad t\in[0,T]
\end{align}
and both conditions in \eqref{HNC} imply that the bulk free energy satisfies the following maximal dissipation law:
\begin{align}
\frac{d}{dt} E_\text{bulk}\big(u(t)\big) + \mom \intO |\grad\mu(t)|^2 \dx = 0, \quad t\in(0,T).
\end{align}
We point out that the Cahn--Hilliard equation subject to the boundary conditions \eqref{HNC} can be interpreted as a gradient flow of type $H^{-1}$ of the bulk free energy $E_\text{bulk}$ \cite{Cowan}.

The Cahn--Hilliard equation \eqref{CH} with homogeneous Neumann conditions \eqref{HNC} is already very well understood and there exists an extensive literature (see, e.g., 
\cite{Abels-Wilke,Bates-Fife,Cherfils,elliotgarcke,elliotzheng,zheng,rybka,pego,miranville-book}). However, it became clear that this model is not satisfactory in some situations as it neglects certain
influences of the boundary to the bulk dynamics, such as separate chemical reactions occurring on the boundary are not taken into account. To provide a better description of
interactions between the solid wall and the binary mixture, physicists 
suggested to add a surface free energy that is also of Ginzburg--Landau type (cf. \cite{Fis1,Fis2,Kenzler}):
\begin{equation}\label{DEF:ENS}
 E_\text{surf}(u)  = \intG \frac{\kappa \delta }{2} \abs{\gradg u}^2 + \frac{1}{\delta} G(u) 
  \dG.
\end{equation}
Hence, the total free energy $E=E_\text{bulk}+E_\text{surf}$ reads as
\begin{equation}\label{DEF:EN}
E(u) 
= \int_\Omega \frac \eps 2|\grad u|^2 + \frac 1 \eps F(u) \dx
	+ \intG \frac{\kappa \delta }{2} \abs{\gradg u}^2 + \frac{1}{\delta} G(u) \dG.
\end{equation} 
Here $\gradg$ denotes the surface gradient on $\Gamma$, $G$
is a surface potential, $\kappa$ is a non-negative parameter acting as a weight for surface diffusion effects and $\delta>0$ is related to the thickness of the interfacial regions on the boundary. In the case $\kappa = 0$ this problem
is related to the moving contact line problem \cite{thompson-robbins}. In view of this energy $E$, various dynamic boundary conditions
have been proposed and analysed in the literature, of which we mention
\cite{colli-fukao-ch,colli-gilardi,Gal1,GalWu,colli-gilardi-sprekels,liero,mininni,miranville-zelik,motoda,racke-zheng,WZ, FukaoWu, Wu}. 

In particular, we now want to highlight two Cahn--Hilliard models with dynamic boundary conditions in more detail. In both models, the dynamic boundary  conditions have  a Cahn--Hilliard type structure and both systems can be interpreted as a gradient flow of type $H^{-1}$ of the total free energy $E$ (see \cite[s.~3]{GK}). However, these models have completely different mass conservation properties.

\paragraph{The GMS model.}

The following model with dynamic boundary condition has been introduced by G.~Goldstein, A.~Miranville and G.~Schimperna \cite{GMS}: 
\begin{subequations}\label{CH:GMS}
	\begin{alignat}{3}
	&u_t = \mom \Lx  \mu, 
	&&\quad \mu = - \eps \Lx u + \eps^{-1} F'( u) &&\quad \text{ in } Q_T,  \label{CH:GMS:1} \\
	& u_t = \mga \LB \theta - \beta \mom \pdnu \mu,
	&&\quad \theta = - \delta \kappa \LB u + \delta^{-1} G'(u) + \eps \pdnu u   &&\quad \text{ on } \Sigma_T, \label{CH:GMS:2} \\
	& \mu\vert_{\Sigma_T} = \beta\theta  &&\quad \text{on } \Sigma_T, \label{CH:GMS:3} \\ 
	& u(0) = u_0   &&\quad \text{on } \overline{\Omega}, \label{CH:GMS:4}
	\end{alignat}
\end{subequations}
where $\beta> 0$. The symbol $\LB$ denotes the Laplace--Beltrami operator on the surface $\Gamma$.  For convenience, we use the authors' initials and call it the \textit{GMS model}. 
It can be regarded as an extension of a model previously introduced by Gal \cite{Gal1} who proposed the equation $ u_t = -\beta \pdnu \mu + \gamma \mu $, for some constant $\gamma$, instead of $\eqref{CH:GMS:2}_1$.
In \eqref{CH:GMS}, the parameter $\mga$ denotes the mobility on the boundary and is assumed to be a positive constant. To describe chemical reactions occurring only at the boundary, an additional chemical potential $\theta$ has been introduced, and so, chemical reactions between the bulk and the surface are taken into account by the coupling condition \eqref{CH:GMS:3}. This means that in this model, the chemical potentials in the bulk and on the boundary can differ by the factor $\beta$, i.e., they are directly proportional. In \cite{GMS}, $\beta$ is even allowed to be a uniformly positive function in $L^\infty(\Ga)$. 
 We can thus say that, by the relation \eqref{CH:GMS:3}, the potentials $\mu$ and $\theta$ are in a chemical equilibrium. 

We observe that a (sufficiently regular) solution to the GMS equation satisfies the mass conservation law
\begin{align}
	\label{GMS:MASS}
	\beta \intO u(t) \dx + \intG u(t) \dG = \beta \intO u(0) \dx + \intG u(0) \dG, \quad t\in [0,T],
\end{align}
which allows one to interpret the parameter $\beta$ as a weight of the bulk mass compared to the surface mass. Moreover, the maximal energy dissipation law 
\begin{align}
	\label{GMS:NRG}
	\frac{d}{dt} E\big(u(t)\big) + \mom \intO |\grad\mu(t)|^2 \dx +  \mga \intG|\gradg \theta(t)|^2 \dG = 0
\end{align}
is satisfied for all $t\in[0,T]$. In particular, we observe that the dissipation rate is greatly influenced by the values of the mobilities $\mom$ and $\mga$. 

In addition to \cite{GMS}, the GMS model is also discussed in the recent book \cite{miranville-book}.

\paragraph{The LW model.}

Another model with dynamic boundary condition has been derived by an energetic variational approach by the third author and H.~Wu \cite{LW}:  
\begin{subequations}
	\label{CH:LW}
	\begin{alignat}{3} 
	&u_t = \mom\Lx  \mu, 
	&&\quad \mu = - \eps \Lx u + \eps^{-1} F'( u) &&\quad \text{ in } Q_T,  \label{CH:LW:1} \\
	& u_t = \mga\LB \theta,
	&&\quad \theta = - \delta \kappa \LB u + \delta^{-1} G'(u) + \eps \pdnu u   &&\quad \text{ on } \Sigma_T, \label{CH:LW:2} \\
	& \pdnu \mu = 0 &&\quad \text{on } \Sigma_T, \label{CH:LW:3} \\ 
	& u(0) = u_0   &&\quad \text{on } \overline{\Omega}, \label{CH:LW:4}
	\end{alignat}
\end{subequations}
which we will refer to as the \textit{LW model}. Again, the function $\theta$ can be interpreted as the chemical potential on the boundary $\Gamma$. The crucial difference to the GMS model is that \eqref{CH:GMS:3} is replaced by the no mass flux condition $\deln\mu = 0$. This means that the chemical potentials $\mu$ and $\theta$ are not directly coupled. However, mechanical interactions between the bulk and the surface materials are still taken into account by the trace relation for the phase-field variables. This is reflected in the equations as the elliptic subproblems $\big(\eqref{CH:LW:1}_1,\eqref{CH:LW:3}\big)$ and $\eqref{CH:LW:2}_1$ are coupled only by the trace relation for $u_t$.

Compared to \eqref{GMS:MASS}, we obtain distinctly different mass conservation laws
\begin{align}
\label{LW:MASS}
\intO u(t) \dx = \intO u(0) \dx 
\quad\text{and}\quad 
\intG u(t) \dG = \intG u(0) \dG, \quad t\in [0,T],
\end{align}
meaning that the bulk mass and the surface mass are conserved \textit{separately}. However, the maximal energy dissipation law \eqref{GMS:NRG} is still satisfied by solutions of this system.

 For an efficient numerical treatment of system \eqref{CH:LW}, we refer the reader to \cite{Metzger2019}. 

Let us mention that a variant of the system \eqref{CH:LW} was proposed and investigated in \cite{KL}, where equation $\eqref{CH:LW:2}$ is replaced by 
\begin{align*} 
	v_t = \mga \LB \theta,
	\quad \theta = - \delta \kappa \LB v + \delta^{-1} G'(v) + \eps \pdnu u   \quad \text{ on } \Sigma_T,
\end{align*}
with a function $v$ that can be interpreted as the difference in volume fractions of two different materials restricted to the boundary. The relation between $u$ and $v$ is described by the Robin type transmission condition
\begin{align*} 
	\eps K\deln u = H(v) - u \quad \text{ on } \Sigma_T
\end{align*}
with $K>0$ and a function $H\in C^2(\R)$ satisfying suitable growth conditions. In particular, it is rigorously established in \cite{KL} that, in the case $H(s) = s$, solutions of this model converge to solutions of \eqref{CH:LW} in the limit $K\to 0$ in some suitable sense.

\paragraph{A more general class of  dynamic boundary conditions  based on finite, positive  reaction rates.}
To provide a more general description of the interactions between the materials in the bulk and the materials on the surface, we now propose that $\mu$ and $\theta$ are coupled by the Robin type boundary condition $L \pdnu  \mu = \beta  \theta -  \mu$ where $L>0$ acts as a relaxation parameter.
The system of equations then reads as 
\begin{subequations}\label{CH:INT}
\begin{alignat}{3}
& u_t = \mom\Lx  \mu, 
&&\quad  \mu = - \eps \Lx  u + \eps^{-1} F'( u) 
&&\quad \text{ in } Q_T, \label{CH:INT:1}\\
& u_t = \mga \LB  \theta - \beta \mom \pdnu\mu, 
&&\quad  \theta = - \delta\kappa \LB  u + \delta^{-1} G'( u) +  \eps  \pdnu  u 
&&\quad \text{ on } \Sigma_T, \label{CH:INT:2}\\
&L \pdnu  \mu = \beta  \theta -  \mu &&\quad \text{on } \Sigma_T, \label{CH:INT:3} \\
& u(0) = u_0 &&\quad \text{in } \overline{\Omega}, \label{CH:INT:4}
\end{alignat}
\end{subequations}
where $\beta\neq 0$ and $L>0$.  Here, in contrast to the GMS model \eqref{CH:GMS}, the chemical potentials $\mu$ and $\theta$ are generally not directly proportional, i.e., they are not in equilibrium.
Reactions between the materials are taken into account by the relation \eqref{CH:INT:3} where the constant $1/L$ can be interpreted as the reaction rate.  Here, the term \textit{reactions} is to be understood in a general sense including chemical reactions as well as adsorption or desorption processes. The mass flux $\deln \mu$, i.e., the motion of the materials towards and away from the boundary, is directly driven by differences in the chemical potentials. 

We observe that solutions of \eqref{CH:INT} satisfy the same mass conservation law \eqref{GMS:MASS} as solutions of the GMS model \eqref{CH:GMS}. However, we obtain an additional term in the dissipation rate depending on the relaxation parameter $L$. To be precise, it holds that 
\begin{align}
\label{INT:NRG}
\frac{d}{dt} E\big(u(t)\big) + \mom \intO |\grad\mu(t)|^2 \dx + \mga \intG |\gradg \theta(t)|^2 \dG + \frac{\mom}{L} \intG (\beta\theta-\mu)^2 \dG = 0
\end{align}
for all $t\in[0,T]$.  In particular, this implies that the total free energy $E$ is decreasing along solutions and since it is bounded from below (at least for reasonable choices of $F$ and $G$), we infer that $\frac{d}{dt} E(u(t))$ converges to zero as $t\to\infty$. As a consequence, the chemical potentials will tend to the equilibrium $\mu = \beta \theta$ over the course of time.

The Robin type condition \eqref{CH:INT:3} now allows us to establish a connection between the GMS model \eqref{CH:GMS} and the LW model \eqref{CH:LW} despite their very different chemical and physical properties. Suppose that $\beta>0$ and that $(u^L, \mu^L, \theta^L)$ is a solution of the system \eqref{CH:INT} corresponding to the parameter $L>0$.
Let $(u_*, \mu_*, \theta_*)$ denote its formal limit as $L\to 0$ and let $(u^*, \mu^*, \theta^*)$ denote its formal limit as $L\to\infty$. 
Passing to the limit in the Robin boundary condition, we deduce that 
\begin{align*}
\beta\theta_* = \mu_* \quad\text{on } \Sigma_T
\quad\text{and}\quad
\pd_n \mu^* = 0\quad\text{on } \Sigma_T.
\end{align*}
 This corresponds to the limit cases of instantaneous reactions ($1/L\to\infty$), where the chemical potentials are always in equilibrium, and  a vanishing reaction rate  ($1/L\to0$).
We infer that $(u_*, \mu_*, \theta_*)$ is a solution to the GMS model while $(u^*, \mu^*, \theta^*)$ is a solution to the LW model. These formal considerations are established rigorously in Section~4. In this regard, the Cahn--Hilliard system \eqref{CH:INT} can be interpreted as an \textit{interpolation} between the GMS model and the LW model  by
using finite, positive reaction rates.  

\paragraph{Structure of this paper.} Our paper is structured as follows: In Section~2 we introduce some notation, assumptions, preliminaries and important tools. Section~3 is devoted to the existence, uniqueness  and regularity of weak solutions to \eqref{CH:INT}, 
as well as a summary of the well-posedness results for the GMS model \eqref{CH:GMS} and the LW model \eqref{CH:LW}.
In Section~4 we investigate the asymptotic limits $L\to\infty$ and $L\to 0$, establishing also convergence rates for these limits. In Section~5 we present an efficient, unconditionally stable numerical scheme to solve the problems \eqref{CH:GMS}, \eqref{CH:LW}, and \eqref{CH:INT}, 
demonstrating also the convergence of discrete solutions.  Finally, in Section~6 we present and interpret the plots of several numerical simulations to illustrate the convergence results for $L\to 0$ and $L\to\infty$. We also measure some of the corresponding numerical convergence rates and discuss to what extent they match our analytical predictions.

\section{Notation and preliminaries}

\paragraph{Notation.} Throughout this paper we use the following notation:  For any $1 \leq p \leq \infty$ and $k \geq 0$, the standard Lebesgue and Sobolev spaces defined on $\Omega$ are denoted as $L^p(\Omega)$ and $W^{k,p}(\Omega)$, along with the norms $\norm{\cdot}_{L^p(\Omega)}$ and $\norm{\cdot}_{W^{k,p}(\Omega)}$.  For the case $p = 2$, these spaces become Hilbert spaces and we use the notation $H^k(\Omega) = W^{k,2}(\Omega)$. Note that $H^0(\Omega)$ can be identified with $L^2(\Omega)$. A similar notation is used for Lebesgue and Sobolev spaces on $\Gamma$.  For any Banach space $X$, we denote its dual space by $X'$ and the associated duality pairing by $\inn{\cdot}{\cdot}_X$.  If $X$ is a Hilbert space, we denote its inner product by $(\cdot, \cdot)_X$.  We define
\begin{align*}
\mean{u}_\Omega := \begin{cases}
\frac{1}{\abs{\Omega}} \inn{u}{1}_{H^1(\Omega)} & \text{ if } u \in H^1(\Omega)', \\
\frac{1}{\abs{\Omega}} \int_\Omega u \dx & \text{ if } u \in L^1(\Omega)
\end{cases}
\end{align*}
as the spatial mean of $u$, where $\abs{\Omega}$ denotes the $d$-dimensional Lebesgue measure of $\Omega$.  The spatial mean for $v \in H^1(\Gamma)'$ and $v \in L^1(\Gamma)$ can be defined analogously. The definition of a tangential gradient on a Lipschitz surface can be found in \cite[Defn.~3.1]{Buffa}. 
For brevity, we also use the notation
\begin{align*}
\LL^p := L^p(\Omega) \times L^p(\Gamma)
\quad\text{and}\quad
\H^k := H^k(\Omega) \times H^k(\Gamma)
\quad \text{for all $p\in [1,\infty]$ and $k\ge 0$}.
\end{align*}

\bigskip

\paragraph{Assumptions.}
\begin{enumerate}[label=$(\mathrm{A \arabic*})$, ref = $\mathrm{A \arabic*}$]
	\item \label{ass:dom} We take $\Omega \subset \R^d$ with $d \in \{2,3\}$ to be a bounded domain whose Lipschitz boundary is denoted by $\Gamma$. Moreover, we fix an arbitrary final time $T>0$ and we write $Q_T:=\Omega\times(0,T)$ as well as $\Sigma_T:=\Gamma\times(0,T)$.
	\item \label{ass:const} We assume that the constants that are involved in the systems \eqref{CH:GMS}, \eqref{CH:LW} and \eqref{CH:INT} satisfy $\mo, \mg, \eps,\delta>0$, $L>0$, $\kappa\ge 0$ and $\beta\neq 0$ 
	. In Section 3 and Section 4, we set $\mo = \mg = \eps = \delta = 1$ in \eqref{CH} as their values have no impact on the mathematical analysis we will carry out. Since the regularity of weak solutions to the systems \eqref{CH:INT}, \eqref{CH:GMS} and \eqref{CH:LW} will depend on the parameter $\kappa\ge 0$, it is convenient to use the following notation: 
	\begin{align}
	\label{DEF:XY}
	\Xk := 
	\begin{cases}
	H^{1/2}(\Gamma) &\text{ if } \kappa=0,\\
	H^1(\Gamma) &\text{ if } \kappa>0,
	\end{cases}
	\qquad
	\Yk := 
	\begin{cases}
	H^1(\Gamma)' &\text{ if } \kappa=0,\\
	L^2(\Gamma) &\text{ if } \kappa>0.
	\end{cases}
	\end{align}
	\item  \label{ass:pot}
	We assume that the potentials $F$ and $G$ are non-negative and exhibit a decomposition $F=F_1+F_2$ and $G=G_1+G_2$ with  $F_1,F_2,G_1,G_2 \in C^1(\R)$ such that the following properties hold:
	\begin{enumerate}[label=$(\mathrm{A 3 \roman*})$, ref = $\mathrm{A 3 \roman*}$]
		\item \label{ass:pot:1} $F_1$ and $G_1$ are convex non-negative functions.
		\item \label{ass:pot:2} There exist exponents $p,q\ge 2$ as well as constants $a_{F},c_{F}>0$ and $b_{F}\ge 0$ such that for all $s\in\R$,
		\begin{alignat*}{3}
		\SwapAboveDisplaySkip
		a_{F}\abs{s}^p - b_{F} &\le \;F(s) &&\le c_{F}(1+\abs{s}^p), \\
		a_{G}\abs{s}^q - b_{G} &\le \;G(s) &&\le c_{G}(1+\abs{s}^q), \\
		a_{F'}\abs{s}^{p-1} - b_{F'} &\le \abs{F'(s)} && \le c_{F'}(1+\abs{s}^{p-1}), \\
		a_{G'}\abs{s}^{q-1} - b_{G'} &\le \abs{G'(s)} && \le c_{G'}(1+\abs{s}^{q-1}).
		\end{alignat*}
		This means that $F$ and $G$ have polynomial growth of order $p$ and $q$, respectively.
		\item \label{ass:pot:3} $F_2'$ and $G_2'$ are Lipschitz continuous. Consequently, there exist positive constants $d_F$, $d_G$, $d_{F'}$ and $d_{G'}$ such that for all $s\in\R$,
		\begin{alignat*}{3}
		&\abs{F_2'(s)} \le d_{F'}(1+\abs{s}), \quad &&\abs{G_2'(s)} \le d_{G'}(1+\abs{s}), \\ 
		&\abs{F_2(s)} \le d_F(1+\abs{s}^2), \quad &&\abs{G_2(s)} \le d_G(1+\abs{s}^2).
		\end{alignat*}
	\end{enumerate}
	\item \label{ass:pot:reg} For the higher regularity results we additionally assume that $\Omega$ is of class $C^3$, that $p\le 4$ in \eqref{ass:pot:2} and that there exist a positive constants $c_{F''},c_{G''}>0$ such that
	\begin{align*}
		0\le F_1''(s) \leq c_{F''}(1 + |s|^{p-2}), \quad 0\le G_1''(s) \leq 
		\begin{cases}
			c_{G''} (1 + |s|^{q-2}), &\text{if}\; \kappa>0,\\
			c_{G''} &\text{if}\; \kappa=0
		\end{cases} 
	\end{align*}
	for all $s\in\R$. 
\end{enumerate}

\begin{remark}	\label{rem:pot} 
	We point out that the smooth double well potential
	\begin{align*}
	W_\text{dw}(s)=\tfrac 1 4 (s^2-1)^2,\quad s\in\R,
	\end{align*} 
	is a suitable choice for $F$ and $G$ as it satisfies \eqref{ass:pot} with $p = 4$ and $q = 4$. However, singular potentials like the logarithmic potential or the obstacle potential are not admissible.
\end{remark}

\paragraph{Preliminaries.}
\begin{enumerate}[label=$(\mathrm{P \arabic*})$, ref = $\mathrm{P \arabic*}$]
\item \label{pre:V} For fixed $\kappa \geq 0$ we define the Hilbert space
\begin{align*}
\Vk & := \begin{cases}
\{ \phi \in H^1(\Omega) \, : \, \phi \vert_\Gamma \in H^1(\Gamma) \}, & \kappa > 0, \\
H^1(\Omega), & \kappa = 0,
\end{cases}
\end{align*}
endowed with the inner product and its induced norm
\begin{align*}
(\phi, \psi)_{\V^{\kappa}} & := \begin{cases}
(\phi, \psi)_{H^1(\Omega)} + (\phi \vert_\Gamma, \psi \vert_\Gamma)_{H^1(\Gamma)}, & \kappa > 0, \\
(\phi, \psi)_{H^1(\Omega)}, & \kappa = 0,
\end{cases}
\qquad \norm{\phi}_{\V^{\kappa}} := (\phi, \phi)_{\V^{\kappa}}^{1/2}.
\end{align*}
Moreover, we use the notation $\V := \V^1 = \{ \phi \in H^1(\Omega) \, : \, \phi \vert_\Gamma \in H^1(\Gamma) \}$ and we define
\begin{align*}
	\inn{\phi}{\zeta}_{\V,\beta} := \beta \inn{\phi}{\zeta}_{H^1(\Omega)} + \inn{\phi}{\zeta}_{H^1(\Gamma)}
\end{align*}
for all functions $\phi\in \V'$ and $\zeta\in\V$. If $\beta>0$, this product defines a duality pairing of $\V'$ and $\V$ which is equivalent to the standard one. In particular,
\begin{align*}
	\norm{\phi}_{\V',\beta} := \sup\big\{\, \big|\inn{\phi}{\zeta}_{\V,\beta}\big| \,:\, \zeta\in\V \text{ with } \norm{\zeta}_{\V} = 1 \big\}
	\quad \text{ for } \phi\in\V'
\end{align*}
defines a norm on the space $\V'$ which is equivalent to the standard norm.
\item \label{pre:H} For any $\beta\neq 0$ and $m\in\R$, we define 
\begin{align*}
	\Hm &:= \left\{ (\eta,\xi) \in \H^1 \,:\, \beta\abs{\Omega}\mean{\eta}_\Om + \abs{\Gamma}\mean{\xi}_\Ga = m\right\}.
\end{align*}
For any $L > 0$ and $\beta \neq 0$ we introduce an inner product on $\Ho$ by
\begin{align*}
\bigscp{(\phi,\psi)}{(\eta,\xi)}_{L,\beta} := \intO \nabla \phi \cdot \nabla \eta \dx + \intG \gradg \psi \cdot \gradg \xi + \frac{1}{L}(\beta \psi-\phi)(\beta\xi-\eta) \dG,
\end{align*}
for all $(\phi,\psi),(\eta,\xi) \in \Ho$. Its induced norm is given $\norm{\,\cdot\,}_{L,\beta}:= \scp{\cdot}{\cdot}_{L,\beta}^{1/2}$. 
\item \label{pre:W} 
For any $\beta\neq 0$, $m\in \R$ and any $\kappa \geq 0$, we define 
\begin{align*}
	\Wm &:= \left\{ \eta \in \Vk \,:\, \beta\abs{\Omega}\mean{\eta}_\Om + \abs{\Gamma}\mean{\eta}_\Ga = m\right\} \subset \Vk, \\
	\Wod & :=  \left\{ \phi \in (\Vk)' \,:\, \beta\abs{\Omega}\mean{\phi}_\Om + \abs{\Gamma}\mean{\phi}_\Ga = 0 \right\} \subset (\Vk)'.
\end{align*}
Let $\phi \in \Wod$ be arbitrary.
From \cite[Thm.~3.3]{knopf-liu} we infer the existence of a unique weak solution  $\SS(\phi) = (\SS_\Omega(\phi),\SS_\Gamma(\phi))\in \Ho$ to the elliptic problem
\begin{subequations}
	\label{EQ:LIN}
\begin{align}
- \Lx \SS_\Om &= - \phi && \text{ in } \Omega, \\
- \LB \SS_\Ga + \beta \pdnu \SS_\Om &= - \phi && \text{ on } \Gamma, \\
\pdnu \SS_\Om &= \tfrac 1 L(\beta \SS_\Ga - \SS_\Om) && \text{ on } \Gamma.
\end{align}
\end{subequations}
This means that $\SS(\phi)$ satisfies the weak formulation
\begin{align}
	\label{WF:LIN}
	\bigscp{\SS(\phi)}{(\zeta,\xi)}_{L,\beta} = - \inn{\phi}{\zeta}_{H^1(\Om)} - \inn{\phi}{\xi}_{H^1(\Ga)}
\end{align}
for all test functions $(\zeta,\xi)\in\H^1$. Thus, we can define the solution operator
\begin{align*}
\SS: \Wod \to \Ho,\quad \phi \mapsto \SS(\phi) = (\SS_\Omega(\phi),\SS_\Gamma(\phi)),
\end{align*}
and according to \cite[Cor.~3.5]{knopf-liu}, 
\begin{align*}
	\bigscp{\phi}{\psi}_{L,\beta,*} := \bigscp{\SS(\phi)}{\SS(\psi)}_{L,\beta}, \quad
	\norm{\cdot}_{L,\beta,*}:=\scp{\cdot}{\cdot}_{L,\beta,*}^{1/2} \quad
	\text{for all}\; \phi, \psi\in \Wod
\end{align*}
defines an inner product and its induced norm on the space $\Wod$. 
Since $\Wo\subset \Wod$, $\scp{\cdot}{\cdot}_{L,\beta,*}$ can also be used as an inner product on $\Wo$.
Moreover, $\norm{\cdot}_{L,\beta,*}$ is also a norm on $\Wo$ but $\Wo$ is not complete with respect to this norm.
\end{enumerate}

\smallskip

\begin{remark}\label{rem:gradflow}
	To motivate the implicit time discretisation used in the proof of the well-posedness result Theorem~\ref{THM:WP}, we point out that 
	the Cahn--Hilliard system \eqref{CH:INT} can be expressed as a gradient flow of the energy $E$ that was introduced in \eqref{DEF:EN}  with respect to the inner product $\scp{\cdot}{\cdot}_{L,\beta,*}$ on $\Wod$. The gradient flow equation reads as follows:
	\begin{align}
	\label{EQ:GF:DIR}
		\bigscp{\delt u}{\eta}_{L,\beta,*} = - \frac{ \delta E(u) }{\delta u}[\eta], \quad\text{for all}\; \eta \in \Wo \cap L^\infty(\Omega),\; \eta\vert_\Gamma\in L^\infty(\Gamma).
	\end{align}
	The requirement $ \delt u \in \Wod$  will be verified in Theorem~\ref{THM:WP}. 
	For a more detailed derivation of the gradient flow equation in similar situations see \cite[s.\,3]{GK} and \cite[s.\,3]{KL}.
\end{remark}

\smallskip
We will also need the following interpolation type inequality:

\begin{lemma} \label{LEM:INT}
	Suppose that \eqref{ass:dom} and \eqref{ass:const} hold. Then, for any $\alpha>0$, there exists a constant $C_\alpha>0$ depending only on $L$, $\alpha$, $\beta$ and $\Omega$ such that for all $u\in \Wo$,
	\begin{align*}
	\norm{u}_{L^2(\Om)}^2 + \norm{u}_{L^2(\Ga)}^2 \le \alpha \norm{\grad u}_{L^2(\Omega)}^2 + C_\alpha \norm{u}_{L,\beta,*}^2 .
	\end{align*}
\end{lemma}

\medskip

\noindent The proof can be found in the appendix.

\section{Well-posedness}
\label{section:wellposedness}
\subsection{Weak well-posedness of the  reaction rate dependent model}\label{sec:wellposed}

\begin{thm}[Weak well-posedness for the system \eqref{CH:INT}] \label{THM:WP}
	Suppose that \eqref{ass:dom} - \eqref{ass:pot} hold and let $m\in\R$ be arbitrary. 
	Then, for any initial datum $u_0 \in \Wm$ satisfying $F(u_0)\in L^1(\Omega)$ and $G(u_0)\in L^1(\Gamma)$, there exists a unique weak solution $(u,\mu,\theta)$ of the system \eqref{CH} in the following sense:
	\begin{enumerate}[label=$(\mathrm{\roman*})$, ref = $\mathrm{\roman*}$]
	\item The functions $(u,\mu,\theta)$ have the following regularity
	\begin{align}
	\label{REG:INT}
	\left\{\;
	\begin{aligned}
	u & \in C^{0,\frac{1}{2}}([0,T];H^1(\Omega)') \cap C^{0,\frac{1}{4}}([0,T];L^2(\Omega)) \\
	&\qquad \cap L^\infty\big(0,T;H^1(\Omega)\cap L^p(\Omega)\big) 
		\cap H^1(0,T;H^1(\Omega)'), \\
	u\vert_{\Sigma_T}  & \in C^{0,\frac 1 2}([0,T];H^1(\Gamma)') \cap C([0,T];L^2(\Ga))\\
	&\qquad \cap L^{\infty}(0,T;\Xk \cap L^q(\Gamma)) \cap H^1(0,T;H^1(\Gamma)'), \\
	u\vert_{\Sigma_T}  &\in C^{0,\frac 1 4}([0,T];L^2(\Ga)) \;\;\,\text{ if } \kappa>0, \\
	\mu &\in L^2(0,T;H^1(\Omega)),
	\qquad
	\theta \in L^2\big(0,T;H^1(\Gamma)\big).
	\end{aligned}
	\right.
	\end{align}
	and
	it holds that $u(t) \in \Wm$ for all $t\in[0,T]$. 
	\item
	The weak formulation
	\begin{subequations}
	\label{WF:INT}
	\begin{align}
	\label{WF:INT:1}
	\inn{\delt u}{ w}_{H^1(\Omega)} &= - \intO \grad\mu \cdot \grad w \dx 
	+  \intG \tfrac 1 L(\beta\theta-\mu)  w \dG, \\
	\label{WF:INT:2}
	\inn{\delt u}{ z}_{H^1(\Gamma)} &= - \intG \gradg\theta \cdot \gradg z \dG 
	- \intG \tfrac 1 L (\beta\theta-\mu) \beta z \dG, \\
	\label{WF:INT:3}
	\begin{split}
	\intO \mu \eta \dx + \intG \theta \eta \dG 
		& = \intO \grad u\cdot\grad\eta + F'(u)\eta \dx  \\
		&\quad + \intG \kappa \gradg u\cdot\gradg\eta + G'(u)\eta  \dG	 
	\end{split}
	\end{align}
	\end{subequations}
	is satisfied almost everywhere in $[0,T]$ for all test functions $w\in H^1(\Omega)$, $z\in H^1(\Gamma)$, $\eta\in \Vk \cap L^\infty(\Omega)$ with $\eta\vert_\Gamma \in L^\infty(\Gamma)$. Moreover, the initial condition $u(0)=u_0$ is satisfied a.e.~in $\overline{\Omega}$.
	\item For $E$ as defined in \eqref{DEF:EN}, the energy inequality 
	\begin{align}
	\label{WF:INT:DISS}
	\begin{aligned}
	&E\big(u(t)\big) 
	+ \frac 1 2 \int_0^t 
		\norm{\grad\mu(s)}_{L^2(\Omega)}^2 
		+ \norm{\gradg\theta(s)}_{L^2(\Gamma)}^2
		+ \frac 1 L \norm{\beta\theta(s) - \mu(s)}_{L^2(\Gamma)}^2 
	\ds \\
	&\quad \le E(u_0)
	\end{aligned}
	\end{align}
	is satisfied for all $t\in[0,T]$.
\end{enumerate}
If we additionally assume that \eqref{ass:pot:reg} holds, then 
\begin{align}
	\label{REG:INT:2}
	\left\{
	\begin{aligned}
	&u\in L^2(0,T;H^3(\Omega)), &&\quad\text{and}\quad u\vert_{\Sigma_T} \in L^2(0,T;H^3(\Gamma)) &&\quad \text{if}\; \kappa >0, \\
	&u\in L^2(0,T;H^{5/2}(\Omega)), &&\quad\text{and}\quad u\vert_{\Sigma_T}  \in L^2(0,T;H^2(\Gamma)) &&\quad \text{if}\; \kappa =0.
	\end{aligned}
	\right.
\end{align}
\end{thm}

To prove the assertions, we construct approximate solutions via an implicit time discretisation of the gradient flow equation \eqref{EQ:GF:DIR}. This technique which goes back to \cite{ambrosio} was first applied on a Cahn--Hilliard problem with dynamic boundary conditions in \cite{GK}, and later also in \cite{KL}. In the subsequent proof, we will employ the same strategy. Although some of the steps will be similar to those in \cite{GK} or \cite{KL}, a lot of arguments require a different reasoning. 

Due to the Robin type coupling condition \eqref{CH:INT:3}, the vector-spaces and the inner product involved in the gradient flow equation \eqref{EQ:GF:DIR} differ greatly from those in \cite{GK,KL}.
The crucial difference is that in the LW model, the chemical potentials are not coupled via the boundary condition $\deln\mu =0$, and as a consequence, the bulk and surface contribution in the inner product can be considered separately. 
This is not the case in the gradient flow equation \eqref{EQ:GF:DIR} as the space $\Wm$ and the inner product $\scp{\cdot}{\cdot}_{L,\beta,*}$ already comprises an interaction of bulk and surface quantities.
In particular, we thus require a new customized estimate for functions in $\Wo$ (see Lemma~\ref{LEM:INT}).

\begin{proof}
In this proof, we use the letter $C$ to denote generic positive constants independent of $N$, $n$ and $\tau$ that may change their value from line to line. The proof is split into several steps.

\subparagraph{Step 1: Implicit time discretisation.} Let $N \in \N$ be arbitrary and let $\tau := T/N$ denote the step size in time. For $n\in\{0,...,N\}$, we define functions $u^n$ recursively by the following construction.  The iterate with index zero is defined as the initial datum, i.e., $u^0=u_0$.  If $u^n$ is already constructed, we choose $u^{n+1}$ as a minimiser of the functional
\begin{align}
	\label{DEF:JN}
	J_n:\Wm \to \R,\quad u\mapsto \frac{1}{2 \tau} \norm{u - u^n}_{L,\beta,*}^2 + E(u)
\end{align}
where $\Wm$ is defined in \eqref{pre:W}. Note that $J_n$ may attain the value $+\infty$, see \cite[s.~4]{KL}.
The existence of such a minimiser will be addressed in Step 2. 
As $F_1$ and $G_1$ are convex, we can proceed as in \cite[Lem.\,3.2]{garckeelas} to infer that the Euler--Lagrange equation
\begin{align}\label{EQ:EUL}
\begin{aligned}
	0 \;&=\;  \scp{\frac{u^{n+1}-u^n}{\tau}}{\hat \eta}_{L,\beta,*} 
		+ \intO \grad u^{n+1} \cdot\grad \hat \eta + F'(u^{n+1})\hat \eta \dx \\
		&\quad + \intG \kappa \gradg u^{n+1}\cdot\gradg\hat \eta + G'(u^{n+1})\hat \eta \dG
\end{aligned}
\end{align}
holds for all directions $\hat \eta \in \Wo \cap L^\infty(\Omega)$ with $\hat \eta\vert_\Gamma \in L^\infty(\Gamma)$. This can be interpreted as a discretisation of the gradient flow equation \eqref{EQ:GF:DIR}. A straightforward computation shows that \eqref{EQ:EUL} is equivalent to	
\begin{align}
\label{EQ:EUL:D}
\begin{aligned}
	\intO \mathring\mu^{n+1} \hat \eta \dx + \intG \mathring\theta^{n+1} \hat\eta \dG &= \intO \grad u^{n+1} \cdot\grad\hat \eta + F'(u^{n+1})\hat \eta \dx \\
	&\qquad + \intG \kappa \gradg u^{n+1}\cdot\gradg\hat \eta + G'(u^{n+1})\hat \eta \dG
\end{aligned}
\end{align}
for all $\hat \eta \in \Wo \cap L^\infty(\Omega)$ with $\hat \eta\vert_\Gamma \in L^\infty(\Gamma)$, where 
\begin{align}
\label{DEF:MT0}
	(\mathring\mu^{n+1},\mathring\theta^{n+1}):=\SS\big(\tfrac 1 \tau (u^{n+1}-u^n)\big)\in \Ho.
\end{align}
For arbitrary $\eta \in \Vk \cap L^\infty(\Omega)$ with $\eta \vert_\Gamma \in L^\infty(\Gamma)$, we see that if $\beta |\Omega| + |\Gamma| \neq 0$, then 
\begin{align*}
\hat \eta = \eta + c_0, \quad c_0 = -\frac{\beta \intO \eta \dx + \intG \eta \dG}{\beta |\Omega| + |\Gamma|}
\end{align*}
satisfies $\hat \eta \in \Wo \cap L^\infty(\Omega)$ with $\hat \eta \vert_\Gamma \in L^\infty(\Gamma)$.  Then, we define a constant $c^{n+1} \in \R$ independent of the test function $\eta$ by
\begin{align*}
c^{n+1} = \frac{\intO F'(u^{n+1}) - \mathring{\mu}^{n+1} \dx + \intG G'(u^{n+1}) - \mathring{\theta}^{n+1} \dG}{\beta |\Omega| + |\Gamma|},
\end{align*}
so that the pair of functions
\begin{align}\label{mu:theta:n+1}
\mu^{n+1} := \mathring{\mu}^{n+1} + \beta c^{n+1} \quad \text{ and } \quad \theta^{n+1} := \mathring{\theta}^{n+1} + c^{n+1}
\end{align}
satisfies for arbitrary $\eta \in \Vk \cap L^\infty(\Omega)$ with $\eta \vert_\Gamma \in L^\infty(\Gamma)$,
\begin{equation}\label{dwf:int:3}
\begin{aligned}
\intO \mu^{n+1} \eta \dx + \intG \theta^{n+1} \eta \dG &= \intO \grad u^{n+1} \cdot\grad\eta + F'(u^{n+1})\eta \dx \\
	&\quad + \intG \kappa \gradg u^{n+1}\cdot\gradg\eta + G'(u^{n+1})\eta \dG.
\end{aligned}
\end{equation}
In the case $\beta |\Omega| + |\Gamma| = 0$, the above constant $c_0$ is not defined.  Hence, we consider fixing an arbitrary $\zeta \in C^\infty_c(\Omega)$ that is not identically zero, and define
\begin{align*}
\hat \eta = \eta + c_1 \zeta, \quad c_1 = - \frac{\beta \intO \eta \dx + \intG \eta \dG}{\beta \intO \zeta \dx},
\end{align*}
which satisfies $\hat \eta \in \Wo \cap L^\infty(\Omega)$ with $\hat \eta \vert_\Gamma \in L^\infty(\Gamma)$.  Then, we define the constant $c^{n+1} \in \R$ that is independent of $\eta$ as
\begin{align*}
c^{n+1} = \frac{\intO F'(u^{n+1}) \zeta - \mathring{\mu}^{n+1} \zeta + \nabla u^{n+1} \cdot \nabla \zeta \dx}{\beta \intO \zeta \dx}
\end{align*}
so that $\mu^{n+1}$ and $\theta^{n+1}$ as defined in \eqref{mu:theta:n+1} satisfy \eqref{dwf:int:3}.
By this construction, we find that the triplet $(u^{n+1},\mu^{n+1},\theta^{n+1})$ satisfies the equations
\begin{subequations}
	\label{DWF:INT}	
	\begin{align}
	\label{DWF:INT:1}
	&\intO \tfrac 1 \tau (u^{n+1}-u^n) w \dx = - \intO \grad\mu^{n+1} \cdot \grad w \dx 
	+  \intG \tfrac 1 L(\beta\theta^{n+1}-\mu^{n+1})  w \dG, \\
	\label{DWF:INT:2}
	&\intG \tfrac 1 \tau (u^{n+1}-u^n) z \dG = - \intG \gradg\theta^{n+1} \cdot \gradg z \dG 
	- \intG \tfrac 1 L (\beta\theta^{n+1}-\mu^{n+1}) \beta z \dG, \\[1ex]
	\label{DWF:INT:3}
	&\begin{aligned}
	\intO \mu^{n+1} \eta \dx + \intG \theta^{n+1} \eta \dG &= \intO \grad u^{n+1} \cdot\grad\eta + F'(u^{n+1})\eta \dx \\
	&\quad + \intG \kappa \gradg u^{n+1}\cdot\gradg\eta + G'(u^{n+1})\eta \dG,
	\end{aligned}
	\end{align}
\end{subequations}
for all test functions $w\in H^1(\Omega)$, $z\in H^1(\Gamma)$ and $\eta \in \Vk \cap L^\infty(\Omega)$ with $\eta\vert_\Gamma \in L^\infty(\Gamma)$.
This system can be interpreted as an implicit time discretisation of the weak formulation \eqref{WF:INT}. In this context, the collection $(u^{n}, \mu^{n}, \theta^{n})_{n =1, \dots, N}$ represents a time-discrete approximate solution.  For $t \in [0,T]$ and $n \in \{1,2, \dots, N\}$, we define the piecewise constant extension
\begin{align}\label{pi:cons}
(u_N, \mu_N, \theta_N)(\cdot, t) :=  (u_N^n, \mu_N^n, \theta_N^n) :=  (u^n, \mu^n, \theta^n),
\end{align}
for $t \in ((n-1)\tau, n\tau]$ and the piecewise linear extension
\begin{align}\label{pi:lin}
(\lu_N, \lmu_N, \lthN)(\cdot, t)  := \alpha(u_N^n, \mu_N^n, \theta_N^n) + (1-\alpha)(u_N^{n-1}, \mu_N^{n-1}, \theta_N^{n-1})
\end{align}
for any $\alpha \in [0,1]$ and $t = \alpha n \tau + (1-\alpha)(n-1)\tau$.

\subparagraph{Step 2: Existence of a minimiser.}
We apply the direct method in the calculus of variations to show that the functional $J_n$ has at least one minimiser in the set $\Wm$. To this end, we assume that $u^n$ is already constructed as described in Step 1. Recalling the definition of the energy functional \eqref{DEF:EN} and that the potentials $F$ and $G$ are bounded from below according to \eqref{ass:pot:2}, we infer that
\begin{align}
	\label{BND:JN}
	J_n(u) \ge - C^* \quad\text{where}\quad C^* := b_F\abs{\Omega} + b_G\abs{\Gamma},
\end{align}
for all $ u\in \Wm $. Consequently, $M:=\inf_{\Wm} J_n$ exists  and is finite, and we can find a minimising sequence $(u_k)_{k\in\N}\subset \Wm$ such that
\begin{align*}
	\lim_{k \to \infty} J_n(u_k) = M, 
	\quad \text{and}\quad	
	J_n(u_k) \le M+1 \;\; \text{ for all } k\in\N.
\end{align*}
From the definition of the functional $J_n$ (see \eqref{DEF:JN}) we deduce that 
\begin{align}
	\label{BND:MC}
	\frac 1 2 \norm{\grad u_k}_{L^2(\Omega)}^2 
	+ \frac \kappa 2 \norm{\gradg u_k}_{L^2(\Gamma)}^2
	+ \intO F( u_k ) \dx + \intG G( u_k ) \dG
	\le C,
\end{align}
for all $k\in\N$. Now the growth estimates \eqref{ass:pot:2} for $F$ and $G$ imply that
the sequence $(u_k)$ is bounded in $\Vk$. Hence, according to the Banach--Alaoglu theorem, there exists a function $\bar u \in \Vk$ such that $u_k\wto \bar u$ in $\Vk$ along a non-relabelled subsequence. Recalling the compact embeddings $H^1(\Omega) \emb L^2(\Omega)$ and $H^1(\Omega)\emb L^2(\Gamma)$, consequently $u_k \to \bar u$ in $L^2(\Omega)$ and $u_k \to \bar u$ in $L^2(\Gamma)$ along a non-relabelled subsequence, so that  $\bar u\in \Wm$.  It remains to show that $\bar u$ is actually a minimser of the functional $J_n$.
Since $F$ and $G$ are continuous and non-negative, we can use Fatou's lemma to infer that
\begin{align}
\label{LIMSUP:FG}
\intO F(\bar u) \dx \le \underset{k \to \infty}{\lim\inf}\intO F(u_k) \dx,
\quad\text{and}\quad
\intG G(\bar u) \dG  \le \underset{k\to \infty}{\lim\inf} \intG G(u_k) \dG.
\end{align}
As all other components of the energy are continuous and convex, we conclude that
\begin{align*}
J_n(\bar u) \leq \underset{k \to \infty}{\lim\inf} J_n(u_k) = M.
\end{align*}
This proves that $\bar u$ is a minimiser of $J_n$ on the set $\Wm$.

\subparagraph{Step 3: Uniform estimates.}
Next, we establish uniform estimates for the piecewise constant extension. We claim that
\begin{equation}
\label{EST:PCE}
\begin{aligned}
	\norm{u_N}_{L^\infty(0,T;H^1(\Omega) \cap L^p(\Omega))} + \norm{u_N}_{L^\infty(0,T;\Xk \cap L^q(\Gamma))} \quad& \\
	+ \norm{\mu_N}_{L^2(0,T;H^1(\Omega))}  + \norm{\theta_N}_{L^2(0,T;H^1(\Gamma))} &\leq C.
\end{aligned}
\end{equation}

To prove this assertion, we follow the reasoning in \cite[s.~4]{GK} and \cite[s.~5]{KL}. As $u^{n+1}$ is a minimiser of the functional $J_n$ on the set $\Wm$, we obtain the a priori estimate
\begin{align}
\label{EST:APR}
\frac{1}{2\tau} \norm{u^{n+1}-u^n}_{L,\beta,*}^2 + E(u^{n+1}) \le E(u^n) \;\; \text{ for all } n\in\{0,1, \dots ,N-1\}.
\end{align}
By induction we conclude that $E(u^{n}) \le E(u_0)$ for all $n \in \{0, 1, \dots, N\}$. Thanks to the definition of $E$ (see \eqref{DEF:EN}), we infer that
\begin{align}\label{EST:APR2}
&\frac{1}{2} \norm{\grad u^{n+1}}_{L^2(\Omega)}^2 
	+ \frac \kappa 2 \norm{\gradg u^{n+1}}_{L^2(\Gamma)}^2
	+ \intO F(u^{n+1}) \dx
	+ \intG G(u^{n+1}) \dG
	\leq E(u_0) + C^*.
\end{align}
From the growth assumptions \eqref{ass:pot:2} we deduce the uniform bound
\begin{align}
\label{EST:UVN}
\norm{u_N}_{H^1(\Omega)} + \norm{u_N}_{L^p(\Omega)} + \norm{u_N}_{\Xk} + \norm{u_N}_{L^q(\Ga)}\le C.
\end{align}
To derive a uniform bound on $\mu_N$ we can argue as in \cite[s.~4]{GK} and \cite[s.~5]{KL}. Proceeding this way, we use a generalised Poincar\'e inequality (see \cite[p.~242]{Alt}) to obtain the estimate
\begin{align}
\label{EST:MUN}
\norm{\mu^{n+1}}_{L^2(\Omega)} \le C\, \left(1+\norm{\grad \mu^{n+1}}_{L^2(\Omega)}\right)
\end{align}
for all $n \in \{0, 1,...,N-1\}$. To bound $\norm{\grad \mu^{n+1}}_{L^2(\Omega)}$ we first show that an energy dissipation law holds true on the discrete level.
We recall that, according to \eqref{pi:cons},
\begin{align*}
 \big(u_N,\mu_N,\theta_N)(s) = \big(u_N,\mu_N,\theta_N)(t) = \big(u_N^n,\mu_N^n,\theta_N^n)
\end{align*}
for all $s\in (t-\tau,t]$, $n\in\{1,...,N-1\}$ where $t=\tau n$ is fixed. Using the definitions of $\mu_N$ and $\theta_N$ and recalling \eqref{DEF:MT0} as well as the a priori estimate \eqref{EST:APR}, a straightforward computation shows that
\begin{align*}
&E\big(u_N(t)\big) + \frac 1 2 \int_{t-\tau}^t \norm{\grad\mu_N(s)}_{L^2(\Omega)}^2  + \norm{\gradg\theta_N(s)}_{L^2(\Gamma)}^2 + \frac 1 L \norm{\beta\theta_N(s)-\mu_N(s)}_{L^2(\Gamma)}^2 \ds \\
&\quad = E\big(u_N(t)\big) + \frac{1}{2\tau^2} \int_{t-\tau}^t \norm{u_N(s)-u_N(s-\tau)}_{L,\beta,*}^2 \ds\\
&\quad \le E\big(u_N(t)\big) + \frac{1}{2\tau} \norm{u_N(t)-u_N(t-\tau)}_{L,\beta,*}^2 
\;\le  E\big(u_N(t-\tau)\big).
\end{align*}
Performing a simple induction we conclude that
\begin{align}
\label{IEQ:EUN}
\begin{aligned}
&E\big(u_N(t)\big) + \frac 1 2 \int_{0}^t \norm{\grad\mu_N(s)}_{L^2(\Omega)}^2 
+ \norm{\gradg\theta_N(s)}_{L^2(\Gamma)}^2 
+ \frac 1 L \norm{\beta\theta_N(s)-\mu_N(s)}_{L^2(\Gamma)}^2\ds \\
&\quad \le E(u_0).
\end{aligned}
\end{align}
In particular, for $t=N\tau \equiv T$, we get
\begin{align}
\label{EST:GMUN}
\norm{\grad\mu_N}_{L^2(0,T;L^2(\Omega))}^2 +  \norm{\gradg\theta_N}_{L^2(0,T;L^2(\Gamma))}^2 
 \le 2E(u_0) 
\le C.
\end{align}
In combination with \eqref{EST:MUN} we infer that $\mu_N$ is uniformly bounded in $L^2(0,T;H^1(\Omega))$.
It remains to establish the uniform bound on $\theta_N$. From the growth estimates \eqref{ass:pot:2} (particularly, 
the upper bounds for $F'$ and $G'$) and \eqref{EST:UVN} we obtain that
\begin{alignat*}{2}
\intO \abs{F'(u^{n+1})} \dx &\le C + C \norm{u^{n+1}}_{L^p(\Omega)}^{p-1} &&\le C\\
\intG \abs{G'(u^{n+1})} \dG &\le C + C \norm{u^{n+1}}_{L^q(\Gamma)}^{q-1} &&\le C.
\end{alignat*}
for all $n\in \{0,...,N\}$.
Now, testing \eqref{DWF:INT:3} with $\eta\equiv 1$ and using the above estimates yields
\begin{align*}
\abs{\intG \theta^{n+1} \dG} 
	\le C + C\norm{\mu^{n+1}}_{L^2(\Omega)}.
\end{align*}
Using Poincar\'e's inequality on $\Gamma$ and the estimate \eqref{EST:MUN}, it follows that
\begin{align}
\label{EST:THN}
	\norm{\theta^{n+1}}_{L^2(\Ga)} 
	\le C\, \left(1+\norm{\grad \mu^{n+1}}_{L^2(\Omega)} +\norm{\gradg \theta^{n+1}}_{L^2(\Ga)}\right) 
\end{align}
for all $n \in \{0, 1,...,N-1\}$. Hence, by \eqref{EST:GMUN} we conclude that $\theta_N$ is uniformly bounded in $L^2(0,T;H^1(\Gamma))$.

\subparagraph{Step 4: H\"older-in-time estimates.}
We now use interpolation type arguments to show that the piecewise linear extension is H\"older continuous in time. In particular, we claim that for all $t,s\in[0,T]$,
\begin{subequations}
\begin{alignat}{4}
\SwapAboveDisplaySkip
\label{EST:HLD:U}
\norm{\lu_N(t)-\lu_N(s)}_{L^2(\Omega)} 
	&\le C \abs{t-s}^{\frac 1 4},\\
\label{EST:DIFF:U}
\norm{u_N(t)-\lu_N(t)}_{L^2(\Omega)} 
	&\le C \tau^{\frac 1 4}, \\
\label{EST:HLD:UV}
\norm{\lu_N(t)-\lu_N(s)}_{H^1(\Omega)'} + \norm{\lu_N(t)-\lu_N(s)}_{H^1(\Gamma)'} 
	&\le C \abs{t-s}^{\frac 1 2} \\
\label{EST:DIFF:UV}
\norm{u_N(t)-\lu_N(t)}_{H^1(\Omega)'} + \norm{u_N(t)-\lu_N(t)}_{H^1(\Gamma)'} 
	&\le C \tau^{\frac 1 2},  \\
\label{EST:HLD:V}
\norm{\lu_N(t)-\lu_N(s)}_{L^2(\Gamma)} 
	&\le C \abs{t-s}^{\frac 1 4} &&\quad\text{if}\; \kappa > 0,\\
\label{EST:DIFF:V}
\norm{u_N(t)-\lu_N(t)}_{L^2(\Gamma)} 
	&\le C \tau^{\frac 1 4} &&\quad\text{if}\; \kappa > 0
\end{alignat}
\end{subequations}
as well as
\begin{align}
\label{EST:HLD}
\norm{\delt \bar u_N}_{L^2(0,T;H^1(\Omega)')} + \norm{\delt \bar u_N}_{L^2(0,T;H^1(\Gamma)')} \le C.
\end{align}

To prove this claim, we first infer from \eqref{DWF:INT:1} and \eqref{DWF:INT:2} that for any
$w\in H^1(\Om)$, $z\in H^1(\Ga)$ and almost all $r  \in[0,T]$,
\begin{subequations}
	\label{EQ:DTU}
\begin{align}
\label{EQ:DTU:1}	\inn{\delt\lu_N( r )}{w}_{H^1(\Om)} 
		&= - \intO \grad \mu_N(r) \cdot \grad w \dx 
		+ \intG \tfrac 1 L \big(\beta \theta_N(r ) - \mu_N(r )\big) w \dG, \\
\label{EQ:DTU:2}	\inn{\delt\lu_N(r )}{z}_{H^1(\Ga)} 
		&= - \intG \gradg \theta_N(r ) \cdot \gradg z \dG 
		- \intG \tfrac 1 L \big(\beta \theta_N(r ) - \mu_N(r )\big) \beta z \dG.
\end{align}
\end{subequations}
Let $s, t \in [0,T]$ be arbitrary and without loss of generality suppose $s < t$. Integrating \eqref{EQ:DTU:1} from $s$ to $t$ and choosing $w = \lu_N(t) - \lu_N(s)$ yields
\begin{align*}
& \| \lu_N( t ) - \lu_N(s) \|_{L^2(\Omega)}^2 \\
& \quad \leq 2 \| \lu_N \|_{L^\infty(0,T;H^1(\Omega))} \big ( \norm{\mu_N}_{L^2(0,T;H^1(\Omega))} + \tfrac{1}{L} \norm{\beta \theta_N - \mu_N}_{L^2(0,T;L^2(\Gamma))} \big ) |t-s|^{\frac 1 2},
\end{align*}
which is \eqref{EST:HLD:U}.  Similarly, if $\kappa > 0$, then integrating \eqref{EQ:DTU:2} from $s$ to $t$, choosing $z = \lu_N(t) - \lu_N(s)$ leads to \eqref{EST:HLD:V}. Moreover, it is clear that \eqref{EST:HLD} follows directly from \eqref{EQ:DTU} and previous uniform estimates on $\nabla \mu_N$, $\gradg \theta_N$ and $\beta \theta_N - \mu_N$.

For almost all $r\in[0,T]$, applying the Cauchy--Schwarz inequality and the continuous embedding $H^1(\Om)\emb L^2(\Ga)$ to \eqref{EQ:DTU:1} yields
\begin{align}
	\label{EST:DTU}
	\begin{aligned}
	&\Big|\inn{\delt\lu_N(r )}{w}_{H^1(\Om)}\Big| \\
	&\quad \le C \left(1 + \tfrac{1}{\sqrt{L}}\right) 
	\Big( \norm{\grad \mu_N(r )}_{L^2(\Om)} 
		+ \tfrac{1}{\sqrt{L}}\norm{\beta\theta_N(r ) 
		- \mu_N(r)}_{L^2(\Ga)} \Big) \norm{w}_{H^1(\Om)}.
	\end{aligned}
\end{align}
For arbitrary $s < t$, using \eqref{IEQ:EUN}, \eqref{EST:DTU} and H\"older's inequality, we conclude that 
\begin{align*}
	& \norm{\lu_N(t)-\lu_N(s)}_{H^1(\Om)'} 
	= \underset{\norm{w}_{H^1(\Om)}=1}{\sup} \Big| \inn{\lu_N(t)-\lu_N(s)}{w}_{H^1(\Om)} \Big| \\
	&\quad \le \underset{\norm{w}_{H^1(\Om)}=1}{\sup}\; \int_s^t \Big| \inn{\delt\lu_N(r)}{w}_{H^1(\Om)} \Big| \,dr \\[1ex]
	&\quad \le C \left(1 + \tfrac{1}{\sqrt{L}}\right)
		\int_s^t 
		\norm{\grad \mu_N(r)}_{L^2(\Om)} 
			+ \tfrac{1}{\sqrt{L}}\norm{\beta\theta_N(r) - \mu_N(r)}_{L^2(\Ga)}  \,dr \\
	&\quad \le C \left(1 + \tfrac{1}{\sqrt{L}}\right) \abs{t-s}^{\frac 1 2} 
		\left( \int_0^t 
		\norm{\grad \mu_N(r)}_{L^2(\Om)}^2 
		+ \tfrac{1}{L}\norm{\beta\theta_N(r) - \mu_N(r)}_{L^2(\Ga)}^2  \,dr
		\right)^{\frac 12} \\
	&\quad \le C \left(1 + \tfrac{1}{\sqrt{L}}\right) \abs{t-s}^{\frac 1 2} .
\end{align*}
In a similar fashion, we can derive the estimate
\begin{align*}
	\norm{\lu_N(t)-\lu_N(s)}_{H^1(\Ga)'} 
	\le C \left(1 + \tfrac{1}{\sqrt{L}}\right)\abs{t-s}^{\frac 1 2} 
\end{align*}
which proves \eqref{EST:HLD:UV} as $C$ may depend on $L$.  
Next, for any $t\in[0,T]$ we can choose $n\in\{1,...,N\}$ and $\alpha\in [0,1]$ such that $t=\alpha n\tau + (1-\alpha)(n-1)\tau$.
Hence, it follows immediately that
\begin{align*}
\|\bar u_N(t)-u_N(t)\|_{X} 
&\le \| \alpha u^n + (1-\alpha) u_N^{n-1}(t) - u_N^n(t) \|_{X} \\
&= (1-\alpha)\,\| u_N^n(t)  - u_N^{n-1}(t) \|_{X} = (1-\alpha) \, \| \bar u_N(n\tau) - \bar u_N((n-1)\tau)\|_{X} 
\end{align*}
for $X=L^2(\Om)$ or $L^2(\Ga)$ or $H^1(\Om)'$ or $H^1(\Ga)'$.  Choosing $t = n\tau$ and $s = (n-1)\tau$ in \eqref{EST:HLD:U}, \eqref{EST:HLD:UV}, \eqref{EST:HLD:V} leads to \eqref{EST:DIFF:U}, \eqref{EST:DIFF:UV} and \eqref{EST:DIFF:V}, respectively.


\subparagraph{Step 5: Convergence assertions and regularity of the limit.} 
We now claim that there exist functions $(u,\mu,\theta)$ satisfying the regularity condition \eqref{REG:INT}
such that the following convergence properties hold along a non-relabelled subsequence:
\begin{align*}
\begin{aligned}
& u_N \to u 
	&&\text{ weakly-* in } L^{\infty}(0,T;H^1(\Omega)\cap L^p(\Omega)), \\
	&&&\quad \text{ strongly in } L^{\infty}(0,T;L^2(\Omega)), 
	\text{ and a.e.~in } Q_T, \\[1ex]
& u_N\vert_{\Sigma_T} \to u\vert_{\Sigma_T} 
	&&\text{ weakly-* in } L^{\infty}(0,T;\Xk \cap L^q(\Ga)), \\
	&&&\quad \text{ strongly in } L^\infty(0,T;\Yk), 
	\text{ and a.e.~in } \Sigma_T,  \\[1ex]
& \lu_N \to u 
	&&\text{ weakly in } H^1(0,T;H^1(\Omega)'), 
	\\ 
	&&&\quad  \text{ strongly in } C^{0,\gamma}([0,T];H^1(\Omega)') \text{ for all } \gamma\in (0,\tfrac 1 2 ),  \\
	&&&\quad  \text{ and strongly in } C^{0,\gamma}([0,T];L^2(\Omega)) \text{ for all } \gamma\in (0,\tfrac 1 4 ),  \\[1ex]
& \lu_N\vert_{\Sigma_T}  \to u\vert_{\Sigma_T}  
	&&\text{ weakly in } H^1(0,T;H^1(\Gamma)'), 
	\\ 
	&&&\quad  \text{ strongly in } C^{0,\gamma}([0,T];H^1(\Ga)') \text{ for all } \gamma\in (0,\tfrac 1 2 ),  \\
	&&&\quad  \text{ and strongly in } C^{0,\gamma}([0,T];L^2(\Ga)) \text{ for all } \gamma\in (0,\tfrac 1 4 ) 
		\text{ if } \kappa>0,  \\[1ex]
& \mu_N \to \mu 
	&&\text{ weakly in } L^2(0,T;H^1(\Omega)), 
	\text{ and weakly in } L^2(0,T;H^{1/2}(\Gamma)),\\
& \theta_N \to \theta 
	&&\text{ weakly in } L^2(0,T;H^1(\Gamma)). 
\end{aligned}
\end{align*}

These convergence assertions can be established using the same methods as in \cite[s.~4.5]{GK} and \cite[s.~5]{KL}.
Moreover, recalling the compact embedding $H^1(\Omega) \emb H^{3/4}(\Omega)$ and the continuous embedding $H^{3/4}(\Omega) \emb H^1(\Omega)'$, we infer from the Aubin--Lions lemma \cite{simon} that $u\in C([0,T];H^{3/4}(\Omega))$. By the continuous embedding $H^{3/4}(\Omega) \emb L^2(\Gamma)$, this additionally yields 
\begin{align*}
	u\vert_{\Sigma_T} \in C([0,T];L^2(\Ga)).
\end{align*}
Hence, the regularity assertion \eqref{REG:INT} is established.

\subparagraph{Step 6: Existence of weak solutions.}
We finally show that the limit $(u, \mu, \theta)$ is a weak solution of the system \eqref{CH}. We already know from Step 5 that the limit $(u, \mu, \theta)$ enjoys the regularity demanded in \eqref{REG:INT}. 
Using the convergence results from Step 5 we may pass to the limit in \eqref{EQ:DTU} after multiplying by an arbitrary $\zeta(t) \in C^\infty_c(0,T)$ and integrating over $(0,T)$. By a standard density argument, this directly implies that \eqref{WF:INT:1} and \eqref{WF:INT:2} are satisfied.
Moreover, we deduce that $F'(u_N)\to F'(u)$ a.e.~in $Q_T$ and $G'(u_N) \to G'(u)$ a.e.~on $\Sigma_T$.  Recalling the growth estimates on $F'$ and $G'$ and the uniform bounds on $u_N$, we can apply Lebesgue's general convergence theorem (see \cite[p.~60]{Alt}) to obtain 
\begin{align*}
	\int_{Q_T} F'(u_N) \zeta(t) \eta \dx \dt  \to \int_{Q_T} F'(u) \zeta(t)\eta \dx \dt, \\
	\int_{\Sigma_T} G'(u_N) \zeta(t) \eta \dG \dt \to \int_{\Sigma_T} G'(u) \zeta(t) \eta \dG \dt.
\end{align*}
This allows us obtain \eqref{WF:INT:3} from passing to the limit  in \eqref{DWF:INT:3}. Hence, the triplet $(u, \mu, \theta)$ satisfies the weak formulation \eqref{WF:INT}. Proceeding as in Step 2 we get
\begin{align*}
	\intO F(u(t)) \dx &\le \underset{N\to \infty}{\lim\inf}\;\intO F(u_N(t)) \dx,\\
	\intG G(u(t)) \dG &\le \underset{N\to \infty}{\lim\inf}\;\intO G(u_N(t)) \dG,
\end{align*}
for almost all $t\in [0,T]$. As all other contributions of the energy functional $E$ are continuous and convex, we can use the convergence properties from Step 5 to verify the energy inequality \eqref{WF:INT:DISS} from \eqref{IEQ:EUN}.

\subparagraph{Step 7: Uniqueness.}
Suppose that $(u_1, \mu_1, \theta_1)$ and $(u_2, \mu_2, \theta_2)$ are two weak solutions to the system \eqref{CH:INT} corresponding to the same initial data. We denote the difference of these solutions by
\begin{align*}
(\lu,\lmu,\ltheta) := (u_1,\mu_1,\theta_1) -(u_2,\mu_2,\theta_2).
\end{align*}
We point out that $\lu(t)\in \Wo$ for all $t\in[0,T]$. Let now $t_0 \in (0,T]$, $ w \in L^2(0,T;H^1(\Omega))$ and $z \in L^2(0,T;H^1(\Gamma))$ be arbitrary. In the following we use the notation $Q_{t_0} = \Omega \times (0, t_0)$, $\Sigma_{t_0} = \Gamma \times (0,t_0)$. We set 
\begin{align}\label{uniq:test}
\tilde w (\cdot,t) :=	\begin{cases}
\int_t^{t_0} w (\cdot,s) \ds, &\text{ if } t\le t_0,\\
0 &\text{ if } t>t_0
\end{cases}
\quad\text{and}\quad
\tilde z (\cdot,t) :=	\begin{cases}
\int_t^{t_0} z (\cdot,s) \ds, &\text{ if } t\le t_0,\\
0 &\text{ if } t>t_0,
\end{cases}
\end{align}
and thus, $\tilde w \in L^2(0,T;H^1(\Omega)) \cap H^1(0,T;L^2(\Omega))$ and $\tilde z \in L^2(0,T;H^1(\Gamma))\cap H^1(0,T;L^2(\Gamma))$. Plugging $\tilde w$ into \eqref{WF:INT:1} and $\tilde z$ into \eqref{WF:INT:2}, we find that
\begin{align*}
&\int_{Q_{t_0}} \lu w \dx \dt + \int_{\Sigma_{t_0}} \lu z \dG \dt \\
&\quad = -\int_{Q_{t_0}} \nabla \left ( \int_{0}^t \lmu \ds \right ) \cdot \nabla w \dx \dt
	 \;- \int_{\Sigma_{t_0}} \gradg \left ( \int_{0}^t \ltheta \ds \right ) \cdot \gradg z \dG \dt \\
&\qquad - \frac 1 L \int_{\Sigma_{t_0}} \left( \beta \int_{0}^t  \ltheta \ds - \int_{0}^t \lmu \ds \right) (\beta z-w) \dG\dt.
\end{align*}
In view of the solution operator $\SS$ from \eqref{pre:W} we obtain the identifications
\begin{align*}
	\SS_\Om(\bar u) = \int_{0}^t \lmu \ds + \beta c, \quad
	\SS_\Ga(\bar u) = \int_{0}^t \ltheta \ds + c 
\end{align*}
for some constant $c \in \R$, and thus,
\begin{align*}
	\delt \SS_\Om(\bar u) = \lmu, \quad
	\delt \SS_\Ga(\bar u) = \ltheta. 
\end{align*} 
We now choose $w = \lmu$ and $z = \ltheta$. Using $\bar u(0)=0$ and $\SS(0)=0$, we find that
\begin{align}
\label{EQ:UNQ}
\int_{Q_{t_0}} \lu \lmu \dx \dt + \int_{\Sigma_{t_0}} \lu \ltheta \dG \dt = - \frac 1 2 \norm{\lu(t_0)}_{L,\beta,*}^2.
\end{align}
For $M > 0$, we define the projection $\PM : \R \to [-M,M]$ as
\begin{align}
\label{DEF:PM}
\PM(s) = \begin{cases}
s & \text{ if } \abs{s} < M, \\
\frac{s}{\abs{s}} M & \text{ if } \abs{s} \geq M.
\end{cases}
\end{align}
Now, for any $M>0$, the test function $\eta = \chi_{[0,t_0]} \PM(\lu)$ belongs to $L^2(0,T;\Vk) \cap L^{\infty}(Q_T)$ and satisfies $\eta\vert_{\Sigma_T} \in \cap L^{\infty}(\Sigma_T)$. Hence, it can be used as a test function in \eqref{WF:INT:3}. Recalling the monotonicity of $F_1'$ and $G_1'$, we infer that
\begin{align}
\label{EST:UNQ}
\begin{aligned}
&\int_{Q_{t_0}} \lmu \PM(\lu) \dx \dt + \int_{\Sigma_{t_0}} \ltheta \PM(\lu) \dG \dt \\
&\quad \ge \int_{Q_{t_0}} \nabla \lu \cdot \nabla \PM(\lu) + \big(F_2'(u_1) - F_2'(u_2)\big) \PM(\lu) \dx \dt \\
&\qquad + \int_{\Sigma_{t_0}} \kappa \gradg \lu \cdot \gradg \PM(\lu) + \big(G_2'(u_1) - G_2'(u_2)\big) \PM(\lu) \dG \dt 
\end{aligned}
\end{align}
Applying the dominated convergence theorem, we can pass to the limit $M \to \infty$, leading to \eqref{EST:UNQ} with $\PM(\lu)$ replaced by $\lu$. Now, in combination with \eqref{EQ:UNQ}, we get
\begin{align}
	\label{EST:UNQ:2}
	\frac 1 2 \norm{\lu(t_0)}_{L,\beta,*}^2 + \norm{\grad\lu}_{L^2(Q_{t_0})}^2 + \kappa \norm{\gradg\lu}_{L^2(\Sigma_{t_0})}^2 \le C_\text{Lip} \big( \norm{\lu}_{L^2(Q_{t_0})}^2 + \norm{\lu}_{L^2(\Sigma_{t_0})}^2 \big)
\end{align}
where the constant $C_\text{Lip}$ depends only on the Lipschitz constants of $F_2'$ and $G_2'$. Invoking Lemma~\ref{LEM:INT} with $\alpha := (2C_\text{Lip})^{-1}$ we deduce from \eqref{EST:UNQ:2} that
\begin{align*}
	\frac 1 2 \norm{\lu(t_0)}_{L,\beta,*}^2 + \frac 1 2 \norm{\grad\lu}_{L^2(Q_{t_0})}^2  
		\le C_\text{Lip} C_\alpha \int_{0}^{t_0} \norm{\lu(t)}_{L,\beta,*}^2 \dt.
\end{align*}
Since $t_0$ was arbitrary this estimate holds for all $t_0\in[0,T]$. Hence, we can apply Gronwall's lemma to infer that 
\begin{align*}
	\norm{\SS(\bar u)}_{L,\beta} = \norm{\bar u}_{L,\beta,*} = 0.
\end{align*}
Recalling that $\SS(\bar u)$ is the weak solution of the system \eqref{EQ:LIN} to the right-hand side $\bar u$, we finally conclude that $\bar u=0$ a.e.~in $\Omega_T$ and $\bar u\vert_{\Sigma_T} =0$ a.e.~on $\Sigma_T$. In view of \eqref{WF:INT:3} for the difference of solutions, we obtain
\begin{align}\label{uniq}
\intO \bar \mu \eta \dx + \intG \bar \theta \eta \dG = 0
\end{align}
for arbitrary $\eta \in \Vk \cap L^\infty(\Omega)$ such that $\eta \vert_\Gamma \in L^\infty(\Gamma)$.  We first consider $\eta \in C^{\infty}_c(\Omega)$ and applying the fundamental lemma of calculus of variations to deduce that $\bar \mu = 0$ a.e.~in $\Omega$.  Then, the first term of \eqref{uniq} vanishes and consequently we infer that $\bar \theta = 0$ a.e.~on $\Gamma$.  Hence, we obtain the uniqueness of weak solutions.

\subparagraph{Step 8: Higher regularity.}
By arguing as in \cite[s.~4]{KL}, one can establish under assumption \eqref{ass:pot:reg} the additional regularity $u\in L^2(0,T;H^2(\Omega))$ for any $\kappa \geq 0$ and also $u\in L^2(0,T;H^2(\Gamma))$ if $\kappa>0$.  Let us sketch the arguments for the regularity assertions in \eqref{REG:INT:2}.  For $\kappa > 0$, since $\mu, F'(u) \in L^2(0,T;H^1(\Om))$ and $u \vert_{\Sigma_T} \in L^2(0,T;H^2(\Ga))$, elliptic regularity theory gives $u \in L^2(0,T;H^{5/2}(\Om))$.  Together with $\Delta u \in L^2(Q_T)$, a variant of the trace theorem implies $\pdnu u \in L^2(0,T;H^1(\Ga))$.  Then, as $\theta, G'(u) \in L^2(0,T;H^1(\Ga))$, by elliptic regularity we have $u \vert_{\Sigma_T} \in L^2(0,T;H^3(\Ga))$.  Employing this more regular boundary trace for $u$ with elliptic regularity yields $u \in L^2(0,T;H^3(\Om))$.

On the other hand, for $\kappa = 0$, we only have $u \in L^2(0,T;H^2(\Om))$ from \cite{KL}.  However, from $\eqref{CH:INT:2}_2$, since $\theta, G'(u) \in L^2(0,T;H^1(\Ga))$, we infer that $\pdnu u \in L^2(0,T;H^1(\Ga))$.  Then, by elliptic regularity theory we obtain $u \in L^2(0,T;H^{5/2}(\Om))$ and by the trace theorem $u \vert_{\Sigma_T} \in L^2(0,T;H^{2}(\Ga))$. 

\smallskip

Now, as all assertions are established, the proof of Theorem~\ref{THM:WP} is complete.
\end{proof}

\subsection{Improved regularity and strong solutions}

\begin{thm}
	\label{THM:SWP}
	Let $m\in\R$ be arbitrary. Suppose that \eqref{ass:dom} - \eqref{ass:pot:reg} hold and that 
	$u_0 \in \Wm$ with $(u_0,u_0\vert_\Gamma)\in\H^3$ if $\kappa > 0$ or with $u_0 \in H^3(\Omega)$ if $\kappa = 0$. Let $(u,\mu,\theta)$ denote the unique weak solution of the system \eqref{CH:INT} to the initial datum $u_0$ in the sense of Theorem~\ref{THM:WP}. Then, in addition to the regularity properties \eqref{REG:INT} and \eqref{REG:INT:2}, it holds that
	\begin{align}
	\label{REG:INT:3}
	\left\{\;
	\begin{aligned}
	&u \in H^1(0,T;\Vk), \quad &&\mu \in L^\infty(0,T;H^1(\Omega)) \cap L^2(0,T;H^2(\Omega)),\\
	&\pdnu \mu \in L^2(0,T;L^2(\Gamma)), \quad  &&\theta \in L^\infty(0,T;H^1(\Gamma)) \cap L^2(0,T;H^2(\Gamma)).
	\end{aligned}
	\right.
	\end{align}
	This means that $(u,\mu,\theta)$ is a strong solution of the system \eqref{CH:INT}.
\end{thm}

\begin{proof}
To prove the assertion we will argue similar to the approach in \cite[s.~4.4]{Colli}.
Here, we use the letter $C$ to denote generic positive constants independent of $N$, $n$ and $\tau$ that may change their value from line to line.
Let $N$, $\tau$ and $(u^n,\mu^n,\theta^n), \; n=0,...,N$ be as defined in Step 1 of the proof of Theorem~\ref{THM:WP}. For brevity, we introduce the notation
\begin{align}
\label{DEF:DQ}
	\del_\tau u^{n+1} = \frac{u^{n+1}-u^n}{\tau},\quad
	\del_\tau \mu^{n+1} = \frac{\mu^{n+1}-\mu^n}{\tau},\quad
	\del_\tau \theta^{n+1} = \frac{\theta^{n+1}-\theta^n}{\tau}
\end{align}
to denote the backward difference quotient in time.
Let $n\in\{0,1,...,N-1\}$ be arbitrary. Testing \eqref{DWF:INT:1} with $w = -\pd_\tau \mu^{n+1} \in H^1(\Omega)$, \eqref{DWF:INT:2} with $z = -\pd_\tau \theta^{n+1} \in H^1(\Gamma)$ and adding the resulting equations leads to
\begin{align}
\label{EQ:REG:1}
\begin{aligned}
& -\intO \pd_\tau u^{n+1} \pd_\tau \mu^{n+1} \dx - \intG \pd_\tau u^{n+1} \pd_\tau \theta^{n+1} \dG \\
& \quad =  \frac{1}{2\tau} \big ( \norm{\nabla \mu^{n+1}}_{L^2(\Omega)}^2 - \norm{\nabla \mu^{n}}_{L^2(\Omega)}^2 + \norm{\nabla (\mu^{n+1} - \mu^n)}_{L^2(\Omega)}^2 \big ) \\
& \qquad + \frac{1}{2\tau} \big ( \norm{\gradg \theta^{n+1}}_{L^2(\Gamma)}^2 - \norm{\gradg \theta^{n}}_{L^2(\Gamma)}^2 + \norm{\gradg  (\theta^{n+1} - \theta^n)}_{L^2(\Gamma)}^2 \big ) \\
& \qquad + \frac{1}{2L \tau} \big (  \norm{\beta \theta^{n+1} - \mu^{n+1}}_{L^2(\Gamma)}^2 - \norm{\beta \theta^{n} - \mu^{n}}_{L^2(\Gamma)}^2 \\
& \qquad\qquad\qquad + \norm{\beta (\theta^{n+1} - \theta^n) -  (\mu^{n+1} - \mu^n)}_{L^2(\Gamma)}^2 \big ).
\end{aligned}
\end{align}
Since \eqref{ass:pot:reg} holds, the variational equation \eqref{DWF:INT:3} now holds for more general test functions $\eta \in \Vk$.
Taking the difference of \eqref{DWF:INT:3} for indices $n$ and $n+1$, and then choosing $\eta = \frac 1 \tau \pd_\tau u^{n+1} \in \Wk$ gives
\begin{align}
\label{EQ:REG:2}
\begin{aligned}
& \intO \pd_\tau \mu^{n+1} \pd_\tau u^{n+1} \dx + \intG \pd_\tau \theta^{n+1} \pd_\tau u^{n+1} \dG \\
& \quad = \norm{\nabla \pd_\tau u^{n+1}}_{L^2(\Omega)}^2 + \kappa \norm{\gradg \pd_\tau u^{n+1}}_{L^2(\Gamma)}^2 
	+ \intO \frac 1 \tau \big( F'(u^{n+1}) - F'(u^{n})\big) \pd_\tau u^{n+1} \dx\\
& \qquad  + \intG \frac 1 \tau \big( G'(u^{n+1}) - G'(u^{n}) \big) \pd_\tau u^{n+1} \dG.
\end{aligned}
\end{align}
Using the monotonicity of $F_1'$ and $G_1'$, the Lipschitz continuity of $F_2'$ and $G_2'$, after summing \eqref{EQ:REG:1} and \eqref{EQ:REG:2} and neglecting some non-negative terms we arrive at
\begin{equation}\label{Imp:reg:1}
\begin{aligned}
& \norm{\nabla \mu^{n+1}}_{L^2(\Omega)}^2 - \norm{\nabla \mu^{n}}_{L^2(\Omega)}^2 + \norm{\gradg \theta^{n+1}}_{L^2(\Gamma)}^2 - \norm{\gradg \theta^n}_{L^2(\Gamma)}^2   \\
& \qquad + \frac{1}{L} \norm{\beta \theta^{n+1} - \mu^{n+1}}_{L^2(\Gamma)}^2 - \frac{1}{L} \norm{\beta \theta^n - \mu^n}_{L^2(\Gamma)}^2 \\
& \qquad + 2 \tau \big ( \norm{\nabla \pd_\tau u^{n+1}}_{L^2(\Omega)}^2 + \kappa \norm{\gradg \pd_\tau u^{n+1}}_{L^2(\Gamma)}^2 \big ) \\[1ex]
& \quad \leq \tau C_\text{Lip} \big ( \norm{\pd_\tau u^{n+1}}_{L^2(\Omega)}^2 + \norm{\pd_\tau u^{n+1}}_{L^2(\Gamma)}^2 \big )
\end{aligned}
\end{equation}
where $C_\text{Lip}>0$ depends only on the Lipschitz constants of $F_2'$ and $G_2'$.
Since $\pd_\tau u^{n+1} \in \Wo$, we invoke Lemma~\ref{LEM:INT} (with $\alpha=C_\text{Lip}^{-1}$) to see that 
\begin{align}
\label{Imp:reg:1.5}
C_\text{Lip} \big ( \norm{\pd_\tau u^{n+1}}_{L^2(\Omega)}^2 + \norm{\pd_\tau u^{n+1}}_{L^2(\Gamma)}^2 \big ) \leq  \norm{\nabla \pd_\tau u^{n+1}}_{L^2(\Omega)}^2 +  C \norm{\pd_\tau u^{n+1}}_{L,\beta,*}^2.
\end{align}
According to \eqref{DEF:MT0} and \eqref{mu:theta:n+1}, the functions $\mu^{n+1}$ and $\theta^{n+1}$ can be expressed as
\begin{align*}
\mu^{n+1} = \SS_\Om(\pd_\tau u^{n+1}) + \beta c^{n+1},\quad \theta^{n+1} = \SS_\Ga(\pd_\tau u^{n+1}) + c^{n+1}.
\end{align*}
It thus follows that
\begin{align}\label{Imp:reg:2}
\begin{aligned}
&\norm{\pd_\tau u^{n+1}}_{L,\beta,*}^2 
= \norm{\SS(\pd_\tau u^{n+1})}_{L,\beta}^2
\\
&\quad = \norm{\nabla \mu^{n+1}}_{L^2(\Omega)}^2 + \norm{\gradg \theta^{n+1}}_{L^2(\Gamma)}^2 + \frac{1}{L}\norm{\beta \theta^{n+1} - \mu^{n+1}}_{L^2(\Gamma)}^2.
\end{aligned}
\end{align} 
Substituting the estimate \eqref{Imp:reg:1.5} and the identity \eqref{Imp:reg:2} into \eqref{Imp:reg:1}, we get 
\begin{equation}\label{Imp:reg:2.5}
\begin{aligned}
& \norm{\nabla \mu^{n+1}}_{L^2(\Omega)}^2 - \norm{\nabla \mu^{n}}_{L^2(\Omega)}^2 + \norm{\gradg \theta^{n+1}}_{L^2(\Gamma)}^2 - \norm{\gradg \theta^n}_{L^2(\Gamma)}^2   \\
& \qquad + \frac{1}{L} \norm{\beta \theta^{n+1} - \mu^{n+1}}_{L^2(\Gamma)}^2 - \frac{1}{L} \norm{\beta \theta^n - \mu^n}_{L^2(\Gamma)}^2 \\
& \qquad + \tau \big ( \norm{\nabla \pd_\tau u^{n+1}}_{L^2(\Omega)}^2 + \kappa \norm{\gradg \pd_\tau u^{n+1}}_{L^2(\Gamma)}^2 \big ) \\
& \quad \leq C \tau \left( \norm{\nabla \mu^{n+1}}_{L^2(\Omega)}^2 + \norm{\gradg \theta^{n+1}}_{L^2(\Gamma)}^2 + \frac{1}{L}\norm{\beta \theta^{n+1} - \mu^{n+1}}_{L^2(\Gamma)}^2 \right).
\end{aligned}
\end{equation}
Now we sum the inequalities \eqref{Imp:reg:2.5} from $n = 0$ to an arbitrary index $k \leq N-1$. With the help of the piecewise constant extensions \eqref{pi:cons} and piecewise linear extensions \eqref{pi:lin}, we find that 
\begin{equation}\label{Imp:reg:3}
\begin{aligned}
& \norm{\nabla \mu^{k+1}}_{L^2(\Omega)}^2 + \norm{\gradg \theta^{k+1}}_{L^2(\Gamma)}^2 + \frac{1}{L} \norm{\beta \theta^{k+1} - \mu^{k+1}}_{L^2(\Gamma)}^2 \\
& \qquad + \int_0^{k\tau} \norm{\nabla \bar u_N'(s)}_{L^2(\Omega)}^2 + \kappa \norm{\gradg \bar u_N'(s)}_{L^2(\Gamma)}^2 \ds \\
& \quad \leq C\int_0^T \norm{\nabla \mu_N(s)}_{L^2(\Omega)}^2 + \norm{\gradg \theta_N(s)}_{L^2(\Gamma)}^2 + \frac{1}{L} \norm{\beta \theta_N(s) - \mu_N(s)}_{L^2(\Gamma)}^2 ds \\
& \qquad + \norm{\nabla \mu_N(0)}_{L^2(\Omega)}^2 + \norm{\gradg \theta_N(0)}_{L^2(\Gamma)}^2 +\frac{1}{L} \norm{\beta \theta_N(0) - \mu_N(0)}_{L^2(\Gamma)}^2,
\end{aligned}
\end{equation}
where the prime indicates the derivative with respect to the time variable.
We now recall that $\mu_N(0) = \mu^0$ and $\theta_N(0) = \theta^0$, which according to \eqref{DWF:INT:3} satisfy
\begin{align}
\label{WF:INT:3:0}
\intO \mu^0 \eta \dx + \intG \theta^0 \eta \dG = \intO \nabla u^0 \cdot \nabla \eta + F'(u^0) \eta \dx + \intG \kappa \gradg u^0 \cdot \gradg \eta + G'(u^0) \eta \dG
\end{align}
for all $\eta \in \Vk$. We can first take $\eta \in C^\infty_c(\Omega) \subset \Vk$ to deduce that 
\begin{align*}
\mu^0 = - \Delta u^0 + F'(u_0) \; \text{ in the sense of distributions in } \Omega.
\end{align*}
By assumption of $u_0 \in H^3(\Omega)$, it holds that $\mu^0 \in H^1(\Omega)$ and the above identity holds a.e.~in $\Omega$.  Then, returning to \eqref{WF:INT:3:0}, we use the above identity to deduce that 
\begin{align*}
\theta^0 = - \kappa \LB u^0 + G'(u^0) + \pdnu u^0 \; \text{ in the sense of distributions on } \Gamma.
\end{align*}
If $\kappa > 0$, by the assumptions $u_0 \vert_\Ga \in H^3(\Ga)$ and \eqref{ass:pot}, we infer that $\theta^0 \in H^1(\Gamma)$, and if $\kappa = 0$, then by the assumption $u_0 \in H^3(\Omega) \subset H^{\frac{5}{2}}(\Ga)$, we see that $\pdnu u^0 \in H^1(\Ga)$ and thus $\theta^0 \in H^1(\Gamma)$ as well.  Hence, recalling the uniform estimate \eqref{IEQ:EUN}, we infer that the right-hand side of \eqref{Imp:reg:3} can be bounded by a constant $C >0$ independent of $N$, $\kappa$ and $\tau$.
As $k$ was arbitrary, we conclude that 
\begin{align}
\label{Imp:reg:4:1}
\norm{\grad\mu_N(t)}_{L^2(\Omega)}^2 + \norm{\gradg \theta_N(t)}_{L^2(\Gamma)}^2 + \frac{1}{L} \norm{\beta \theta_N(t) - \mu_N(t)}_{L^2(\Gamma)}^2 & \leq C,\\
\label{Imp:reg:4:2}
\int_0^T \norm{\nabla \bar u_N'(s)}_{L^2(\Omega)}^2 + \kappa \norm{\gradg \bar u_N'(s)}_{L^2(\Gamma)}^2 \ds &\leq C
\end{align}
for any $t \in (0,T]$. From the estimates \eqref{EST:MUN} and \eqref{EST:THN} we now deduce that 
\begin{align*}
	\norm{\mu_N(t)}_{L^2(\Omega)}^2  + \norm{\theta_N(t)}_{L^2(\Gamma)}^2 \leq C \quad \text{ for all } t \in (0,T].
\end{align*}
Moreover, invoking Lemma~\ref{LEM:INT}, \eqref{IEQ:EUN}, \eqref{Imp:reg:2} and \eqref{Imp:reg:4:2}, it holds that
\begin{equation}\label{Imp:reg:5}
\begin{aligned}
& \int_0^T \norm{\bar u_N'(s)}_{L^2(\Omega)}^2 + \norm{\bar u_N'(s)}_{L^2(\Gamma)}^2 \ds \leq \int_0^T \norm{\nabla \bar u_N'(s)}_{L^2(\Omega)}^2 + C \norm{\bar u_N'(s)}_{L,\beta,*}^2 \ds \\
&  \quad \leq C + C \int_0^T  \norm{\nabla \mu_N(s)}_{L^2(\Omega)}^2 + \norm{\gradg \theta_N(s)}_{L^2(\Gamma)}^2 + \frac{1}{L} \norm{\beta \theta_N(s) - \mu_N(s)}_{L^2(\Gamma)}^2 \ds \\
& \quad \leq C.
\end{aligned}
\end{equation}
Hence, in addition to \eqref{EST:PCE} and \eqref{EST:HLD}, we infer from \eqref{Imp:reg:4:1}-\eqref{Imp:reg:5} the following uniform estimates
\begin{align*}
\norm{\mu_N}_{L^\infty(0,T;H^1(\Omega))} + \norm{\theta_N}_{L^\infty(0,T;H^1(\Gamma))} + \norm{\beta \theta_N - \mu_N}_{L^\infty(0,T;L^2(\Gamma))} \leq C, \\
\norm{\bar u_N'}_{L^2(0,T;H^1(\Omega))} + \norm{\bar u_N'}_{L^2(0,T;L^2(\Gamma))} + \kappa \norm{\gradg \bar u_N'}_{L^2(0,T;L^2(\Gamma))} \leq C,
\end{align*}
leading to limit functions $(u, \mu, \theta)$ exhibiting the additional regularity
\begin{align*}
u \in H^1(0,T;\Vk), \quad \mu \in L^\infty(0,T;H^1(\Omega)), \quad \theta \in L^\infty(0,T;H^1(\Gamma)).
\end{align*}
Returning to \eqref{WF:INT:1} and \eqref{WF:INT:2}, which are the weak formulations of the elliptic problems
\begin{align*}
\begin{cases}
\Delta \mu = \pd_t u & \text{ in } \Omega, \\
\pdnu \mu = \tfrac 1 L(\beta \theta - \mu) & \text{ on } \Gamma,
\end{cases} 
\qquad \LB \theta = \pd_t u + \tfrac 1 L\beta (\beta \theta - \mu) \;\text{ on } \Gamma,
\end{align*}
we invoke elliptic regularity theory (see, e.g., \cite[s.~5, Prop.~7.7]{taylor} for the system in the bulk and \cite[s.~5, Thm.~1.3]{taylor} for the equation on the boundary) to find that 
\begin{align*}
\begin{split}
\norm{\mu}_{H^2(\Omega)} 
& \leq C \big ( \norm{\Delta \mu}_{L^2(\Omega)} + \norm{\mu}_{H^1(\Omega)} + \norm{\pdnu \mu}_{H^{1/2}(\Gamma)} \big ) \\
& = C \big ( \norm{\pd_t u}_{L^2(\Omega)} + \norm{\mu}_{H^1(\Omega)} + \norm{\theta}_{H^1(\Gamma)} \big ), 
\end{split}
\\[1ex]
\begin{split}
\norm{\theta}_{H^2(\Gamma)} 
& \leq C \big ( \norm{\LB \theta}_{L^2(\Gamma)} + \norm{\theta}_{H^1(\Gamma)} \big ) \\
& \leq C \big ( \norm{\pd_t u}_{L^2(\Gamma)} + \norm{\theta}_{H^1(\Gamma)} + \norm{\mu}_{H^1(\Omega)} \big ).
\end{split}
\end{align*}
Hence, we conclude that 
\begin{align*}
\mu \in L^2(0,T;H^2(\Omega)), \quad \pdnu \mu \in L^2(0,T;L^2(\Gamma)), \quad \theta \in L^2(0,T;H^2(\Gamma))
\end{align*}
and thus, the proof is complete.
\end{proof}

\subsection{Well-posedness results for the LW model and the GMS model}

For the reader's convenience, we now also present the well-posedness results for the LW model \eqref{CH:LW} and the GMS model \eqref{CH:GMS}. 

\begin{prop}[Well-posedness of the LW model] \label{prop:lw}
	Suppose that \eqref{ass:dom} - \eqref{ass:pot} hold and let $m = (m_b, m_s)  \in \R^2$ be arbitrary. Then for any 
	\begin{align}
	\label{DEF:VM}
		u_0^* \in \Vm := \{v \in \Vk \, : \, \mean{v}_\Om = m_b, \mean{v}_\Ga = m_s  \}
	\end{align}
	satisfying $F(u_0^*) \in L^1(\Omega)$ and $G(u_0^*) \in L^1(\Gamma)$, there exists a unique weak solution $(u^*, \mu^*, \theta^*)$ to \eqref{CH:LW}
	in the following sense:
	\begin{enumerate}[label=$(\mathrm{\roman*})$, ref = $\mathrm{\roman*}$]
		\item The functions $(u^*,\mu^*,\theta^*)$ have the following regularity
		\begin{align*}
		\left\{\;
		\begin{aligned}
		u^* & \in C([0,T];L^2(\Omega)) \cap L^\infty\big(0,T;H^1(\Omega)\cap L^p(\Omega)\big) \cap H^1(0,T;H^1(\Omega)'), \\
		u^*\vert_{\Sigma_T} & \in C([0,T];L^2(\Gamma)) \cap L^{\infty}(0,T;\Xk \cap L^q(\Ga)) \cap H^1(0,T;H^1(\Gamma)'), \\
		\mu^* &\in L^2(0,T;H^1(\Omega)),
		\qquad
		\theta^* \in L^2(0,T;H^1(\Gamma))
		\end{aligned}
		\right.
		\end{align*}
		and it holds that $u^*(t) \in \Vm$ for all $t\in[0,T]$. 
		\item
		The weak formulation
		\begin{subequations}\label{WF:LW}
			\begin{alignat}{2}
			\label{WF:LW:1}
			0 & = \inn{u^*_t}{w}_{H^1(\Omega)} + \intO \nabla \mu^* \cdot \nabla w \dx, \\
			\label{WF:LW:2}
			0 & = \inn{u^*_t}{z}_{H^1(\Gamma)} + \intG \gradg \theta^* \cdot \gradg z \dG, \\
			\label{WF:LW:3} 0 & = \intO \nabla u^* \cdot \nabla \eta + F'(u^*) \eta - \mu^* \eta \dx + \intG \kappa \gradg u^* \cdot \gradg \eta + G'(u^*) \eta - \theta^* \eta \dG,
			\end{alignat}
		\end{subequations}
		is satisfied almost everywhere in $[0,T]$ for all test functions $w\in H^1(\Omega)$, $z\in H^1(\Gamma)$ and $\eta\in \Vk \cap L^\infty(\Omega)$ with $\eta\vert_\Gamma \in L^\infty(\Gamma)$. Moreover, the initial condition $ u^*(0)=u^*_0 $ is satisfied a.e.~in $\Omega$.
		\item For $E$ as defined in \eqref{DEF:EN}, the energy inequality 
		\begin{align}
		\label{WF:LW:DISS}
		 E\big(u^*(t)\big) + \frac 1 2 \int_0^t \norm{\grad\mu^*(s)}_{L^2(\Om)}^2 
		+ \norm{\gradg \theta^*(s)}_{L^2(\Ga)}^2 \ds 
		\;\le\; E(u_0^*)
		\end{align}
		is satisfied for all $t\in[0,T]$.
	\end{enumerate}
	If we additionally assume that \eqref{ass:pot:reg} holds, then the regularity assertions \eqref{REG:INT:2} also hold.
\end{prop}

The above well-posedness assertion was first established in \cite[Thm.~3.1 and Thm.~3.2]{LW}. In the case $\kappa=0$ the authors needed a strong assumption on the domain $\Omega$ and its boundary $\Gamma$. However, it was later shown in \cite{GK} that this assumption can actually be omitted if a slightly weaker notion of weak solutions is used. For a proof of Proposition~\ref{prop:lw} see \cite[Thm~2.1]{KL}, while the regularity assertion \eqref{REG:INT:2} can be shown with the arguments in Step 8 of Section \ref{sec:wellposed}.

\begin{prop}[Well-posedness of the GMS model]\label{prop:GMS} 
	Suppose that \eqref{ass:dom} - \eqref{ass:pot} hold with $\beta>0$ and let $m \in \R$ be arbitrary. Then for any $u_{0,*} \in \Wm$ satisfying $F(u_{0,*}) \in L^1(\Omega)$ and $G(u_{0,*}) \in L^1(\Gamma)$, there exists a unique weak solution $(u_*, \mu_*, \theta_*)$ to the system \eqref{CH:GMS} in the following sense
	\begin{enumerate}[label=$(\mathrm{\roman*})$, ref = $\mathrm{\roman*}$]
		\item The functions $(u_*,\mu_*,\theta_*)$ have the following regularity
		\begin{align}
		\label{REG:GMS}
		\left\{
		\begin{aligned}
			u_* & \in C([0,T];L^2(\Omega)) \cap L^\infty\big(0,T;H^1(\Omega)\cap L^p(\Omega)\big) 
				\cap H^1(0,T;\V'), \\
			u_*\vert_{\Sigma_T}  & \in C([0,T];L^2(\Ga)) \cap L^{\infty}(0,T;\Xk \cap L^q(\Ga)) 
			, \\
			\mu_* &\in L^2(0,T;H^1(\Omega)),
			\qquad
			\theta_* \in L^2(0,T;H^1(\Gamma))
		\end{aligned}
		\right.
		\end{align}
		and it holds that $\beta\theta_* = \mu_* \vert_{\Sigma_{T}}$ a.e.~on $\Sigma_T$.
		Moreover,
		$u_*(t) \in \Wm$ for all $t\in[0,T]$. 
		\item
		The weak formulation
		\begin{subequations}
			\label{WF:GMS}
			\begin{alignat}{2}
			\label{WF:GMS:1} 0 & = \inn{u_{*,t}}{w}_{\V,\beta}  + \beta\intO \nabla \mu_* \cdot \nabla w \dx + \intG \gradg \theta_* \cdot \gradg w\dG, \\
			\label{WF:GMS:2} 0 & = \intO \nabla u_* \cdot \nabla \eta + F'(u_*) \eta - \mu_* \eta \dx + \intG \kappa \gradg u_* \cdot \gradg \eta + G'(u_*) \eta - \theta_* \eta \dG
			\end{alignat}
		\end{subequations}
		is satisfied almost everywhere in $[0,T]$ for all $w\in \V$ and $\eta\in \Vk \cap L^\infty(\Omega)$ with $\eta\vert_\Gamma \in L^\infty(\Gamma)$. Moreover, the initial condition $ u_*(0)=u_{0,*} $ is satisfied a.e.~in $\Omega$.
		\item For $E$ as defined in \eqref{DEF:EN}, the energy inequality 
		\begin{align}
		\label{WF:GMS:DISS}
		 E\big(u_*(t)\big) + \frac 12 \int_0^t \norm{\grad\mu_*(s)}_{L^2(\Om)}^2 
			+ \norm{\gradg \theta_*(s)}_{L^2(\Ga)}^2 \ds 
		\;\le\; E(u_{0,*})
		\end{align}
		is satisfied for all $t\in[0,T]$.
	\end{enumerate}
	If we additionally assume that \eqref{ass:pot:reg} holds, then the regularity assertions \eqref{REG:INT:2} also hold. 
\end{prop}

A proof of the well-posedness assertion can be found in \cite[Thm.~3.2]{GMS}. We point out that the regularity results $u_* \in C([0,T];L^2(\Omega))$ and $u_*\vert_{\Sigma_T} \in C([0,T];L^2(\Ga))$ are not mentioned in \cite[Thm.~3.2]{GMS} but follow straightforwardly from the Aubin--Lions lemma, see Step 5 of Section \ref{sec:wellposed}. 

To establish the convergence rates in Section~4, we will need the following regularity result for solutions of the GMS model in the case $\kappa > 0$.

\begin{prop}[Higher regularity for the GMS model]\label{prop:GMS:reg}
	Suppose that \eqref{ass:dom} - \eqref{ass:pot} hold with $\kappa,\beta>0$ and let $m \in \R$ be arbitrary. Then, if $u_{0,*} \in  \Wm$ with $(u_{0,*},u_{0,*}\vert_\Ga) \in \H^2$ the unique weak solution $(u_*, \mu_*,\theta_*)$ to the system \eqref{CH:GMS} satisfies the further regularity 
	\begin{alignat*}{4}
	u_* &\in H^1(0,T;L^2(\Omega)), \quad &\mu_* &\in L^\infty(0,T;L^2(\Omega)) \cap L^2(0,T;H^2(\Omega)), \\
	u_* \vert_{\Sigma_T} &\in H^1(0,T;L^2(\Gamma)), \quad  &\mu_* \vert_{\Sigma_T} &\in L^\infty(0,T;L^2(\Gamma)) \cap L^2(0,T;H^2(\Gamma)).
	\end{alignat*}
	This means that $(u_*,\mu_*,\theta_*)$ is a strong solution of the system \eqref{CH:GMS}.
\end{prop}

The assertions of Proposition~\ref{prop:GMS:reg} do not follow immediately from the results in \cite{Colli}. However, they can be established with slight modifications in the proof of Theorem~\ref{THM:SWP}, which follows the line of argument in \cite{Colli}.

\section{Asymptotic limits}

In this section we investigate the asymptotic limits $L\to\infty$ and $L\to 0$ of the system \eqref{CH:INT}.  We first present some general estimates for solutions to the system \eqref{CH:INT}.

\paragraph{Uniform estimates.} Suppose that \eqref{ass:dom} - \eqref{ass:pot} hold and let $u_{0} \in  \Wm$ be any initial datum satisfying $F(u_{0}) \in L^1(\Omega)$ and $G(u_{0}) \in L^1(\Gamma)$.
For any $L>0$, let $(u^L, \mu^L, \theta^L)$ denote the corresponding weak solution to the system \eqref{CH:INT} in the sense of Theorem~\ref{THM:WP}. In the following, we use the letter $C$ to denote generic positive constants independent of $L$. From the energy inequality \eqref{WF:INT:DISS} we conclude that
\begin{align}
\begin{aligned}
\| \Ll u \|_{L^{\infty}(0,T;H^1(\Om))} + \| \Ll u \|_{L^{\infty}(0,T;L^p(\Om))} + \| \Ll u \|_{L^{\infty}(0,T;\Xk)} + \| \Ll u \|_{L^{\infty}(0,T;L^q(\Ga))}
&\le C, \\
\label{lim:1}
\| \nabla \Ll \mu \|_{L^2(Q_T)}^2 
+ \| \gradg \Ll \theta \|_{L^2(\Sigma_T)}^2 
+ \tfrac 1 L \| \beta \Ll \theta - \Ll \mu \|_{L^2(\Sigma_T)}^2 
&\le C.
\end{aligned}
\end{align}
Arguing as in Step 3 of the proof Theorem~\ref{THM:WP}, we additionally infer that
\begin{align}\label{lim:2}
\| \mu^L \|_{L^2(Q_T)} + \| \theta^L \|_{L^2(\Sigma_T)} \leq C.
\end{align}
Proceeding similarly as in Step 4 of the proof of Theorem~\ref{THM:WP} and exploiting the energy inequality \eqref{WF:INT:DISS}, we derive the uniform estimate 
\begin{align}
\label{lim:3:LW}
\| u_t^L \|_{L^2(0,T;H^1(\Omega)')} + \| u_t^L \|_{L^2(0,T;H^1(\Gamma)')} \leq C\left(1+\tfrac{1}{\sqrt{L}}\right).
\end{align}
Let now $w\in\V$ be an arbitrary test function, then testing \eqref{WF:INT:1} with $\beta w$ and \eqref{WF:INT:2} with $w$, summing and integrating the resulting equations yields the bound
\begin{align}
\label{lim:3:GMS}
	\| u_t^L \|_{L^2(0,T;\V')} \leq C \quad \text{ if }\; \beta>0,
\end{align}
where $\V'$ is endowed with the norm $\norm{\cdot}_{\V',\beta}$ as introduced in \eqref{pre:V}.

Assume additionally that \eqref{ass:pot:reg} holds,
we note that the arguments to the regularity assertion \eqref{REG:INT:2} do not involve the parameter $L$, and so we deduce that
\begin{equation}\label{lim:4}
\begin{aligned}
\| (u^L, u^L \vert_{\Sigma_T}) \|_{L^2(0,T;\H^3)} \leq C & \text{ if } \kappa > 0, \\
\| (u^L, u^L \vert_{\Sigma_T}) \|_{L^2(0,T;(H^{5/2}(\Om) \times H^{2}(\Ga))} \leq C & \text{ if } \kappa = 0.
\end{aligned}
\end{equation}

These uniform estimates can now be used to establish our convergence results.

\subsection{Convergence to the LW model as $L\to\infty$}

\begin{thm}[Asymptotic limit $L \to \infty$]\label{THM:ASY:LW}
	Suppose that \eqref{ass:dom} - \eqref{ass:pot} hold and let $(m_b, m_s) \in\R^2$, $\beta \neq 0$ and $\kappa \geq 0$ be arbitrary.  
	For any initial datum $u_{0}^* \in \W_{\beta,m}^\kappa$ with $m := \beta |\Om| m_b + |\Ga| m_s$, $\mean{u_0^*}_\Omega = m_b $, $\mean{u_0^*}_\Gamma = m_s$,   $F(u_{0}^*) \in L^1(\Omega)$ and $G(u_{0}^*) \in L^1(\Gamma)$, let $(u^L, \mu^L, \theta^L)$ denote the unique weak solution of the system \eqref{CH:INT} in the sense of Theorem~\ref{THM:WP}. 
	Then there exist functions $(u^*,\mu^*,\theta^*)$ such that 
	\begin{alignat*}{2}
	u^L &\to u^* 
		&& \text{ weakly-* in } L^\infty(0,T;H^1(\Om) \cap L^p(\Om) ), \\
		&&&\quad \text{ weakly in } H^1(0,T;H^1(\Omega)'), \\
		&&&\quad \text{ and strongly in } C([0,T];L^2(\Om)), \\
	u^L \vert_{\Sigma_T} &\to u^* \vert_{\Sigma_T} \; 
		&& \text{ weakly-* in } L^\infty(0,T;\Xk \cap L^q(\Ga) ), \\
		&&&\quad \text{ weakly in } H^1(0,T;H^1(\Gamma)'), \\
		&&&\quad \text{ and strongly in } C([0,T];L^2(\Gamma)), \\
	\mu^L &\to \mu^* 
		&& \text{ weakly in } L^2(0,T;H^1(\Omega)), \\
	\theta^L & \to \theta^* 
		&& \text{ weakly in } L^2(0,T;H^1(\Gamma)),\\
	\tfrac 1 L (\beta\theta^L - \mu^L \vert_{\Sigma_T}) & \to 0 
		&& \text{ strongly in } L^2(\Sigma_T),
	\end{alignat*}
	as $L\to\infty$, and the limit $(u^*,\mu^*,\theta^*)$ is the unique weak solution of the LW model \eqref{CH:LW} to the initial datum $u_{0}^*$.
	
	\pagebreak[2] 
	
	\noindent  If additionally
	\begin{itemize}
	\item \eqref{ass:pot:reg} holds, then 
	\begin{align*}
	(u^L, u^L \vert_{\Sigma_T}) \to (u^*, u^* \vert_{\Sigma_T}) \text{ weakly in } \begin{cases}
	L^2(0,T; \H^3) &  \text{ if } \kappa > 0, \\
	L^2(0,T; H^{5/2}(\Om) \times H^2(\Ga)) & \text{ if } \kappa = 0.
	\end{cases}
	\end{align*}
	\item \eqref{ass:pot:reg} holds and $(u_0^*,u_0^*\vert_\Gamma) \in \H^3$ if $\kappa>0$ or $u_0^* \in H^3(\Om)$ if $\kappa=0$, then there exists a constant $C> 0$ independent of $L$ and $\kappa$ such that
	\begin{align*}
	\norm{\nabla (u^L - u^*)}_{L^2(Q_T)} +  \sqrt{\kappa}\; \norm{\gradg (u^L - u^*)}_{L^2(\Sigma_T)} \leq \frac{C}{\sqrt{L}}, \\
	\sup_{t \in (0,T)} \left \| \int_0^t \pdnu \mu^L(s) \ds \right \|_{L^2(\Gamma)}  + \norm{\pdnu \mu^L}_{L^2(\Sigma_T)} \leq \frac{C}{\sqrt{L}}
	\end{align*}
	\item \eqref{ass:pot:reg} holds, $\kappa > 0$ and $(u_0^*,u_0^*\vert_\Gamma) \in \H^3$, then there exists a constant $C>0$ independent of $L$ such that 
		\begin{align*}
		\norm{u^L - u^*}_{L^{\infty}(0,T;L^2(\Omega))} + \norm{u^L - u^*}_{L^{\infty}(0,T;L^2(\Gamma))} \leq \frac{C}{L^{1/4}}.
		\end{align*}
\end{itemize}

\end{thm}

\bigskip

\begin{proof}
	In this proof we use the letter $C$ to denote generic positive constants independent of $L$, $N$, $n$ and $\tau$ that may change their value from line to line.	
	\subparagraph{Step 1: Convergence in the limit $L\to\infty$.}
	Let $(L_k)_{k\in\N} \subset [1,\infty)$ denote an arbitrary sequence satisfying $L_k\to \infty$ as $k\to\infty$.
	For any $k\in\N$, let $(u^k, \mu^k, \theta^k) = (u^{L_k}, \mu^{L_k}, \theta^{L_k})$ denote the unique weak solution to the system \eqref{CH:INT} corresponding to the parameter $L_k$.
	Hence, from the uniform bounds \eqref{lim:1}, \eqref{lim:2} and \eqref{lim:3:LW} we infer the existence of functions $(u^*, \mu^*, \theta^*)$ such that 
	\begin{subequations}
		\label{LW:lim:comp}
	\begin{alignat}{4}
	u^k &\to u^* 
	&& \text{ weakly-* in } L^\infty(0,T;H^1(\Om) \cap L^p(\Om)) \\
	&&& \quad \text{ weakly in } H^1(0,T;H^1(\Omega)'), \\
	u^k \vert_{\Sigma_T} &\to u^* \vert_{\Sigma_T} \; 
	&& \text{ weakly-* in } L^\infty(0,T;\Xk \cap L^q(\Ga)) \\
	&&& \quad \text{ weakly in } H^1(0,T;H^1(\Gamma)'), \\
	\mu^k &\to \mu^* && \text{ weakly in } L^2(0,T;H^1(\Omega)), \\
	\theta^k & \to \theta^* && \text{ weakly in } L^2(0,T;H^1(\Gamma)), 
	\end{alignat}
	as $k \to \infty$ along a non-relabelled subsequence.
	By the Aubin--Lions lemma we deduce that
	\begin{alignat}{4}
	u^k &\to u^* &&\quad \text{ strongly in } C([0,T]; L^2(\Omega)) \\
	u^k \vert_{\Sigma_T} &\to u^* \vert_{\Sigma_T} \; &&\quad \text{ strongly in } C([0,T];L^2(\Gamma)),
	\end{alignat}
		\end{subequations}
	as $k \to \infty$ after another subsequence extraction. 
	Moreover, from \eqref{lim:1} it follows that 
	\begin{align}\label{LW:unif}
	\frac{1}{L_k}\norm{\beta\theta^k - \mu^k}_{L^2(\Sigma_T)} \leq \frac{C}{\sqrt{L_k}} \to 0 
	\quad \text{as}\; k\to\infty.
	\end{align}	
	It is clear from the convergence properties in \eqref{LW:lim:comp} that the triplet $(u^*, \mu^*, \theta^*)$ has the desired regularity as stated in item (i) of Proposition~\ref{prop:lw}.  For arbitrary $w\in H^1(\Om)$, $z\in H^1(\Ga)$ and  $\eta \in \Vk \cap L^\infty$ with $\eta\vert_\Ga \in L^\infty(\Ga)$, from the weak formulation \eqref{WF:INT} of the system \eqref{CH:INT} written for $(u^k, \mu^k, \theta^k)$:
	\begin{subequations}
		\label{WFL:INT}
		\begin{align} 
		\label{WFL:INT:1}
		\inn{\delt u^k}{ w}_{H^1(\Omega)}  &= - \intO \grad\mu^k \cdot \grad w  \dx  
		+ \frac{1}{L_k} \intG (\beta \theta^k - \mu^k) w \dG,  \\
		\label{WFL:INT:2}
		\inn{\delt u^k}{ z}_{H^1(\Gamma)} &= - \intG \gradg\theta^k \cdot \gradg z  \dG 
		- \frac{1}{L_k} \intG (\beta \theta^k - \mu^k) \beta z \dG,  \\
		\label{WFL:INT:3}
		\begin{split}
		\intO \mu^k \eta \dx + \intG \theta^k \eta \dG 
		& = \intO \grad u^k\cdot\grad\eta + F'(u^k)\eta \dx  \\
		&\quad + \intG \kappa \gradg u^k\cdot\gradg\eta + G'(u^k)\eta  \dG	,
		\end{split}
		\end{align}
	\end{subequations} 
We multiply all equations in \eqref{WFL:INT} with arbitrary test functions in $C^\infty_c([0,T])$ depending only on $t$, integrate with respect to $t$ from $0$ to $T$, and pass to the limit $k \to \infty$ with the help of the convergence results \eqref{LW:lim:comp} and \eqref{LW:unif}.  For the terms involving $F'(u^k)$ and $G'(u^k)$ the generalised dominated convergence theorem \cite[p.~60]{Alt} can be used.  Hence, we infer that $(u^*, \mu^*, \theta^*)$ satisfies the weak formulation \eqref{WF:LW}. 
	Moreover, using weak lower semicontinuity arguments, we can pass to the limit in the energy inequality \eqref{WF:INT:DISS} for $(u^k, \mu^k, \theta^k)$ whilst neglecting the non-negative boundary integral term involving $L_k$, leading to the energy inequality \eqref{WF:LW:DISS}. Lastly, choosing $w = 1$ in \eqref{WFL:INT:1} and $z = 1$ in \eqref{WFL:INT:2}, multiplying by a $C^\infty_c([0,T])$ function, integrating over $[0,T]$ and passing to the limit leads to the property that for all $t \in (0,T]$,
	\begin{align*}
	\mean{u^*(t)}_\Om = \mean{u^*_0}_\Om = m_b , \quad 
	\mean{u^*(t)}_\Ga = \mean{u^*_0}_\Ga = m_s .
	\end{align*}
	This means $u^*(t) \in \V^\kappa_{(m_b,m_s)} $ for all $t \in [0,T]$. This proves that $(u^*, \mu^*, \theta^*)$ is a weak solution of the LW model \eqref{CH:LW} in the sense of Proposition~\ref{prop:lw}.
	
	If \eqref{ass:pot:reg} holds, then \eqref{lim:4} implies the weak convergence of $(u^k, u^k \vert_{\Sigma_T})$ to $(u^*, u^* \vert_{\Sigma_T})$ in $L^2(0,T;\H^3)$ if $\kappa > 0$ and in $ L^2(0,T;H^{5/2}(\Om) \times H^2(\Ga)) $ if $\kappa = 0$.  Suppose further that the initial condition satisfies $(u_0^*,u_0^*\vert_\Gamma)\in\H^3$ if $\kappa>0$ or $u_0^*\in H^3(\Om)$ if $\kappa=0$, then by Theorem \ref{THM:SWP}, the solution $(u^k, \mu^k, \theta^k)$ is a strong solution to \eqref{CH:INT}, and thus we have the relation 
	\begin{align*}
	\deln\mu^{k} = \frac{1}{L_k} ( \beta\theta^k - \mu^k) \text{ holding a.e.~on }\Sigma_T,
	\end{align*}
	and \eqref{LW:unif} implies the estimate
	\begin{align*}
	\|\deln \mu^k\|_{L^2(\Sigma_T)} \leq \frac{C}{\sqrt{L_k}}.
	\end{align*}

	By uniqueness of solutions to the LW model, which is independent of the choice of the extracted subsequence, we conclude by standard arguments that the above convergence results hold true for the whole sequence. Moreover, as the sequence $(L_k)_{k\in\N}$ was arbitrary, the convergence assertions are established for $L\to\infty$.

	\subparagraph{Step 2: Convergence rates.}
	For $L \in [1,\infty)$, let $(u^L, \mu^L, \theta^L)$ denote the unique weak solution to \eqref{CH:INT} corresponding to initial data $u_0^*$. Recall that $L \deln\mu^L = \beta\theta^L - \mu^L$ holds a.e.~on $\Sigma_T$, we define
	\begin{align*}
		(\hat u,\hat \mu,\hat \theta) := (u^L - u^*,  \mu^L - \mu^*,  \theta^L - \theta^* ).
	\end{align*}
	Then, it follows from the weak formulations \eqref{WF:INT} and \eqref{WF:LW}  that
	\begin{subequations}\label{LW:d}
		\begin{alignat}{2}
		\label{LW:d:1} \inn{\hat u_t}{w}_{H^1(\Omega)} & =  - \intO \nabla \hat \mu \cdot \nabla w \dx + \intG \pdnu \mu^L w \dG,\\
		\label{LW:d:2} \inn{\hat u_t}{z}_{H^1(\Gamma)} & =  - \intG \gradg \hat \theta \cdot \gradg z \dG -  \intG \pdnu \mu^L \beta z \dG, \\
		\label{LW:d:3} 
		\begin{split}
			\int_\Omega \hat\mu \eta \dx + \intG \hat\theta \eta \dG & = \int_\Omega \nabla \hat u \cdot \nabla \eta + \big(F'(u^L) - F'(u^*)\big) \eta \dx \\
			& \quad + \intG \kappa \gradg \hat u \cdot \gradg \eta + \big(G'(u^L) - G'(u^*)\big) \eta \dG,
		\end{split}
		\end{alignat}
	\end{subequations}
	for all test functions $w\in H^1(\Om)$, $z\in H^1(\Ga)$ and  $\eta \in \Vk \cap L^\infty$ with $\eta\vert_\Ga \in L^\infty(\Ga)$.
	Let now $t_0 \in (0,T]$, $ w \in L^2(0,T;H^1(\Omega))$ and $z \in L^2(0,T;H^1(\Gamma))$ be arbitrary. In the following we use once more the notation $Q_{t_0} = \Omega \times (0, t_0)$, $\Sigma_{t_0} = \Gamma \times (0,t_0)$. 
	Proceeding as in Step 7 of the proof of Theorem~\ref{THM:WP}, we find that
	\begin{align}
		\label{IEQ:CR}
	\begin{aligned}
		&\int_{Q_{t_0}} \hat u w \dx \dt + \int_{\Sigma_{t_0}} \hat u z \dG \dt \\
		&\quad = -\int_{Q_{t_0}} \nabla \left ( \int_{0}^t \hat \mu \ds \right ) \cdot \nabla w \dx \dt
		\;- \int_{\Sigma_{t_0}} \gradg \left ( \int_{0}^t \hat \theta \ds \right ) \cdot \gradg z \dG \dt \\
		&\qquad - \frac 1 L \int_{\Sigma_{t_0}} \left( \int_{0}^t \beta \theta^L - \mu^L \ds \right) (\beta z-w) \dG\dt.
	\end{aligned}
	\end{align}
	In the following, we use the notation
	\begin{align}
	\label{DEF:*}
	(1 \star f)(t) = \int_0^t f(s) \ds.
	\end{align}
	Here, $f$ may be scalar or vector-valued. In particular, this implies the relations
	\begin{align}
	\label{REL:*}
	\begin{aligned}
	\frac{d}{dt}\, \frac 1 2 \intO |(1 \star f)(t)|^2 \dx &= \intO \big(1\star f\big)(t) \cdot f(t) \dx, \\[1ex]
	\frac{d}{dt}\, \frac 1 2 \intG |(1 \star f)(t)|^2 \dG &= \intG \big(1\star f\big)(t) \cdot f(t) \dG.
	\end{aligned}
	\end{align}
	Now, plugging $w=\hat\mu$ and $z=\hat\theta$ into \eqref{IEQ:CR} 
	and using the relations \eqref{REL:*} as well as the decomposition
	\begin{align*}
		\intG \big(1 \star \pdnu \mu^L\big)(t)\, \big(\beta \hat \theta(t) - \hat \mu(t)\big) \dG 
			= \intG \big(1 \star \pdnu \mu^L\big)(t)\, \big(L \pdnu \mu^L(t) - (\beta\theta^* - \mu^*)(t)\big) \dG,
	\end{align*}
	for almost all $t\in(0,T)$, a straightforward computation yields
	\begin{align}
	\label{IEQ:CR2}
		\begin{aligned}
		&\frac 1 2 \Big( \norm{ \grad(1\star \hmu)(t_0) }_{L^2(\Om)}^2 		 	
		 	+ \norm{ \gradg(1\star \htheta)(t_0) }_{L^2(\Ga)}^2  
		 	+ L \norm{(1\star \deln \mu^L)(t_0)}_{L^2(\Ga)}^2 \Big)\\
		&\quad= 	-\int_{Q_{t_0}} \hat u \hat \mu \dx \dt 
			- \int_{\Sigma_{t_0}} \hat u \hat \theta \dx \dt
			+ \int_{\Sigma_{t_0}} (1\star\deln \mu^L) (\beta\theta^* - \mu^*) \dG\dt.
		\end{aligned}
	\end{align}
	Furthermore, proceeding as in Step 7 of the proof of Theorem~\ref{THM:WP} we obtain 
	\begin{align}
	\label{IEQ:CR3}
		\begin{aligned}
		&\norm{\grad\hu}_{L^2(Q_{t_0})}^2 + \kappa \norm{\gradg\hu}_{L^2(\Sigma_{t_0})}^2	\\
		&\quad\le \int_{Q_{t_0}} \hat u \hat \mu \dx \dt 
			+ \int_{\Sigma_{t_0}} \hat u \hat \theta \dx \dt
			+ C_\text{Lip} \big( \norm{\hu}_{L^2(Q_{t_0})}^2 + \norm{\hu}_{L^2(\Sigma_{t_0})}^2 \big)
		\end{aligned}
	\end{align}
	where the constant $C_\text{Lip}> 0$ depends only on the Lipschitz constants of $F_2'$ and $G_2'$.
	Adding \eqref{IEQ:CR2} and \eqref{IEQ:CR3} and applying Young's inequality now gives
	\begin{align}
	\label{IEQ:CR4}
		\begin{aligned}
		&\frac 1 2 \norm{ \grad(1\star \hmu)(t_0) }_{L^2(\Om)}^2 		 	
			+ \frac 1 2\norm{ \gradg(1\star \htheta)(t_0) }_{L^2(\Ga)}^2  
			+ \frac L 2 \norm{(1\star \deln \mu^L)(t_0)}_{L^2(\Ga)}^2 \\
		&\qquad + \int_0^{t_0} \norm{\grad\hu}_{L^2(\Om)}^2 + \kappa \norm{\gradg\hu}_{L^2(\Ga)}^2 \dt	\\[1ex]
		&\quad\le \int_0^{t_0} 
			\frac{L}{2} \norm{(1\star \pdnu \mu^L)(t)}_{L^2(\Ga)}^2 
			+ \frac{1}{2L} \norm{\beta\theta^* - \mu^*}_{L^2(\Ga)}^2 \dt \\
		&\qquad + \int_0^{t_0}  C_\text{Lip} \big( \norm{\hu}_{L^2(\Om)}^2 + \norm{\hu}_{L^2(\Ga)}^2 \big) \dt.
		\end{aligned}
	\end{align}
	By the trace theorem as well as the chain of compact embeddings $H^1(\Omega) \emb H^{3/4}(\Omega) \emb H^{-1}(\Omega)$ 
	(where $H^{-1}(\Omega)$ denotes the dual space to $H^1_0(\Omega)$), we obtain the estimate (cf.~\cite[(3.67)]{GMS})
	\begin{align}\label{GMS:interp}
		C_\text{Lip} \Big ( \norm{\hat u}_{L^2(\Omega)}^2 + \norm{\hat u}_{L^2(\Gamma)}^2 \Big ) \leq C \norm{\hat u}_{H^{3/4}(\Omega)}^2 \leq \frac{1}{2} \norm{\nabla \hat u}_{L^2(\Omega)}^2 + C \norm{\hat u}_{H^{-1}(\Omega)}^2.
	\end{align}
	To control the $H^{-1}$-norm of $\hat u$, we introduce the function $\D(\hat u) \in H^2(\Omega) \cap H^1_0(\Omega)$ as the solution to Poisson's equation with homogeneous Dirichlet boundary condition and source term $ \hat u$, i.e., 
	\begin{align*}
	\begin{cases}
	- \Lx \D(\hat u) = \hat u& \text{ in } \Omega, \\
	\D(\hat u) = 0 & \text{ on } \Gamma,
	\end{cases}
	\end{align*}
	for a.e.~$t \in (0,T)$. It is well-known that $\norm{\nabla \D(\cdot)}_{L^2(\Omega)}$ is an equivalent norm to $\norm{\cdot}_{H^{-1}(\Omega)}$ on $H^{-1}(\Om)$.  Moreover, after integrating \eqref{LW:d:1} in time and testing with $\D(\hat u(t))$, we deduce that 
	\begin{equation*}
	\begin{aligned}
	\norm{\hat u(t)}_{H^{-1}(\Omega)}^2 & \leq \norm{\nabla \D(\hat u(t))}_{L^2(\Omega)}^2 
		=  \intO \hat u(t) \D(\hat u(t)) \dx \\
	& \leq C \norm{\nabla (1 \star \hat \mu)(t)}_{L^2(\Omega)} \norm{\nabla \D(\hat u(t))}_{L^2(\Omega)} \\
	& \leq C \norm{\nabla (1 \star \hat \mu)(t)}_{L^2(\Omega)} \norm{\hat u(t)}_{H^{-1}(\Omega)}
	\end{aligned}
	\end{equation*}
	for a.e.~$t\in(0,T)$. Consequently, for a.e.~$t\in(0,T)$,
	\begin{equation}\label{H-1:est}
	\norm{\hat u(t)}_{H^{-1}(\Omega)} 
	 	\leq C \norm{\nabla (1 \star \hat \mu)(t)}_{L^2(\Omega)}. 
	\end{equation} 
	Substituting this estimate into \eqref{GMS:interp} and plugging the resulting estimate into \eqref{IEQ:CR4}, we obtain
	\begin{align}
	\label{IEQ:CR5}
	\begin{aligned}
	&\frac 1 2 \norm{ \grad(1\star \hmu)(t_0) }_{L^2(\Om)}^2 		 	
	+ \frac 1 2\norm{ \gradg(1\star \htheta)(t_0) }_{L^2(\Ga)}^2  
	+ \frac L 2 \norm{(1\star \deln \mu^L)(t_0)}_{L^2(\Ga)}^2 \\
	&\qquad + \int_0^{t_0} \frac 1 2 \norm{\grad\hu}_{L^2(\Om)}^2 + \kappa \norm{\gradg\hu}_{L^2(\Ga)}^2 \dt	\\[1ex]
	&\quad\le \int_0^{t_0} 
	\frac{L}{2} \norm{(1\star \pdnu \mu^L)(t)}_{L^2(\Ga)}^2 
	+ \frac{1}{2L} \norm{\beta\theta^* - \mu^*}_{L^2(\Ga)}^2 + C \norm{\nabla (1 \star \hat \mu)(t)}_{L^2(\Omega)}^2  \dt 
	\end{aligned}
	\end{align}
	Now, since $t_0$ was arbitrary, a Gronwall argument implies the existence of a constant $C$ independent of $L$ and $\kappa$ such that 
	\begin{align*}
	& \sup_{t \in (0,T)} \Big ( \norm{ \nabla (1\star \hat \mu)(t)}_{L^2(\Omega)}^2 + \norm{ \gradg (1\star \hat \theta)(t)}_{L^2(\Gamma)}^2 + L \norm{(1 \star \pdnu \mu^L)(t)}_{L^2(\Gamma)}^2 \Big )  \\
	& \qquad +  \norm{\nabla \hat u}_{L^2(Q_T)}^2 + \kappa \norm{\gradg \hat u}_{L^2(\Sigma_T)}^2  \\[1ex]
	& \quad \leq \frac{C}{L} \norm{\beta\theta^* - \mu^*}_{L^2(\Sigma_T)}^2 \leq \frac{C}{L},
	\end{align*}
	From this we obtain the convergence rates 
	\begin{align}\label{LW:CR:1}
	\norm{\nabla \hat u}_{L^2(Q_T)} + \sqrt{\kappa} \norm{\gradg \hat u}_{L^2(\Sigma_T)} \leq \frac{C}{\sqrt{L}}, \quad  
	\sup_{t \in (0,T)} \left \|\int_0^t \pdnu \mu^L(s) \ds \right \|_{L^2(\Gamma)} \leq \frac{C}{L},
	\end{align}
	for a positive constant $C$ independent of $\kappa$ and $L$.	
	Next, assume $\kappa > 0$ and testing \eqref{LW:d:2}  with $z = \hat u$ yields after integration for a.e.~$t \in (0,T)$,
	\begin{align*}
	\| \hat u(t) \|_{L^2(\Ga)}^2 & \leq 2 \int_0^t \| \gradg \hat \theta \|_{L^2(\Ga)} \| \gradg \hat u \|_{L^2(\Ga)} + |\beta| \| \pdnu \mu^L \|_{L^2(\Ga)} \| \hat u \|_{L^2(\Ga)} \ds \\
	& \leq C \| \gradg (\theta^L - \theta^*) \|_{L^2(\Sigma_T)} \| \gradg \hat u \|_{L^2(\Sigma_T)} + C \| \pdnu \mu^L \|_{L^2(\Sigma_T)}^2 + \int_0^t \| \hat u \|_{L^2(\Ga)}^2 \ds \\
	& \leq \frac{C}{\sqrt{L}} + \int_0^t \| \hat u \|_{L^2(\Ga)}^2 \ds
	\end{align*}
on account of the uniform bound \eqref{lim:1} and the estimate \eqref{LW:CR:1}.   
By Gronwall's inequality we then infer that 
\begin{align*}
\| \hat u \|_{L^\infty(0,T;L^2(\Ga))} \leq \frac{C}{L^{1/4}}.
\end{align*}
Then, testing \eqref{LW:d:1} with $w = \hat u$ yields after integration for a.e.~$t \in (0,T)$,
\begin{align*}
\| \hat u (t) \|_{L^2(\Om)}^2 & \leq C \| \nabla (\mu^L - \mu^*) \|_{L^2(Q_T)} \| \nabla \hat u \|_{L^2(Q_T)} +  \| \pdnu \mu^L \|_{L^2(\Sigma_T)} \| \hat u \|_{L^2(\Sigma_T)} \leq \frac{C}{\sqrt{L}},
\end{align*}
and thus, the proof of Theorem~\ref{THM:ASY:LW} is complete.
\end{proof}

\subsection{Convergence to the GMS model as $L\to 0$}

\begin{thm}[Asymptotic limit $L \to 0$]\label{THM:ASY:GMS}
	Suppose that \eqref{ass:dom} - \eqref{ass:pot} hold and let $m \in \R, \kappa \geq 0$ and $\beta>0$ be arbitrary.  
	For any initial datum $u_{0,*} \in \W_{\beta,m}^\kappa$ with $F(u_{0,*}) \in L^1(\Omega)$ and $G(u_{0,*}) \in L^1(\Gamma)$, let $(u^L, \mu^L, \theta^L)$ denote the unique weak solution of the system \eqref{CH:INT} in the sense of Theorem~\ref{THM:WP}. 
	Then there exist functions $(u_*,\mu_*,\theta_*)$ such that 
	\begin{alignat*}{4}
	u^L &\to u_* && \text{ weakly in } H^1(0,T;\V'), \\
	u^L &\to u_* && \text{ weakly-* in } L^\infty(0,T;H^1(\Om) \cap L^p(\Om)), \\
	&&&\quad \text{ and strongly in } C([0,T];L^2(\Om)), \\
	u^L \vert_{\Sigma_T} &\to u_* \vert_{\Sigma_T} \; && \text{ weakly-* in } L^\infty(0,T;\Xk \cap L^q(\Ga)), \\
	&&&\quad \text{ and strongly in } C([0,T];L^2(\Gamma)), \\
	\mu^L &\to \mu_* && \text{ weakly in } L^2(0,T;H^1(\Omega)), \\
	\theta^L &\to \theta_*\; && \text{ weakly in } L^2(0,T;H^1(\Gamma)), \\
	\beta \theta^L - \mu^L\vert_{\Sigma_T} \; &\to 0 && \text{ strongly in } L^2(\Sigma_T)
	\end{alignat*}
	as $L\to 0$, with 
	\begin{align*}
	\norm{\beta\theta^L -  \mu^L}_{L^2(\Sigma_T)} \leq C\sqrt{L},
	\end{align*}
	and the limit $(u_*,\mu_*,\theta_*)$ is the unique weak solution of the GMS model \eqref{CH:GMS} to the initial datum $u_{0,*}$ with $\mu_*\vert_\Ga = \beta\theta_*$ a.e.~on $\Sigma_T$.
	
	\noindent 
	If additionally
	\begin{itemize}
	\item \eqref{ass:pot:reg} holds, then 
	\begin{align*}
	(u^L, u^L \vert_{\Sigma_T}) \to (u_*, u_* \vert_{\Sigma_T}) \text{ weakly in } \begin{cases}
	L^2(0,T; \H^3) &  \text{ if } \kappa > 0, \\
	L^2(0,T; H^{5/2}(\Om) \times H^2(\Ga)) & \text{ if } \kappa = 0.
	\end{cases}
	\end{align*}
	\item \eqref{ass:pot:reg} holds and $(u_{0,*},u_{0,*}\vert_\Gamma) \in \H^3$ if $\kappa>0$ or $u_{0,*} \in H^3(\Om)$ if $\kappa=0$, then there exists a constant $C> 0$ independent of $L$ and $\kappa$ such that
	\begin{align*}
	\norm{\nabla (u^L - u_*)}_{L^2(Q_T)} +  \sqrt{\kappa}\; 
	\norm{\gradg (u^L - u_*)}_{L^2(\Sigma_T)}  \leq C \sqrt{L}, \\[1ex]
		\sup_{t \in (0,T)} \left \| \int_0^t \big(\beta \theta^L -  \mu^L\big)(s) \ds \right \|_{L^2(\Gamma)}  \leq CL, \\[1ex]	
	\end{align*}
	\item \eqref{ass:pot:reg} holds, $\kappa > 0$ and $(u_{0,*},u_{0,*}\vert_\Gamma) \in \H^3$, then there exists a constant $C> 0$ independent of $L$ such that
	\begin{align*}
		\norm{u^L - u_*}_{L^{\infty}(0,T;L^2(\Omega))} + \norm{u^L - u_*}_{L^{\infty}(0,T;L^2(\Gamma))} \leq C L^{1/4}.
	\end{align*}
	\end{itemize}
\end{thm}

\bigskip

\begin{proof}
	In this proof we use the letter $C$ to denote generic positive constants independent of $L$, $N$, $n$ and $\tau$ that may change their value from line to line.
\subparagraph{Step 1: Convergence in the limit $L\to 0$.}
	Let $(L_k)_{k\in\N} \subset (0,1]$ denote an arbitrary sequence satisfying $L_k\to 0$ as $k\to\infty$.
	For any $k\in\N$, let $(u_k, \mu_k, \theta_k) = (u^{L_k}, \mu^{L_k}, \theta^{L_k})$ denote the unique weak solution to the system \eqref{CH:INT} corresponding to the parameter $L_k$. Then, in the limit $k\to\infty$, we infer from \eqref{lim:1}, \eqref{lim:2} and \eqref{lim:3:GMS} the existence of limit functions $(u_*, \mu_*, \theta_*)$ such that
	\begin{subequations}
		\label{GMS:lim:comp}
	\begin{alignat}{2}
	u_k &\to u_* &&\quad \text{ weakly in } H^1(0,T;\V'), \\
	u_k &\to u_* &&\quad \text{ weakly-* in } L^{\infty}\big(0,T;H^1(\Om) \cap L^p(\Om)\big), \\
	u_k\vert_{\Sigma_T} &\to u_*\vert_{\Sigma_T} &&\quad \text{ weakly-* in } L^{\infty}(0,T;\Xk  \cap L^q(\Ga)) ,\\
	\mu_k &\to \mu_* &&\quad \text{ weakly in } L^2(0,T;H^1(\Omega)), \\
	\theta_k &\to \theta_* &&\quad \text{ weakly in } L^2(0,T;H^1(\Gamma)), \\
	\label{CONV:DIR}
	\beta \theta_k - \mu_k \vert_{\Sigma_T} &\to 0 &&\quad \text{ strongly in } L^2(\Sigma_T).
	\end{alignat}
	along a non-relabelled subsequence. In particular, \eqref{CONV:DIR} implies that $\mu_* \vert_{\Sigma_T} = \beta\theta_*$. 
	Using the Aubin--Lions lemma, we conclude that
	\begin{alignat}{2}
		u_k &\to u_* &&\quad \text{ strongly in } C([0,T];L^2(\Omega)), \\
		u_k\vert_{\Sigma_T} &\to u_*\vert_{\Sigma_T} &&\quad \text{ strongly in } C([0,T];L^2(\Gamma)).
	\end{alignat}
	\end{subequations}
	It is clear from the convergence properties in \eqref{GMS:lim:comp} that the triplet $(u_*, \mu_*, \theta_*)$ has the desired regularity as stated in item (i) of Proposition~\ref{prop:GMS}.  For arbitrary $w\in \V$ and  $\eta \in \Vk \cap L^\infty$ with $\eta\vert_\Ga \in L^\infty(\Ga)$, testing \eqref{WF:INT:1} with $\beta w$, \eqref{WF:INT:2} with $w$ and \eqref{WF:INT:3} with $\eta$ gives
		\begin{align*}
		\inn{\delt u_k}{ w}_{\V,\beta} 
		&= - \beta \intO \grad\mu_k \cdot \grad w \dx - \intG \gradg\theta_k \cdot \gradg w \dG,  \\
		\intO \mu_k \eta \dx + \intG \theta_k \eta \dG 
		 &= \intO \grad u_k\cdot\grad\eta + F'(u_k)\eta \dx  
		+ \intG \kappa \gradg u_k\cdot\gradg\eta + G'(u_k)\eta  \dG.	 
		\end{align*}
	After multiplying the above by arbitrary test functions in $C^\infty_c([0,T])$ and integrating with respect to $t$ from $0$ to $T$, we can apply the convergence properties \eqref{GMS:lim:comp} to pass to the limit in the resulting equations, leading to the assertion that $(u_*, \mu_*, \theta_*)$ satisfies \eqref{WF:GMS}. Again by weak lower semicontinuity arguments, passing to the limit in the energy inequality \eqref{WF:INT:DISS} leads to \eqref{WF:GMS:DISS}, and so $(u_*, \mu_*, \theta_*)$ is the unique weak solution of the GMS model \eqref{CH:GMS} in the sense of Proposition~\ref{prop:GMS}. 

	If additionally \eqref{ass:pot:reg} holds, then we obtain as before the weak convergence of $(u_k, u_k \vert_{\Sigma_T})$ to $(u_*, u_* \vert_{\Sigma_T})$ in $L^2(0,T;\H^3)$ if $\kappa > 0$ and in $ L^2(0,T;H^{5/2}(\Om) \times H^2(\Ga)) $ if $\kappa = 0$. 
		
	By uniqueness of solutions to the GMS model, which is independent of the choice of the extracted subsequence, we conclude by standard arguments that the above convergence results 
	hold true for the whole sequence.  Moreover, as the sequence $(L_k)_{k\in\N}$ was arbitrary, the convergence assertions for $L \to 0$ are established.
	\subparagraph{Step 2: Convergence rates.}
	For $L \in (0,1]$, let $(u^L, \mu^L, \theta^L)$ denote the unique solution to \eqref{CH:INT} corresponding to the initial data $u_{0,*}$ in the sense of Theorem~\ref{THM:WP}.   
	 We now use the notation 
	\begin{align*}
	(\hat u,\hat \mu,\hat \theta) := (u^L - u_*,  \mu^L - \mu_*,  \theta^L - \theta_* ),
	\end{align*} 
	Recalling that $\beta \theta_* - \mu_* = 0$ a.e. on $\Sigma_T$, 
	the convergence rate
	\begin{align*}
		\norm{\beta \theta^L - \mu^L }_{L^2(\Sigma_T)} = \norm{\beta \hat \theta - \hat \mu}_{L^2(\Sigma_T)} \leq C \sqrt{L}
	\end{align*} 
	follows directly from \eqref{lim:1}, with a constant $C$ independent of $L$ and $\kappa$.  Under \eqref{ass:pot:reg} and the assumption $(u_{0,*},u_{0,*}\vert_\Gamma) \in \H^3$ if $\kappa>0$ or $u_{0,*} \in H^3(\Om)$ if $\kappa=0$, the triplet $(u^L, \mu^L, \theta^L)$ is a strong solution to \eqref{CH:INT} in the sense of Theorem~\ref{THM:SWP}. 
	In light of Proposition~\ref{prop:GMS:reg} for the limit solutions $(u_*, \mu_*, \theta_*)$,  we see that
	\begin{subequations}\label{GMS:d}
		\begin{alignat}{2}
		\label{GMS:d:1} 0 & = \inn{\hat u_t}{w}_{H^1(\Omega)} + \intO \nabla \hat \mu \cdot \nabla w \dx - \intG \pdnu \hat \mu \, w \dG, \\
		\label{GMS:d:2} 0 & = \inn{\hat u_t}{z}_{H^1(\Gamma)} + \intG \gradg \hat \theta \cdot \gradg z + \pdnu \hat \mu \, \beta z \dG, \\
		\begin{split}
		\label{GMS:d:3} 0 & = \intO \nabla \hat u \cdot \nabla \eta + (F'(\Ll u) - F'(u_*) - \hat \mu) \eta \dx \\
		& \quad + \intG \kappa \gradg \hat u \cdot \gradg \eta + (G'(u^L) - G'(u_*) - \hat \theta) \eta \dG,
		\end{split}
		\end{alignat}
	\end{subequations}
	for almost all $t \in (0,T)$ and for all $w \in H^1(\Omega)$, $z \in H^1(\Gamma)$ and $\eta \in\Vk \cap L^\infty(\Omega)$ with $\eta\vert_\Ga \in L^\infty(\Ga)$. 
	The only difference to \eqref{LW:d} is that here $\pdnu \mu^L$ is replaced by $\pdnu \hat \mu$.  
	Let now $t_0 \in (0,T]$ be arbitrary. Once more, we write $Q_{t_0} = \Omega \times (0, t_0)$ and  $\Sigma_{t_0} = \Gamma \times (0,t_0)$ and we  use  the notation introduced in \eqref{DEF:*}. 
	Recalling that
	\begin{align}
	\label{GMS:rel:1} 
	&\pdnu \hat \mu = \pdnu \mu^L - \pdnu \mu_* 
		= \tfrac{1}{L} ( \beta\theta^L - \mu^L) - \pdnu \mu_* 
		= \tfrac{1}{L} (\beta\hat \theta - \hat \mu) - \pdnu \mu_*
		\quad\text{a.e.~on}\; \Sigma_T, 
	\end{align}
	and thus,
	\begin{align}
	\label{GMS:rel:2}  
	&\intG \big(1 \star \pdnu \hat \mu \big)(t)\, \big(\beta\hat \theta - \hat \mu\big)(t) \dG 
		= \intG \big(1 \star \big[\tfrac 1 L (\beta \hat \theta - \hat \mu) - \pdnu \mu_*\big]\big)(t) \,
		\big(\beta\hat \theta - \hat \mu\big)(t) \dG
	\end{align}
	for almost all $t\in(0,T)$. Invoking also the relation \eqref{REL:*} we can proceed as in the derivation of \eqref{IEQ:CR5} to conclude that
	\begin{align}
	\label{GMS:Err:1}
	\begin{aligned}
	&  \frac 1 2 \norm{ \nabla (1\star \hat \mu)(t_0)}_{L^2(\Omega)}^2 
	+ \frac 1 2 \norm{ \gradg (1\star \hat \theta)(t_0)}_{L^2(\Gamma)}^2 
	+ \frac{1}{2L} \norm{(1 \star [\beta \hat \theta - \hat \mu])(t_0)}_{L^2(\Gamma)}^2  \\
	& \qquad +\int_0^{t_0}  \norm{\nabla \hat u}_{L^2(\Omega)}^2 + \kappa \norm{\gradg \hat u}_{L^2(\Gamma)}^2 \dt \\
	& \quad \leq C \int_0^{t_0}  \norm{ \nabla (1\star \hat \mu)(t) }_{L^2(\Omega)}^2 \dt
		+ \int_0^{t_0} \intG \big(1\star \deln \mu_*\big)(t) \, \big( \beta\hat\theta(t)-\hat\mu(t) \big) \dG \dt 
	\end{aligned}
	\end{align}
	By Fubini's theorem, the Cauchy--Schwarz inequality and Young's inequality, we see that 
	\begin{align*}
	\begin{aligned}
		&\int_0^{t_0} \intG (1 \star \pdnu \mu_*)(t) (\beta \hat \theta(t) - \hat \mu(t)) \dG \dt = \intG \int_0^{t_0} \int_s^{t_0} \pdnu \mu_*(s) (\beta \hat \theta(t) - \hat \mu(t)) \dt \ds  \dG \\
		& \quad = \int_0^{t_0} \intG  \pdnu \mu_*(s)\, \Big[ \big(1 \star (\beta \hat\theta - \hat \mu)\big)(t_0) - \big(1 \star (\beta \hat\theta - \hat \mu)\big)( s ) \Big] \dG \ds \\
		& \quad \leq \int_0^{t_0} \norm{ \pdnu \mu_*(s) }_{L^2(\Ga)} \, \Big[ \norm{ 1 \star ( \beta \hat \theta - \hat \mu)(t_0) }_{L^2(\Ga)} 
			+ \norm{ 1 \star ( \beta \hat \theta - \hat \mu)(s) }_{L^2(\Ga)} \Big] \ds \\
		& \quad \leq \int_0^{t_0} CL \norm{\pdnu \mu_*(s)}_{L^2(\Gamma)}^2 
			+ \frac{1}{4L} \norm{1\star (\beta \hat \theta - \hat \mu)(s)}_{L^2(\Gamma)}^2 \ds \;
			+ \frac{1}{4L} \norm{1\star (\beta \hat \theta - \hat \mu)(t_0)}_{L^2(\Gamma)}^2 .
	\end{aligned}
	\end{align*} 
	Plugging this estimate into \eqref{GMS:Err:1} we arrive at
	\begin{align*}
	&  \frac 1 2 \norm{ \nabla (1\star \hat \mu)(t_0)}_{L^2(\Omega)}^2 
		+ \frac 1 2 \norm{ \gradg (1\star \hat \theta)(t_0)}_{L^2(\Gamma)}^2 
		+ \frac{1}{4L} \norm{(1 \star [\beta \hat \theta - \hat \mu])(t_0)}_{L^2(\Gamma)}^2  \\
	& \qquad +\int_0^{t_0}  \norm{\nabla \hat u(t)}_{L^2(\Omega)}^2 + \kappa \norm{\gradg \hat u(t)}_{L^2(\Gamma)}^2 \dt \\
	& \quad \leq CL \norm{\pdnu \mu_*}_{L^2(\Sigma_T)}^2 
		+ C \int_0^{t_0}  \Big ( \norm{ \nabla (1\star \hat \mu)(t)}_{L^2(\Omega)}^2 + \frac{1}{4L} \norm{(1 \star [\beta\hat \theta - \hat \mu])(t)}_{L^2(\Gamma)}^2 \Big ) \dt.
	\end{align*}
	Invoking the integral form of Gronwall's inequality, we deduce the existence of a constant $C$ independent of $L$ and $\kappa$ such that 
	\begin{align*}
	\begin{aligned}
	& \sup_{t \in (0,T)} \Big ( \norm{ \nabla (1\star \hat \mu)(t)}_{L^2(\Omega)}^2 + \norm{ \gradg (1\star \hat \theta)(t)}_{L^2(\Gamma)}^2 + \frac{1}{L}\norm{(1 \star [\beta \hat \theta - \hat \mu])(t)}_{L^2(\Gamma)}^2 \Big )  \\
	& \qquad +  \norm{\nabla \hat u}_{L^2(Q_T)}^2 + \kappa \norm{\gradg \hat u}_{L^2(\Sigma_T)}^2  \\
	& \quad \leq CL 
	\end{aligned}
	\end{align*}
	which implies the convergence rates 
	\begin{align}\label{GMS:CR:1}
	\norm{\nabla \hat u}_{L^2(Q_T)} +  \sqrt{\kappa}\, \norm{\gradg \hat u}_{L^2(\Sigma_T)} \leq C\sqrt{L}, \quad \sup_{t \in (0,T)}  \left \| \int_0^t ( \beta \hat \theta - \hat \mu)(s) \ds \right \|_{L^2(\Gamma)} \leq  CL.
	\end{align}
	Now, assuming $\kappa > 0$ and choose $w = \beta\hat u$ in \eqref{GMS:d:1} and $z = \hat u$ in \eqref{GMS:d:2}, so that upon summing and integrating in time over $(0,t)$, we obtain
	\begin{align*}
		\beta\norm{\hat u(t)}_{L^2(\Omega)}^2 + \norm{\hat u(t)}_{L^2(\Gamma)}^2 
		& \leq \beta \norm{\nabla \hat \mu}_{L^2(Q_T)} \norm{\nabla \hat u}_{L^2(Q_T)} 
			+ \norm{\gradg \hat \theta}_{L^2(\Sigma_T)} \norm{\gradg \hat u}_{L^2(\Sigma_T)} \\
		& \leq C \sqrt{L}
	\end{align*}
	after invoking \eqref{GMS:CR:1} and the uniform boundedness of $\nabla \mu^L$ and $\gradg \theta^L$ due to \eqref{lim:1}. As $\beta>0$, this leads to the convergence rate
	\begin{align*}
	\norm{\hat u}_{L^\infty(0,T;L^2(\Om))}  +
	\norm{\hat u}_{L^\infty(0,T;L^2(\Ga))}  \leq CL^{1/4},
	\end{align*}
	and thus, the proof of Theorem~\ref{THM:ASY:GMS} is complete.
\end{proof}

\section{Numerical analysis}
 In this section, we assume that $\beta,\kappa>0$. 
 We derive an unconditionally stable, fully discrete finite element scheme which allows us to investigate  the model \eqref{CH:INT} as well as the limit models \eqref{CH:GMS} and \eqref{CH:LW}  numerically. We establish the existence of discrete solutions to this scheme and prove convergence for arbitrary $L\in [0,\infty]$ in the limit of vanishing spatial and temporal discretisation parameters. For simplicity, we also set $\mo = \mg = \eps=\delta=1$ in the subsequent approach, although different values are used for the simulations in Section~6.
 
 As the model \eqref{CH:INT} interpolates between the GMS model and the LW model, it naturally inherits the peculiarities of both.
Therefore, a discrete scheme that can be applied to the complete family of models needs to cope with the intricacies of both approaches.
In particular, the scheme has to include both chemical potentials $\muO$ and $\muG$, while ensuring $\beta\muG=\mu\vert_{\Sigma_T}$ for $L\searrow0$ and $\del_\n \muO=0$ for $L\nearrow\infty$.
Furthermore, extending the ideas from \cite{Metzger2019} to derive explicit, $u$-dependent expressions for the chemical potentials is deemed to be necessary to prevent the discrete scheme from becoming ill-conditioned for small time increments (see Section~5 in \cite{Metzger2019}).

\subsection{Technical preliminaries}
Concerning the discretisation in time, we consider
\begin{enumerate}[label=$(\mathrm{T})$, ref = $\mathrm{T}$]
\labitem{$(\mathrm{T})$}{item:disc:time} the time interval $I:=[0,T)$ that is subdivided into intervals $I_n:=[t_n,t_{n+1})$ with $t_0=0$ and $t_{n+1}=t_n+\tau_n$ for time increments $\tau_n>0$ and $n=0,...,N-1$ with $t_N=T$. For simplicity, we take $\tau_n\equiv\tau=\tfrac{T}{N}$ for $n=0,...,N-1$.
\end{enumerate}

 Throughout this section we assume the spatial domain $\Omega\subset\R^{d}$, $d\in\tgkla{2,3}$ to be bounded, convex, and polygonal (if $d = 2$) or polyhedral (if $d = 3$) to avoid additional technicalities.
When considering a smoother domain, one has to approximate $\Omega$ and $\Gamma$ by an $h$-dependent family of polygonal domains $\tgkla{\Omega\h}\h$ with boundaries $\tgkla{\Gamma\h}\h$ (cf. \cite{Dziuk1988,DziukElliot2013,ElliottRanner2012}).
This approach of course introduces an additional geometric error which also has to be considered.
For the application of this technique to the GMS model, we refer the reader to \cite{Harder2020}, where an error estimate for a semi-discrete finite element scheme for \eqref{CH:GMS} was derived. 

We introduce partitions $\Th$ of $\Omega$ and $\Th^\Gamma$ of $\Gamma=\partial\Omega$ depending on a spatial discretisation parameter $h > 0$ satisfying the following assumptions:
\begin{enumerate}[label=$(\mathrm{S \arabic*})$, ref = $\mathrm{S \arabic*}$]
\labitem{$(\mathrm{S1})$}{item:disc:space} Let $\tgkla{\Th}_{h>0}$ a quasiuniform family (in the sense of \cite{BrennerScott}) of partitions of $\Omega$ into disjoint, open, non-obtuse simplices $K$, so that
\begin{align*}
\overline{\Omega}\equiv\bigcup_{K\in\Th}\overline{K} \quad \text{ with }\max_{K\in\Th}\diam\trkla{K}\leq h\,.
\end{align*}
\labitem{$(\mathrm{S2})$}{item:disc:gamma} Let $\tgkla{\Th^\Gamma}_{h>0}$ a quasiuniform family of partitions of $\Gamma$ into disjoint, open, non-obtuse simplices $K^\Gamma$, so that
\begin{align*}
\forall K^\Gamma\in\Th^\Gamma~~\exists!K\in\Th\text{~such~that~} \overline{K^\Gamma}=\overline{K}\cap\Gamma\,,
\end{align*}
and
\begin{align*}
\Gamma\equiv\bigcup_{K^\Gamma\in\Th^\Gamma}\overline{K^\Gamma} \quad \text{ with }\max_{K^\Gamma\in\Th^\Gamma}\diam\trkla{K^\Gamma}\leq h\,.
\end{align*}
\end{enumerate}
The above assumption implies that $\Th^\Gamma$ is compatible to $\Th$ in the sense that all elements in $\Th^\Gamma$ are edges (if $d = 2$) or faces (if $d = 3$) of elements in $\Th$.
For the approximation of the phase-field $u$ and the chemical potential $\muO$ we use continuous, piecewise linear finite element functions on $\Th$.
This space, denoted by $\UhO$, is spanned by basis functions $\tgkla{\chi_{h,k}}\Sub[-2pt]{k=1, \dots, \dim\UhO}$ that also form a dual basis to the vertices 
$\tgkla{\bs{x}_k}\Sub[-2pt]{k=1, \dots, \dim\UhO}$ of $\Th$, i.e., $\chi_{h,k}\trkla{\bs{x}_k}=\delta_{k,l}$ for $k,l=1,...,\dim\UhO$.

Analogously, we denote the space of continuous, piecewise linear finite element functions on $\Th^\Gamma$ by $\UhG$, which is spanned by basis functions $\tgkla {\chi_{h,k}^\Gamma}\Sub[-2pt]{k=1,...,\dim\UhG}$ that also form a dual basis to the vertices $\tgkla{\bs{x}_k^\Gamma}\Sub[-2pt]{k=1,...,\dim\UhG}$ of $\Th^\Gamma$, 
i.e., $\chi^\Gamma_{h,k}\trkla{\bs{x}^\Gamma_k}=\delta_{k,l}$ for $k,l=1,...,\dim\UhG$.  
Due to the compatibility condition for $\Th$ and $\Th^\Gamma$, we have 
\begin{align}\label{eq:compatibility}
\UhG=\operatorname{span}\tgkla{\zeta\h \vert_{\Gamma} \,:\,\zeta\h\in\UhO}\,.
\end{align}
Without loss of generality, we assume that the first $\dim\UhG$ vertices of $\Th$ are located on $\Gamma$, 
i.e., $\tgkla{\bs{x}_k^\Gamma}\Sub[-2pt]{k=1,...,\dim\UhG}=\tgkla{\bs{x}_k}\Sub[-2pt]{k=1,...,\dim\UhG}$.
As all functions in $\UhO$ are continuous in $\overline{\Omega}$, we will often suppress the trace operator $\cdot\vert_\Gamma$ to simplify the notation.
We define the nodal interpolation operators $\Ihop\,:\,C^0\trkla{\overline{\Omega}}\rightarrow \UhO$ and $\IhGop\,:\,C^0\trkla{\Gamma}\rightarrow \UhG$ by
\begin{align}\label{eq:def:Ih}
\Ih{a}:=\sum_{k=1}^{\dim\UhO}a\trkla{\bs{x}_k}\chi_{h,k}\,, \quad \text{ and } \quad
\IhG{a}:=\sum_{k=1}^{\dim\UhG}a\trkla{\bs{x}_k}\chi^\Gamma_{h,k}\,.
\end{align}
 It is well-known that on the finite element spaces $\UhO$ and $\UhG$ the discrete $L^2$-norms given by $\big(\iOmega\Ih{\tabs{\,\cdot\,}^2}dx\big)^{1/2}$  and $\big(\iGamma\IhG{\tabs{\,\cdot\,}^2}d\Gamma\big)^{1/2}$  are equivalent to $\norm{\cdot}_{L^2\trkla{\Omega}}$  and $\norm{\cdot}_{L^2\trkla{\Gamma}}$, respectively.
Furthermore,  the following estimates (that can be found in \cite[Lem.~2.1]{Metzger2019}) hold true:
\begin{lemma}\label{lem:ihfe}
Let $\Th$ and $\Th^\Gamma$ satisfy \eqref{item:disc:space} and \eqref{item:disc:gamma}.
Furthermore, let $p\in[1,\infty)$, $1\leq q\leq\infty$, and $ q^*=\tfrac{q}{q-1} $ for $q<\infty$ or $q^*=1$ for $q=\infty$.
Then,
\begin{align}
\norm{\trkla{\ids-\Ihop}\tgkla{f\h g\h}}_{L^p\trkla{\Omega}}&\leq C h^2\norm{\nabla f\h}_{L^{pq}\trkla{\Omega}}\norm{\nabla g\h}_{L^{pq^*}\trkla{\Omega}}\,,\\
\norm{\trkla{\ids-\IhGop}\tgkla{\tilde{f}\h \tilde{g}\h}}_{L^p\trkla{\Gamma}}&\leq C h^2\norm{\gradg \tilde{f}\h}_{L^{pq}\trkla{\Gamma}}\norm{\gradg \tilde{g}\h}_{L^{pq^*}\trkla{\Gamma}}\,.
\end{align}
holds true for all $f\h,\,g\h\in\UhO$ and $\tilde{f}\h,\,\tilde{g}\h\in\UhG$.
\end{lemma}
In the forthcoming analysis,  we consider any initial datum $(u_0,u_0\vert_\Ga) \in \H^2$, and potentials $F$ and $G$  satisfying \eqref{ass:pot} with $p=q=4$. In addition we make the following assumption: 
\begin{enumerate}[label=$(\mathrm{D})$, ref = $\mathrm{D}$]
\labitem{$(\mathrm{D})$}{item:potentialsbounds} The convex parts $F_1$ and $G_1$ as well as the concave parts $F_2$ and $G_2$ can be further decomposed into a polynomial part of degree four and an additional part having a globally Lipschitz continuous first derivative. 
\end{enumerate}
In particular, we may thus choose the penalised double-well potential 
\begin{align}
W\trkla{s}:=\tfrac14\trkla{1-s^2}^2+\tfrac1{\delta^\prime}\max\tgkla{\tabs{s}-1,0}^2\label{eq:penalisedW}
\end{align}
with $0<\delta^\prime<\!\!<1$.  
The assumption $(u_0,u_0\vert_\Ga) \in \H^2$ allows us to define the discrete initial condition $u\h^0 \in \UhO$ via $u\h^0:=\Ih{u_0}$.  An immediate consequence is
\begin{subequations}
\begin{align}
\iOmega\tabs{\nabla u\h^0}^2 \dx + \iOmega\Ih{F\trkla{u\h^0}} \dx + \iGamma\tabs{\gradg u\h^0}^2  \dG +\iGamma\IhG{G\trkla{u\h^0}} \dG \leq C\trkla{u_0}\,,\label{eq:initialcond:bound}\\
\norm{u\h^0-u_0}_{H^1\trkla{\Omega}} +\norm{u\h^0-u_0 \vert_{\Gamma}}_{H^1\trkla{\Gamma}}\rightarrow 0\quad\text{for~}h\searrow0\label{eq:initialcond:conv}\,.
\end{align}
\end{subequations}
When passing to the limit $\trkla{h,\tau}\searrow\trkla{0,0}$, we will also need a compatibility condition for $h$ and $\tau$. 
In particular, we will assume that
\begin{enumerate}[label=$(\mathrm{C})$, ref = $\mathrm{C}$]
\labitem{$(\mathrm{C})$}{item:htau} $\tfrac{h^4}{\tau}\searrow0$ when passing to the limit $\trkla{h,\tau}\searrow\trkla{0,0}$.
\end{enumerate}
Furthermore, we introduce the matrices
\begin{subequations}
\begin{align}
\trkla{\MOmega}_{ij}&:=\iOmega\Ih{\chi_{hj}\chi_{hi}} \dx &&\forall i,j=1,...,\dim\UhO\,,\\
\trkla{\MGamma}_{ij}&:=\iGamma\IhG{\chi_{hj}^\Gamma\chi_{hi}^\Gamma} \dG&&\forall i,j=1,...,\dim\UhG\,,\\
\trkla{\LOmega}_{ij}&:=\iOmega\nabla\chi_{hj}\cdot\nabla\chi_{hi} \dx&&\forall i,j=1,...,\dim\UhO\,,\\
\trkla{\LGamma}_{ij}&:=\iGamma\gradg\chi_{hj}^\Gamma\cdot\gradg\chi_{hi}^\Gamma \dG&&\forall i,j=1,...,\dim\UhG,
\end{align}
\end{subequations}
and with a slight misuse of notation, we write $f\trkla{U\nn}$ when we apply a function $f$ to all components of $U\nn$.
Due to our consideration that the first $\dim\UhG$ vertices of $\Th$ are located on $\Gamma$, we can define an extension operator $\textend{\cdot}{\OmegaSymb}\,:\,\R^{\dim\UhG}\rightarrow\R^{\dim\UhO}$ via
\begin{align*}
\R^{\dim\UhG}\ni A\mapsto \begin{pmatrix} A\\0\end{pmatrix} \in\R^{\dim\UhO}
\end{align*}
and the restriction operator $\trestr{\cdot}{\GammaSymb}\,:\,\R^{\dim\UhO}\rightarrow\R^{\dim\UhG}$, which restricts a vector to its first $\dim\UhG$ entries.  For matrices, we define analogous restriction operators by splitting a matrix $\mb{A}\in\R^{\dim\UhO\times\dim\UhO}$ into submatrices
\begin{align}\label{eq:def:restriction}
\begin{matrix*}[l]
&\restr{\mb{A}}{\GammaSymb\times\GammaSymb}\in\R^{\dim\UhG\times\dim\UhG}\,,& \restr{\mb{A}}{\GammaSymb\times\InnerSymb}\in\R^{\dim\UhG\times\trkla{\dim\UhO-\dim\UhG}}\,,\\
& \restr{\mb{A}}{\InnerSymb\times\GammaSymb}\in\R^{\trkla{\dim\UhO-\dim\UhG}\times\dim\UhG}\,,&\restr{\mb{A}}{\InnerSymb\times\InnerSymb}\in\R^{\trkla{\dim\UhO-\dim\UhG}\times\trkla{\dim\UhO-\dim\UhG}}\,,\\
& \restr{\mb{A}}{\GammaSymb\times\OmegaSymb}\in\R^{\dim\UhG\times\dim\UhO}\,, &\restr{\mb{A}}{\InnerSymb\times\OmegaSymb}\in\R^{\trkla{\dim\UhO-\dim\UhG}\times\dim\UhO}\,,\\
&\restr{\mb{A}}{\OmegaSymb\times\GammaSymb}\in\R^{\dim\UhO\times\dim\UhG}\,,
&\restr{\mb{A}}{\OmegaSymb\times\InnerSymb}\in\R^{\dim\UhO\times\trkla{\dim\UhO-\dim\UhG}}\,,
\end{matrix*}
\end{align}
such that
\begin{align}
\mb{A}=\begin{pmatrix}
\restr{\mb{A}}{\GammaSymb\times\GammaSymb} & \restr{\mb{A}}{\GammaSymb\times\InnerSymb}\\
\restr{\mb{A}}{\InnerSymb\times\GammaSymb} & \restr{\mb{A}}{\InnerSymb\times\InnerSymb}
\end{pmatrix} = \begin{pmatrix}
\restr{\mb{A}}{\GammaSymb\times\OmegaSymb}\\\restr{\mb{A}}{\InnerSymb\times\OmegaSymb}
\end{pmatrix} = \begin{pmatrix}
\restr{\mb{A}}{\OmegaSymb\times\GammaSymb}&\restr{\mb{A}}{\OmegaSymb\times\InnerSymb}
\end{pmatrix}\,.
\end{align}
In the above, we employed the notation $\InnerSymb$ to denote the collection of degrees of freedoms corresponding to the interior nodal points of $\Omega$.  We also define an extension operator $\textend{\cdot}{\OmegaSymb\times\OmegaSymb}\,:\,\R^{\dim\UhG\times\dim\UhG}\rightarrow \R^{\dim\UhO\times\dim\UhO}$ via
\begin{align*}
\R^{\dim\UhG\times\dim\UhG}\ni\mb{A}\mapsto\begin{pmatrix}
\mb{A}&\mb{0}\\\mb{0}&\mb{0}
\end{pmatrix}\in\R^{\dim\UhO\times\dim\UhO}\,.
\end{align*}
\subsection{Derivation of the numerical scheme}

For $L\in(0,\infty)$, a finite element discretisation of the model \eqref{CH:INT} reads as
\begin{subequations}\label{eq:fe}
\begin{align}
\iOmega\Ih{\dtaum u\h\nn w\h} + \nabla\muOh\nn&\cdot\nabla w\h \dx -L^{-1}\iGamma\IhG{\rkla{\beta\muGh\nn-\muOh\nn}w\h} \dG = 0\,,\label{eq:fe:1}\\
\iGamma\IhG{\dtaum u\h\nn z\h} + \gradg\muGh\nn&\cdot\gradg z\h +\beta L^{-1} \IhG{\rkla{\beta\muGh\nn-\muOh\nn}z\h} \dG = 0\,,\label{eq:fe:2}\\
\iOmega\Ih{\muOh\nn\eta\h} \dx +\iGamma\IhG{\muGh\nn\eta\h} \dG &=\iOmega\nabla u\h\nn\cdot\nabla\eta\h \dx  +\kappa\iGamma\gradg u\h\nn\cdot\gradg\eta\h \dG \label{eq:fe:3} \\
\notag & \quad + \iOmega\Ih{\rkla{F_1^\prime\trkla{u\h\nn}+F^\prime_2 \trkla{u\h\no}}\eta\h} \dx \\
\notag &\quad + \iGamma\IhG{\rkla{G_1^\prime\trkla{u\h\nn}+G_2^\prime\trkla{u\h\no}}\eta\h} \dG \,,
\end{align}
\end{subequations}
holding for all $w\h\in\UhO$, $z\h\in\UhG$, and $\eta\h\in\UhO$.  In the above we have used the backward difference quotient $\dtaum a\nn:=\tau^{-1}\trkla{a\nn-a\no}$, and so, for given $u\h\no\in\UhO$, we search for $u\h\nn,~\muOh\nn\in\UhO$ and $\muGh\nn\in\UhG$ satisfying \eqref{eq:fe}. 
 Unlike the scheme used in Section \ref{section:wellposedness}, \eqref{eq:fe} is based on a convex-concave decomposition of the nonlinear functions $F$ and $G$, as this approach allows for an unconditionally stable discrete scheme (cf.~Lemma \ref{lem:matrix:energy}). 

Using the matrix notation introduced in the previous section, and collecting the nodal values of $u\h\nn$, $u\h\no$, $\muOh\nn$, and $\muGh\nn$ into the vectors $U\nn$, $U\no$, $\muOvec\nn$, and $\muGvec\nn$, we can express \eqref{eq:fe} equivalently as
\begin{subequations}
\begin{align}
&\MOmega U\nn-\MOmega U\no+\tau\LOmega\muOvec\nn-\tau L^{-1}\extend{\ekla{\MGamma\rkla{\beta\muGvec\nn-\restr{\muOvec\nn}{\GammaSymb}}}}{\OmegaSymb}
\;\,=0\,,\label{eq:matrix:phi:bulk}\\
&\MGamma \restr{U\nn}{\GammaSymb}-\MGamma\restr{U\no}{\GammaSymb} +\tau\LGamma\muGvec\nn+\tau\beta L^{-1}\MGamma\rkla{\beta\muGvec\nn-\restr{\muOvec\nn}{\GammaSymb}}
=0\,,\label{eq:matrix:phi:boundary}\\
&
\MOmega\muOvec\nn+\extend{\ekla{\MGamma\muGvec\nn}}{\OmegaSymb}  = \LOmega U\nn+\kappa\extend{\ekla{\LGamma \restr{U\nn}{\GammaSymb}}}{\OmegaSymb} + \MOmega\rkla{F_1^\prime\trkla{U\nn}+F_2^\prime\trkla{U\no}} \notag\\
&\hspace{120pt} +\extend{\ekla{\MGamma\rkla{G_1^\prime\trkla{\restr{U\nn}{\GammaSymb}}+G_2^\prime\trkla{\restr{U\no}{\GammaSymb}}}}}{\OmegaSymb}\,.\label{eq:matrix:potential}
\end{align}
\end{subequations}
Restricting \eqref{eq:matrix:phi:bulk} to the boundary and comparing with \eqref{eq:matrix:phi:boundary} leads to the compatibility condition
\begin{align}\label{eq:compatibility:1}
\begin{aligned}
&\tfrac{L}{L+1}\restr{\ekla{\MOmega^{-1}\LOmega\muOvec\nn}}{\GammaSymb}-\tfrac{1}{L+1}\restr{\MOmega^{-1}}{\GammaSymb\times\GammaSymb}\MGamma\rkla{\beta\muGvec\nn-\restr{\muOvec\nn}{\GammaSymb}}\\
&\quad =\tfrac{L}{L+1}\MGamma^{-1}\LGamma\muGvec\nn+\beta\tfrac{1}{L+1}\rkla{\beta\muGvec\nn-\restr{\muOvec\nn}{\GammaSymb}}.
\end{aligned}
\end{align}
Upon rearranging and recalling that $\MOmega$ is a diagonal matrix, \eqref{eq:compatibility:1} can be written as
\begin{align}\label{eq:compatibility:2}
&\underbrace{\rkla{\tfrac{L}{L+1}\restr{\MOmega^{-1}}{\GammaSymb\times\GammaSymb}\restr{\LOmega}{\GammaSymb\times\GammaSymb} +\tfrac{1}{L+1}\restr{\MOmega^{-1}}{\GammaSymb\times\GammaSymb}\MGamma+\beta\tfrac{1}{L+1}\eye}}_{=:\mb{A}}\restr{\muOvec\nn}{\GammaSymb} \\ \notag
&\;\;+ \underbrace{\rkla{\tfrac{L}{L+1}\restr{\MOmega^{-1}}{\GammaSymb\times\GammaSymb}\restr{\LOmega}{\GammaSymb\times\InnerSymb}}}_{=:\mb{B}}\restr{\muOvec\nn}{\InnerSymb}
= \underbrace{\rkla{\tfrac{L}{L+1}\MGamma^{-1}\LGamma +\beta\tfrac{1}{L+1}\restr{\MOmega^{-1}}{\GammaSymb\times\GammaSymb}\MGamma + \beta^2\tfrac{1}{L+1}\eye}}_{=:\mb{C}}\muGvec\nn\,.
\end{align}
Combining \eqref{eq:matrix:potential} with \eqref{eq:compatibility:2}, we are able to determine $\muOvec\nn$ and $\muGvec\nn$ for given $U\nn$ and $U\no$ by solving the linear system
\begin{align}\label{eq:mu:system}
\begin{pmatrix}
\restr{\MOmega}{\GammaSymb\times\GammaSymb} &\mb{0}&\MGamma\\
\mb{0}&\restr{\MOmega}{\InnerSymb\times\InnerSymb}&\mb{0}\\
\mb{A}&\mb{B}&-\mb{C}
\end{pmatrix}\begin{pmatrix}
\restr{\muOvec\nn}{\GammaSymb}\\\restr{\muOvec\nn}{\InnerSymb}\\\muGvec\nn
\end{pmatrix} = \begin{pmatrix}
R_{\GammaSymb}\trkla{U\nn} \phantom{\big|}\\ \RInn\trkla{U\nn} \phantom{\Big|} \\ \mb{0}
\end{pmatrix}
\end{align}
with 
\begin{align*}
\begin{split}
R_{\GammaSymb}\trkla{U\nn}& :=\restr{\LOmega}{\GammaSymb\times\OmegaSymb}U\nn+\restr{\MOmega}{\GammaSymb\times\OmegaSymb}\rkla{F_1^\prime\trkla{U\nn}+F_2^\prime\trkla{U\no}}\\
&\quad +\kappa\LGamma\restr{U\nn}{\GammaSymb}+\MGamma\rkla{G_1^\prime\trkla{\restr{U\nn}{\GammaSymb}}+G_2^\prime\trkla{\restr{U\no}{\GammaSymb}}}, 
\end{split}
\\ 
\RInn\trkla{U\nn} & := \restr{\LOmega}{\InnerSymb\times\OmegaSymb}U\nn+\restr{\MOmega}{\InnerSymb\times\OmegaSymb}\rkla{F_1^\prime\trkla{U\nn}+F_2^\prime\trkla{U\no}}\,,
\end{align*}
where we have suppressed the dependence of $R_{\GammaSymb}$ and $\RInn$ on $U\no$, as $U\no$ is known from the last time step.  Solving \eqref{eq:mu:system} for $\trestr{\muOvec\nn}{\GammaSymb}$, $\muOvec\nn\vert\SubInnerSymbH$, and $\muGvec\nn$ gives the equations
\begin{subequations}
\begin{align}
\muGvec\nn=&\,-\MGamma^{-1}\restr{\MOmega}{\GammaSymb\times\GammaSymb}\restr{\muOvec\nn}{\GammaSymb}
+\MGamma^{-1}R_{\GammaSymb}\trkla{U\nn}\,,\label{eq:potential:muGvec}\\
\restr{\muOvec\nn}{\InnerSymb}=&\,\restr{\MOmega}{\InnerSymb\times\InnerSymb}^{-1}\RInn\trkla{U\nn}\,,\label{eq:potential:muOInner}\\
\mb{A}\restr{\muOvec\nn}{\GammaSymb}=&\,-\mb{B}\restr{\muOvec\nn}{\InnerSymb}+\mb{C}\muGvec\nn\,.\label{eq:potential:muOG}
\end{align}
\end{subequations}
Plugging \eqref{eq:potential:muGvec} and \eqref{eq:potential:muOInner} into \eqref{eq:potential:muOG} and multiplying by $\trestr{\MOmega}{\GammaSymb\times\GammaSymb}$, we obtain
\begin{align}\label{eq:expression:muOG}
\mb{N}\restr{\muOvec\nn}{\GammaSymb}=-\restr{\MOmega}{\GammaSymb\times\GammaSymb}\mb{B}\restr{\MOmega}{\InnerSymb\times\InnerSymb}^{-1}\RInn\trkla{U\nn}+\restr{\MOmega}{\GammaSymb\times\GammaSymb}\mb{C}\MGamma^{-1} R_{\GammaSymb}\trkla{U\nn}
\end{align}
with
\begin{align}\label{eq:definition:N}
\begin{split}
\mb{N}:=&\,\restr{\MOmega}{\GammaSymb\times\GammaSymb}\mb{A}+\restr{\MOmega}{\GammaSymb\times\GammaSymb}\mb{C}\MGamma^{-1}\restr{\MOmega}{\GammaSymb\times\GammaSymb}\\
=&\,\tfrac{L}{L+1}\rkla{\restr{\LOmega}{\GammaSymb\times\GammaSymb}+\restr{\MOmega}{\GammaSymb\times\GammaSymb}\MGamma^{-1}\LGamma\MGamma^{-1}\restr{\MOmega}{\GammaSymb\times\GammaSymb}}\\
&\,+\tfrac{1}{L+1}\rkla{\MGamma+2\beta\restr{\MOmega}{\GammaSymb\times\GammaSymb}+\beta^2\restr{\MOmega}{\GammaSymb\times\GammaSymb}\MGamma^{-1}\restr{\MOmega}{\GammaSymb\times\GammaSymb}}\,.
\end{split}
\end{align}
By following along similar lines of argument in \cite[Lem.~2.4]{Metzger2019}, the matrix $\mb{N}$ is symmetric and positive definite.
Therefore, \eqref{eq:expression:muOG}, \eqref{eq:potential:muOInner}, and \eqref{eq:potential:muGvec} provide explicit, $U\nn$-dependent expressions for $\muOvec\nn$ and $\muGvec\nn$.  Multiplying \eqref{eq:matrix:phi:boundary} with $\beta^{-1}$ and adding to \eqref{eq:matrix:phi:bulk}, we obtain using \eqref{eq:potential:muGvec} the discrete scheme
\begin{align}\label{eq:discscheme}
\begin{aligned}
&\rkla{\MOmega+\beta^{-1}\extend{\MGamma}{\OmegaSymb\times\OmegaSymb}}\rkla{U\nn-U\no}+\tau\LOmega\muOvec\nn\\
&\qquad +\tau\beta^{-1}\extend{\ekla{\LGamma\MGamma^{-1}R_{\GammaSymb}\trkla{U\nn}-\LGamma\MGamma^{-1}\restr{\MOmega}{\GammaSymb\times\GammaSymb}\restr{\muOvec\nn}{\GammaSymb}}}{\OmegaSymb}=0
\end{aligned}
\end{align}
with $\muOvec\nn$ given by \eqref{eq:expression:muOG} and \eqref{eq:potential:muOInner}.
Since the parameter $L$ only appears in the numerical scheme as prefactors $\tfrac1{L+1}$ and $\tfrac{L}{L+1}$, the proposed scheme is also well-defined for $L=0$ and the formal limit $L = \infty$, whereby in the latter we set $\tfrac{1}{L+1} = 0$ and $\tfrac{L}{L+1} = 1$.  In the following, we will analyze \eqref{eq:discscheme} and show that we indeed recover discretisations of \eqref{CH:GMS} and \eqref{CH:LW} for $L\searrow0$ and $L\nearrow\infty$, respectively.

As the compatibility condition \eqref{eq:compatibility:1} will be a crucial ingredient for the analysis of the proposed scheme, we will verify that our expressions for $\muOvec\nn$ and $\muGvec\nn$ satisfy \eqref{eq:compatibility:1}.
From \eqref{eq:potential:muGvec} and \eqref{eq:potential:muOInner}, we obtain
\begin{align*}
&\tfrac{L}{L+1}\restr{\ekla{\MOmega^{-1}\LOmega\muOvec\nn}}{\GammaSymb}-\tfrac{L}{L+1}\MGamma^{-1}\LGamma\muGvec\nn \\
&\quad -\tfrac{1}{L+1}\restr{\MOmega^{-1}}{\GammaSymb\times\GammaSymb}\MGamma\rkla{\beta\muGvec\nn - \restr{\muOvec\nn}{\GammaSymb}}
	-\beta \tfrac{1}{L+1}\rkla{\beta\muGvec\nn - \restr{\muOvec\nn}{\GammaSymb}}\\[1ex]
&=\tfrac{L}{L+1}\restr{\MOmega^{-1}}{\GammaSymb\times\GammaSymb}\restr{\LOmega}{\GammaSymb\times\GammaSymb}\restr{\muOvec\nn}{\GammaSymb}+\tfrac{L}{L+1}\restr{\MOmega^{-1}}{\GammaSymb\times\GammaSymb}\restr{\LOmega}{\GammaSymb\times\InnerSymb}\restr{\MOmega}{\InnerSymb\times\InnerSymb}^{-1} \RInn\trkla{U\nn}\\
&\quad+\tfrac{L}{L+1}\MGamma^{-1}\LGamma\MGamma^{-1}\restr{\MOmega}{\GammaSymb\times\GammaSymb}\restr{\muOvec\nn}{\GammaSymb}-\tfrac{L}{L+1}\MGamma^{-1}\LGamma\MGamma^{-1} R_{\GammaSymb}\trkla{U\nn}\\
&\quad+\tfrac{1}{L+1}\beta\restr{\muOvec\nn}{\GammaSymb} -\beta \tfrac{1}{L+1}\restr{\MOmega^{-1}}{\GammaSymb\times\GammaSymb} R_{\GammaSymb}\trkla{U\nn} +\tfrac{1}{L+1}\restr{\MOmega^{-1}}{\GammaSymb\times\GammaSymb}\MGamma\restr{\muOvec\nn}{\GammaSymb}\\
&\quad+\tfrac{1}{L+1}\beta^2\MGamma^{-1}\restr{\MOmega}{\GammaSymb\times\GammaSymb}\restr{\muOvec\nn}{\GammaSymb} 	
	-\tfrac{1}{L+1}\beta^2\MGamma^{-1}R_{\GammaSymb}\trkla{U\nn}+\tfrac{1}{L+1}\beta\restr{\muOvec\nn}{\GammaSymb} \\[1ex]
&=\restr{\MOmega^{-1}}{\GammaSymb\times\GammaSymb}\mb{N}\restr{\muOvec\nn}{\GammaSymb}-\tfrac{L}{L+1}\MGamma^{-1}\LGamma\MGamma^{-1} R_{\GammaSymb}\trkla{U\nn} \\
&\quad+\tfrac{L}{L+1}\restr{\MOmega^{-1}}{\GammaSymb\times\GammaSymb}\restr{\LOmega}{\GammaSymb\times\InnerSymb}\restr{\MOmega}{\InnerSymb\times\InnerSymb}^{-1} \RInn\trkla{U\nn}\\
&\quad-\beta \tfrac{1}{L+1}\restr{\MOmega^{-1}}{\GammaSymb\times\GammaSymb}R_{\GammaSymb}\trkla{U\nn}-\tfrac{1}{L+1}\beta^2\MGamma^{-1} R_{\GammaSymb}\trkla{U\nn}\,\\[1ex]
&= \restr{\MOmega^{-1}}{\GammaSymb\times\GammaSymb}\mb{N}\restr{\muOvec\nn}{\GammaSymb} + \mb{B}\restr{\MOmega}{\InnerSymb\times\InnerSymb}^{-1}\RInn\trkla{U\nn} - \mb{C}\MGamma^{-1} R_{\GammaSymb}\trkla{U\nn},
\end{align*}
which vanishes due to \eqref{eq:expression:muOG}.

Although \eqref{eq:discscheme} is based on the sum of \eqref{eq:matrix:phi:bulk} and \eqref{eq:matrix:phi:boundary} multiplied by $\beta^{-1}$, solutions to \eqref{eq:discscheme}, if they exist, satisfy \eqref{eq:matrix:phi:bulk} and \eqref{eq:matrix:phi:boundary} individually.
\begin{lemma}\label{lem:equiv}
For any $L\geq 0$ such that $\tfrac{1}{L+1}\,,\tfrac{L}{L+1}\in\tekla{0,1}$, let $U\nn$ be a solution to \eqref{eq:discscheme} for given $U\no$. 
Then, $U\nn$ satisfies
\begin{align}
\MOmega\trkla{U\nn-U\no}+\tau\LOmega\muOvec\nn+\beta^{-1}\extend{\ekla{\MGamma\rkla{\restr{U\nn}{\GammaSymb}-\restr{U\no}{\GammaSymb}}+\tau\LGamma\muGvec\nn}}{\OmegaSymb}=0\,,\label{eq:matrixscheme:1}
\\[1ex]
\beta^{-1}\tfrac{L}{L+1}\MGamma\trkla{\restr{U\nn}{\GammaSymb}-\restr{U\no}{\GammaSymb}}+\tau\beta^{-1}\tfrac{L}{L+1}\LGamma\muGvec\nn + \tau\tfrac{1}{L+1}\MGamma\trkla{\beta\muGvec\nn-\restr{\muOvec\nn}{\GammaSymb}}=0\label{eq:matrixscheme:2}
\end{align}
with $\muOvec\nn$ and $\muGvec\nn$ defined in \eqref{eq:expression:muOG}, \eqref{eq:potential:muOInner}, and \eqref{eq:potential:muGvec}.
Furthermore, $U\nn$, $U\no$, $\muOvec\nn$, and $\muGvec\nn$ satisfy \eqref{eq:matrix:potential}.
\end{lemma}

\begin{proof}
The validity of \eqref{eq:matrix:potential} follows directly from the definitions \eqref{eq:potential:muOInner}, \eqref{eq:potential:muGvec} and the definitions of $R_{\GammaSymb}\trkla{U\nn}$ and $\RInn\trkla{U\nn}$.  Moreover, using \eqref{eq:potential:muGvec}, a solution of \eqref{eq:discscheme} clearly satisfies \eqref{eq:matrixscheme:1}. Therefore, it remains to show that it also satisfies \eqref{eq:matrixscheme:2}.  By \eqref{eq:matrixscheme:1} it holds that 
\begin{align*}
0 & = \tfrac{L}{L+1}\trkla{\trkla{\restr{\MOmega}{\GammaSymb\times\GammaSymb}+\beta^{-1}\MGamma}\trkla{\restr{U\nn}{\GammaSymb}-\restr{U\no}{\GammaSymb}}+\tau\beta^{-1}\LGamma\muGvec\nn+\tau\restr{\tekla{\LOmega\muOvec\nn}}{\GammaSymb}}\\[1ex]
& =  \tfrac{L}{L+1}\trkla{\trkla{\restr{\MOmega}{\GammaSymb\times\GammaSymb}+\beta^{-1}\MGamma}\trkla{\restr{U\nn}{\GammaSymb}-\restr{U\no}{\GammaSymb}}} \\
& \quad + \tau \tfrac{L}{L+1} \trkla{\restr{\MOmega}{\GammaSymb\times\GammaSymb}+\beta^{-1}\MGamma} \trkla{\MGamma^{-1} \LGamma \muGvec\nn} \\
& \quad + \tau  \tfrac{L}{L+1}  \restr{\tekla{\LOmega \muOvec\nn}}{\GammaSymb} - \tau \tfrac{L}{L+1} \trkla{\restr{\MOmega}{\GammaSymb\times\GammaSymb} \MGamma^{-1} \LGamma \muGvec\nn}.
\end{align*}
Using the following identity from the rearrangement of \eqref{eq:compatibility:1} 
\begin{align*}
&\tfrac{1}{L+1}\beta\restr{\MOmega^{-1}}{\GammaSymb\times\GammaSymb}\rkla{\restr{\MOmega}{\GammaSymb\times\GammaSymb}+\beta^{-1}\MGamma}\rkla{\beta\muGvec\nn-\restr{\muOvec\nn}{\GammaSymb}} \\
&\quad = \tfrac{L}{L+1}\restr{\tekla{\MOmega^{-1}\LOmega\muOvec\nn}}{\GammaSymb}-\tfrac{L}{L+1}\MGamma^{-1}\LGamma\muGvec\nn\,,
\end{align*}
we arrive at 
\begin{align*}
0 & =  \tfrac{L}{L+1}\trkla{\trkla{\restr{\MOmega}{\GammaSymb\times\GammaSymb}+\beta^{-1}\MGamma}\trkla{\restr{U\nn}{\GammaSymb}-\restr{U\no}{\GammaSymb}}} \\
& \quad + \tau \tfrac{L}{L+1} \trkla{\restr{\MOmega}{\GammaSymb\times\GammaSymb}+\beta^{-1}\MGamma} \trkla{\MGamma^{-1} \LGamma \muGvec\nn} \\
& \quad + \tau \tfrac{1}{L+1}\beta \rkla{\restr{\MOmega}{\GammaSymb\times\GammaSymb}+\beta^{-1}\MGamma}\rkla{\beta\muGvec\nn-\restr{\muOvec\nn}{\GammaSymb}}.
\end{align*}
Multiplying by $\trkla{\restr{\MOmega}{\GammaSymb\times\GammaSymb}+\beta^{-1}\MGamma}^{-1}$ and then by $\beta^{-1} \MGamma$ yields \eqref{eq:matrixscheme:2}.
\end{proof}

\begin{lemma}\label{lem:matrix:energy}
Given $U\no\in\R^{\dim\UhO}$, let $U\nn\in\R^{\dim\UhO}$ be a solution to \eqref{eq:discscheme} with $\muOvec\nn\in\R^{\dim\UhO}$ and $\muGvec\nn\in\R^{\dim\UhG}$ be defined in \eqref{eq:expression:muOG}, \eqref{eq:potential:muOInner}, and \eqref{eq:potential:muGvec}.  Under \eqref{item:disc:time}, \eqref{item:disc:space}, \eqref{item:disc:gamma} and \eqref{ass:pot},
the following estimate holds true:
\begin{equation}\label{matrix:energy}
\begin{aligned}
& \tfrac12 {U\nn}^T\LOmega U\nn + \tfrac12 \rkla{U\nn-U\no}^T\LOmega\rkla{U\nn-U\no}  \\
& \qquad +\tfrac12\kappa\restr{U\nn}{\GammaSymb}^T\LGamma\restr{U\nn}{\GammaSymb} +\tfrac12\kappa\rkla{\restr{U\nn}{\GammaSymb}-\restr{U\no}{\GammaSymb}}^T\LGamma \rkla{\restr{U\nn}{\GammaSymb}-\restr{U\no}{\GammaSymb}}\\
& \qquad +\boldsymbol{1}^T\MOmega  F\trkla{U\nn}  +\boldsymbol{1}_{\GammaSymb}^T\MGamma G\trkla{U\nn} +\tau{\muOvec\nn}^T\LOmega\muOvec\nn +\tau{\muGvec\nn}^T\LGamma\muGvec\nn + \mathcal{B}_L\\[1ex]
& \quad \leq \tfrac12{U\no}^T\LOmega U\no + \tfrac12\kappa \restr{U\no}{\GammaSymb}^T\LGamma\restr{U\no}{\GammaSymb}\\
& \qquad +\boldsymbol{1}^T\MOmega  F\trkla{U\no} +\boldsymbol{1}_{\GammaSymb}^T\MGamma G\trkla{U\no}\,,
\end{aligned}
\end{equation}
with $\boldsymbol{1}:=\trkla{1,...,1}^T\in \R^{\dim\UhO}$, $\boldsymbol{1}_{\GammaSymb}:=\restr{\boldsymbol{1}}{\GammaSymb}$, and 
\begin{align*}
\mathcal{B}_L:=\left\{\begin{matrix}
\tau L^{-1}\rkla{\beta\muGvec\nn-\restr{\muOvec\nn}{\GammaSymb}}^T\MGamma \rkla{\beta\muGvec\nn-\restr{\muOvec\nn}{\GammaSymb}}&\text{if~}L>0,\\
0&\text{if~} L = 0.
\end{matrix}\right. 
\end{align*}
Furthermore, we have $\beta\muGvec\nn=\restr{\muOvec\nn}{\GammaSymb}$ if $L=0$.
\end{lemma}

\begin{proof}
Multiplying \eqref{eq:discscheme} by  the transpose of the vector 
\begin{align}
\rkla{\MOmega+\beta^{-1}\extend{\MGamma}{\OmegaSymb\times\OmegaSymb}}^{-1} \MOmega {\muOvec\nn}  +\extend{\ekla{\rkla{\restr{\MOmega}{\GammaSymb\times\GammaSymb}+\beta^{-1}\MGamma}^{-1}\MGamma\muGvec\nn}}{\OmegaSymb}\,,
\end{align}
and using \eqref{eq:potential:muGvec} we obtain
\begin{align*}
\begin{split}
0& = {\muOvec\nn}^T\MOmega\rkla{U\nn-U\no}+{\muGvec\nn}^T\MGamma\rkla{\restr{U\nn}{\GammaSymb}-\restr{U\no}{\GammaSymb}} \\
& \quad +\tau{\muOvec\nn}^T\LOmega\rkla{\MOmega+\beta^{-1}\extend{\MGamma}{\OmegaSymb\times\OmegaSymb}}^{-1}\MOmega\muOvec\nn \\
& \quad +\tau{\muGvec\nn}^T\MGamma\rkla{\restr{\MOmega}{\GammaSymb\times\GammaSymb}+\beta^{-1}\MGamma}^{-1}\restr{\ekla{\LOmega\muOvec\nn}}{\GammaSymb}  \\
& \quad + \tau\beta^{-1}\restr{\muOvec\nn}{\GammaSymb}^T\restr{\MOmega}{\GammaSymb\times\GammaSymb}\rkla{\restr{\MOmega}{\GammaSymb\times\GammaSymb}+\beta^{-1}\MGamma}^{-1}\LGamma\muGvec\nn \\
& \quad + \tau\beta^{-1}{\muGvec\nn}^T\MGamma\rkla{\restr{\MOmega}{\GammaSymb\times\GammaSymb}+\beta^{-1}\MGamma}^{-1}\LGamma\muGvec\nn\\
& =:  I_1 + I_2 + I_3 + I_4 + I_5 + I_6\,.
\end{split}
\end{align*}
By the convexity of $F_1$ and concavity of $F_2$, it is easy to see that for any $a, b \in \R$,
\begin{align*}
F_1(a) - F_1(b) \leq F_1'(a)(a-b), \quad F_2(a) - F_2(b) \leq F_2'(b)(a-b).
\end{align*}
Then, testing \eqref{eq:matrix:potential} with $\rkla{U\nn-U\no}$ leads to
\begin{align}
\begin{split}
I_1 + I_2 & \geq \tfrac12 {U\nn}^T\LOmega U\nn + \tfrac12 \rkla{U\nn-U\no}^T\LOmega\rkla{U\nn-U\no} -\tfrac12{U\no}^T\LOmega U\no \\
&\quad +\tfrac12\kappa\restr{U\nn}{\GammaSymb}^T\LGamma\restr{U\nn}{\GammaSymb} +\tfrac12\kappa\rkla{\restr{U\nn}{\GammaSymb}-\restr{U\no}{\GammaSymb}}^T\LGamma\rkla{\restr{U\nn}{\GammaSymb}-\restr{U\no}{\GammaSymb}} \\
& \quad -\tfrac12\kappa \restr{U\no}{\GammaSymb}^T\LGamma\restr{U\no}{\GammaSymb}
+\boldsymbol{1}^T\MOmega F\trkla{U\nn}-\boldsymbol{1}^T\MOmega  F\trkla{U\no} \\
&\quad +\boldsymbol{1}_{\GammaSymb}^T\MGamma G\trkla{U\nn}-\boldsymbol{1}_{\GammaSymb}^T\MGamma G\trkla{U\no}\,.
\end{split}
\end{align}
For the terms $I_3, \dots, I_6$, we use the compatibility condition \eqref{eq:compatibility:1}.
For the case $L>0$, \eqref{eq:compatibility:1} can be written as
\begin{align}
\begin{aligned}
&\restr{\ekla{\MOmega^{-1}\LOmega\muOvec\nn}}{\GammaSymb}-\MGamma^{-1}\LGamma\muGvec\nn \\
&\quad= \beta L^{-1}\restr{\MOmega^{-1}}{\GammaSymb\times\GammaSymb}\rkla{\restr{\MOmega}{\GammaSymb\times\GammaSymb}+\beta^{-1}\MGamma}\rkla{\beta\muGvec\nn-\restr{\muOvec\nn}{\GammaSymb}}\,.
\end{aligned}
\end{align}
Then, using the symmetry of the matrices $\MOmega$ and $\MGamma$, we find that
\begin{align*}
I_5 + I_6 & = \tau \beta^{-1}\restr{\muOvec\nn}{\GammaSymb}^T \MGamma \rkla{\restr{\MOmega}{\GammaSymb\times\GammaSymb}+\beta^{-1}\MGamma}^{-1} \restr{\MOmega}{\GammaSymb\times\GammaSymb}\restr{\ekla{\MOmega^{-1}\LOmega\muOvec\nn}}{\GammaSymb} \\
& \quad - \tau L^{-1} \restr{\muOvec\nn}{\GammaSymb}^T \MGamma \rkla{\beta\muGvec\nn-\restr{\muOvec\nn}{\GammaSymb}} + \tau\beta^{-1}{\muGvec\nn}^T\MGamma\rkla{\restr{\MOmega}{\GammaSymb\times\GammaSymb}+\beta^{-1}\MGamma}^{-1}\LGamma\muGvec\nn.
\end{align*}
On the other hand,
\begin{align*}
I_3 + I_4 
& =  \tau{\muOvec\nn}^T\LOmega\rkla{\MOmega+\beta^{-1}\extend{\MGamma}{\OmegaSymb\times\OmegaSymb}}^{-1}\MOmega\muOvec\nn  \\
& \quad + \tau{\muGvec\nn}^T \restr{\MOmega}{\GammaSymb\times\GammaSymb} \rkla{\restr{\MOmega}{\GammaSymb\times\GammaSymb}+\beta^{-1}\MGamma}^{-1} \LGamma\muGvec\nn +  \beta L^{-1} \tau{\muGvec\nn}^T\MGamma \rkla{\beta\muGvec\nn-\restr{\muOvec\nn}{\GammaSymb}},
\end{align*}
and so we infer
\begin{align}\label{I3I6}
I_3 + \cdots + I_6 & = \tau L^{-1}\rkla{\beta\muGvec\nn-\restr{\muOvec\nn}{\GammaSymb}}^T\MGamma \rkla{\beta\muGvec\nn-\restr{\muOvec\nn}{\GammaSymb}} + \tau{\muGvec\nn}^T  \LGamma\muGvec\nn + \tau{\muOvec\nn}^T  \LOmega\muOvec\nn.
\end{align}
Combining with the inequality for $I_1 + I_2$ we arrive at \eqref{matrix:energy} for the case $L > 0$.  Meanwhile, for the case $L = 0$, we directly infer from the compatibility condition \eqref{eq:compatibility:1} that $\beta\muGvec\nn=\trestr{\muOvec\nn}{\GammaSymb}$.  Then, we obtain directly
\begin{align*}
I_3 + I_4 = \tau{\muOvec\nn}^T\LOmega\muOvec\nn, \quad I_5 + I_6 = \tau{\muGvec\nn}^T\LGamma\muGvec\nn,
\end{align*}
which leads to \eqref{matrix:energy} for the case $L = 0$.
\end{proof}
Next, we use the a priori estimate established in Lemma \ref{lem:matrix:energy} to prove the existence of discrete solutions.
\begin{lemma}
 Given $U\no\in\R^{\dim\UhO}$, under \eqref{item:disc:time}, \eqref{item:disc:space}, \eqref{item:disc:gamma}, and \eqref{ass:pot}, there exists at least one vector $U\nn\in\R^{\dim\UhO}$ solving \eqref{eq:discscheme}.
\end{lemma}
\begin{proof}
Firstly, we note that $\muOvec\nn$ and $\muGvec\nn$ are uniquely determined if $U\nn$ and $U\no$ are given.  Hence, for given $U\no$ and an arbitrary vector $U$, we use the notation $\muOvec\trkla{U}$ and $\muGvec\trkla{U}$ to denote the corresponding vectors for the chemical potentials.  In particular, $\muOvec\trkla{U\nn} = \muOvec\nn$ and $\muGvec\trkla{U\nn} = \muGvec\nn$.  

Next, testing \eqref{eq:discscheme} by $\boldsymbol{1}$ shows that 
\begin{align*}
\boldsymbol{1}^T\rkla{\MOmega+\beta^{-1}\extend{\MGamma}{\OmegaSymb\times\OmegaSymb}}\rkla{U\nn-U\no}=0,
\end{align*}
and so, without loss of generality we assume that $\boldsymbol{1}^T\rkla{\MOmega+\beta^{-1}\extend{\MGamma}{\OmegaSymb\times\OmegaSymb}}U^0 = 0$, which in turn implies
\begin{align}\label{num:mean:0}
\boldsymbol{1}^T\rkla{\MOmega+\beta^{-1}\extend{\MGamma}{\OmegaSymb\times\OmegaSymb}}U\nn=0 \quad \forall n \geq 1.
\end{align}
A consequence is the following Poincar\'e-type inequality: There exists a positive constant $c$ such that for all vectors $U$ fulfilling \eqref{num:mean:0},
\begin{align}\label{dis:Poin}
 {U}^T\rkla{\MOmega+\beta^{-1}\extend{\MGamma}{\OmegaSymb\times\OmegaSymb}} U \leq c \left ( {U}^T\rkla{\LOmega+\beta^{-1}\extend{\LGamma}{\OmegaSymb\times\OmegaSymb}} U \right ).
\end{align}
Recalling the definition of the matrices $\MOmega$ and $\LOmega$, if we associate the vector $U$ to a function $u\h \in \UhO$ then the above inequality \eqref{dis:Poin} reads as
\begin{align*}
\| u\h \|_{L^2(\Omega)}^2 + \beta^{-1} \| u\h \|_{L^2(\Gamma)}^2  \leq c \big ( \| \nabla u\h \|_{L^2(\Omega)}^2 + \beta^{-1} \| \gradg u\h \|_{L^2(\Gamma)}^2 \big )
\end{align*}
for functions $u\h \in U_h^\Omega$ such that $\beta\abs{\Omega}\mean{u\h}_\Om + \abs{\Gamma}\mean{u\h}_\Ga = 0$.  We mention the proof of this Poincar\'e-type inequality follows from the usual contradiction argument using the condition $\beta\abs{\Omega}\mean{u}_\Om + \abs{\Gamma}\mean{u}_\Ga = 0$.

We can establish the existence of discrete solutions as follows.  Assuming that \eqref{eq:discscheme} has no solution in the closed set
\begin{align*}
B_R:=\gkla{W\in \R^{\dim\UhO}\,:\,\boldsymbol{1}^T\rkla{\MOmega+\beta^{-1}\extend{\MGamma}{\OmegaSymb\times\OmegaSymb}} W=0\text{~and~} W^T\MOmega W\leq R^2}
\end{align*}
for any $R>0$, the function
\begin{align*}
\mathcal{G}\trkla{U}& :=U-U\no+\tau \rkla{\MOmega+\beta^{-1}\extend{\MGamma}{\OmegaSymb\times\OmegaSymb}}^{-1}\LOmega\muOvec\trkla{U} \\
& \quad +\tau \beta^{-1} \rkla{\MOmega+\beta^{-1}\extend{\MGamma}{\OmegaSymb\times\OmegaSymb}}^{-1}\extend{\ekla{\LGamma\muGvec\trkla{U}}}{\OmegaSymb}
\end{align*}
has no roots in $B_R$, and consequently, the function
\begin{align}\label{fp}
\mathcal{H}\trkla{U}:=-R\frac{\mathcal{G}\trkla{U} }{\sqrt{\mathcal{G}\trkla{U}^T\MOmega\mathcal{G}\trkla{U}}}
\end{align}
is a continuous mapping from $B_R$ to $\partial B_R\subset B_R$.
According to Brouwer's fixed point theorem, there exists at least one fixed point $U^*$ of $\mathcal{H}$.
In the following, we show that $U^*$ satisfies
\begin{align}\label{eq:contradiction}
0<{U^*}^T \Big ( \MOmega\muOvec\trkla{U^*}+\extend{\ekla{\MGamma\muGvec\trkla{U^*}}}{\OmegaSymb} \Big ) <0
\end{align}
for $R$ sufficiently large.  This contradiction shows that our initial assumption on the non-existence of roots of $\mathcal{G}$ in $B_R$ is false, implying the existence of solutions to \eqref{eq:discscheme}.

For convenience we denote $V=\MOmega\muOvec\trkla{U^*}+\extend{\ekla{\MGamma\muGvec\trkla{U^*}}}{\OmegaSymb}$.  
To obtain the first inequality in \eqref{eq:contradiction},  we use \eqref{eq:potential:muGvec}, \eqref{eq:potential:muOInner}, Young's inequality with $0<\alpha \ll 1$, and the convex-concave decomposition of $F$ and $G$ to deduce that 
\begin{align*}
\begin{split}
{U^*}^T V=&\,{U^*}^T\MOmega\muOvec\trkla{U^*}+\restr{U^*}{\GammaSymb}^T\MGamma\muGvec\trkla{U^*}\\
=&\,  {U^*}^T\LOmega U^* +\kappa\restr{U^*}{\GammaSymb}^T\LGamma\restr{U^*}{\GammaSymb}\\
&+ {U^*}^T\MOmega\rkla{F^\prime_1\trkla{U^*}+F_2^\prime\trkla{0}}+ {U^*}^T\MOmega\rkla{F^\prime_2\trkla{U\no}-F^\prime_2\trkla{0}}\\
&+ \restr{U^*}{\GammaSymb}^T\MGamma\rkla{G^\prime_1\trkla{\restr{U^*}{\GammaSymb}}+G^\prime_2\trkla{0}} +\restr{U^*}{\GammaSymb}^T\MGamma\rkla{G^\prime_2\trkla{\restr{U\no}{\GammaSymb}}-G^\prime_2\trkla{0}}\\
\geq& \min(1, \kappa \beta) {U^*}^T\rkla{\LOmega+\beta^{-1}\extend{\LGamma}{\OmegaSymb\times\OmegaSymb}} U^*+ \boldsymbol{1}^T\MOmega\rkla{F\trkla{U^*}-F\trkla{0}} \\
& \, -\alpha {U^*}^T\MOmega U^* + \boldsymbol{1}_{\GammaSymb}\MGamma\rkla{G\trkla{\restr{U^*}{\GammaSymb}}-G\trkla{0}} -\alpha\restr{U^*}{\GammaSymb}^T\MGamma\restr{U^*}{\GammaSymb} -C_\alpha
\end{split}
\end{align*}
for some constant $C_\alpha > 0$ depending only on $\alpha$,  $U\no$, $F_2^\prime\trkla{U\no}$, $F_2^\prime\trkla{0}$, $G_2^\prime\trkla{U\no}$, and $G_2^\prime\trkla{0}$.  Since $F$ and $G$ are bounded from below, after applying the Poincar\'e-type inequality \eqref{dis:Poin}, we obtain for some positive constant $\tilde{c}$ independent of $U^*$ that
\begin{align*}
{U^*}^T V \geq \tilde{c}{U^*}^T\rkla{\MOmega+\beta^{-1}\extend{\MGamma}{\OmegaSymb\times\OmegaSymb}}U^*  -\alpha {U^*}^T\MOmega U^* -\alpha\restr{U^*}{\GammaSymb}^T\MGamma\restr{U^*}{\GammaSymb} -C_\alpha\,.
\end{align*}
Choosing $\alpha$ sufficiently small, we absorb the second and third term into the first term and infer for positive constants $\hat{c}$ and $C$ independent of $U^*$ that
\begin{align*}
{U^*}^T V \geq \hat{c} {U^*}^T\rkla{\MOmega+\beta^{-1}\extend{\MGamma}{\OmegaSymb\times\OmegaSymb}}U^*- C \geq \hat{c} {U^*}^T\rkla{\MOmega}U^* - C = \hat{c} R^2 - C.
\end{align*}
Then, choosing $R$ sufficiently large yields the first inequality ${U^*}^T V > 0$.
To derive the second inequality in \eqref{eq:contradiction}, we recall the computations from the proof of Lemma \ref{lem:matrix:energy} and \eqref{dis:Poin} which provide
\begin{align*}
&\mathcal{G}\trkla{U^*}^TV\geq c{U^*}^T\rkla{\LOmega+\beta^{-1}\extend{\LGamma}{\OmegaSymb\times\OmegaSymb}}U^*-C\geq \tilde{c}{U^*}^T\rkla{\MOmega+\beta^{-1}\extend{\MGamma}{\OmegaSymb\times\OmegaSymb}}U^* -C \\
&\quad \geq \tilde{c}R^2-C\,,
\end{align*}
where the right-hand side is positive for $R$ sufficiently large. 
Hence, using \eqref{fp}, we see that $\mathcal{G}\trkla{U^*}^T V > 0$ is equivalent to the second inequality ${U^*}^TV < 0$ in \eqref{eq:contradiction}.
\end{proof}

\subsection{Uniform bounds}
In this section, we collect uniform bounds on the discrete solutions established in the last section.
As shown in Lemma \ref{lem:equiv}, given $u\h\no\in\UhO$, for any $L \geq 0$, the proposed scheme is equivalent to finding $u\h\nn\in\UhO$ satisfying
\begin{subequations}\label{eq:fe:final}
\begin{alignat}{2}
& \iOmega\Ih{\dtaum u\h\nn w\h} + \nabla\muOh\nn\cdot\nabla w\h \dx +\beta^{-1}\iGamma \IhG{\dtaum u\h\nn w\h} +\gradg\muGh\nn\cdot\gradg w\h \dG=0\,,\label{eq:fe:final:1}\\
&\tfrac{1}{L+1}\iGamma L \big ( \IhG{\dtaum u\h\nn z\h}+ \gradg\muGh\nn\cdot \gradg z\h \big ) + \beta \trkla{\beta\muGh\nn-\muOh\nn}z\h \dG=0\,,\label{eq:fe:final:2}\\
& \iOmega\Ih{\muOh\nn\eta\h} \dx +\iGamma\IhG{\muGh\nn\eta\h} \dG =\iOmega  \nabla u\h\nn\cdot\nabla\eta\h + \Ih{\trkla{F^\prime_1\trkla{u\h\nn}+F^\prime_2\trkla{u\h\no}}\eta\h} \dx \nonumber\\
&\qquad\qquad \qquad +\iGamma \kappa\gradg u\h\nn\cdot\gradg\eta\h + \IhG{\trkla{G^\prime_1\trkla{u\h\nn}+G_2^\prime\trkla{u\h\no}}\eta\h} \dG\,,\label{eq:fe:final:3}
\end{alignat}
\end{subequations}
for all $w\h,\eta\h\in\UhO$ and $z\h\in\UhG$, with $\muOh\in\UhO$, $\muGh\in\UhG$ uniquely prescribed by $u\h\nn,\,u\h\no\in\UhO$. It is worth noting that in the limit $L\rightarrow\infty$, \eqref{eq:fe:final:1} and \eqref{eq:fe:final:2} become
\begin{align}
 \iOmega\Ih{\dtaum u\h\nn w\h}  + \nabla\muOh\nn\cdot\nabla w\h =0 \dx \,,\quad 
 \iGamma\IhG{\dtaum u\h\nn z\h} + \gradg\muGh\nn\cdot \gradg z\h \dG =0,\label{eq:fe:LW1}
\end{align}
which together with \eqref{eq:fe:final:3} is a discretisation of \eqref{CH:LW} that was analysed in \cite{Metzger2019}.  On the other hand, for the case $L=0$, \eqref{eq:fe:final:2} reduces to $\iGamma\IhG{\trkla{\beta\muGh\nn-\muOh\nn}z\h} \dG =0$, and together with \eqref{eq:fe:final:1} and \eqref{eq:fe:final:3} we obtain a discretisation of \eqref{CH:GMS}.

\begin{lemma}\label{lem:energy:global}
Given $u\h\no \in \UhO$,  under \eqref{item:disc:time}, \eqref{item:disc:space}, \eqref{item:disc:gamma}, and \eqref{ass:pot},  let $\trkla{u\h\nn,\muOh\nn,\muGh\nn}\in \UhO\times\UhO\times\UhG$ be a solution to \eqref{eq:fe:final} for $n=1,...,N$.  Then, there exists a constant $C > 0$ depending only on $u_0$ and $\kappa$ such that 
\begin{align}\label{disc:unif:bdd}
\notag & \max_{n=0,...,N} \norm{u\h\nn}_{H^1\trkla{\Omega}}^2 +\max_{n=0,...,N} \iOmega\Ih{F\trkla{u\h\nn}} \dx + \max_{n=0,...,N} \norm{u\h\nn}_{H^1\trkla{\Gamma}}^2 \\
\notag & \quad +\max_{1,...,N}\iGamma\IhG{G\trkla{u\h\nn}} \dG +\sum_{n=1}^N\norm{\nabla u\h\nn-\nabla u\h\no}_{L^2\trkla{\Omega}}^2 + \sum_{n=1}^N\norm{\gradg u\h\nn-\gradg u\h\no}_{L^2\trkla{\Gamma}}^2  \\
& \quad + \tau\sum_{n=1}^N\norm{\muOh\nn}_{H^1\trkla{\Omega}}^2  +\tau\sum_{n=1}^N\norm{\muGh\nn}_{H^1\trkla{\Gamma}}^2  \leq C.
\end{align}
Additionally it holds that 
\begin{align*}
\tau \sum_{n=1}^N \norm{\beta\muGh\nn-\muOh\nn}_{L^2\trkla{\Gamma}}^2 \leq CL.
\end{align*}
\end{lemma}
\begin{proof}
Summing \eqref{matrix:energy} over the time steps  from $n = 0$ to $n = k \leq N$,  applying \eqref{eq:initialcond:bound}  and then take the maximum over $k$  yields
\begin{align*}
& \max_{n=1,...,N}\tfrac{1}{2} \iOmega\abs{\nabla u\h\nn}^2 \dx +\max_{n=1,...,N}\iOmega\Ih{F\trkla{u\h\nn}} +\max_{n=1,...,N}\tfrac{\kappa}2\iGamma\abs{\gradg u\h\nn}^2 \\
& \quad  +\max_{n=1,...,N}\iGamma\IhG{G\trkla{u\h\nn}} \dG+\sum_{n=1}^N\tfrac{1}{2} \iOmega\abs{\nabla u\h\nn-\nabla u\h\no}^2 \dx  \\
& \quad + \sum_{n=1}^N\tfrac{\kappa}{2}\iGamma\abs{\gradg u\h\nn-\gradg u\h\no}^2 \dG + \tau\sum_{n=1}^N \iOmega\abs{\nabla\muOh\nn}^2 \dx +\tau\sum_{n=1}^N\iGamma\abs{\gradg\muGh\nn}^2 \dG  \leq C\trkla{u_0}\,.
\end{align*}
Furthermore, we have
\begin{align*}
\tau\sum_{n=1}^N \iGamma\IhG{\abs{\beta\muGh\nn-\muOh\nn}^2} \dG \leq C\trkla{u_0}L\,,
\end{align*}
where the statement for $L=0$ is trivial due to Lemma \ref{lem:matrix:energy}.
 By \eqref{ass:pot:2}, the  bounds on $\Ih{F\trkla{u\h\nn}}$ and $\IhG{G\trkla{u\h\nn}}$ also provide bounds on $\Ih{|u\h\nn|^2}$ and $\IhG{|u\h\nn|^2}$ which allows us to deduce the bounds in the $H^1$-norms.  For the $L^2$-norms on $\muOh\nn$ and $\muGh\nn$, we can employ similar arguments used above in Step 3 of Section \ref{sec:wellposed}, see also \cite[Cor.~4.1]{Metzger2019}.
\end{proof}

\begin{lemma}\label{lem:nikolskii}
 Suppose that \eqref{item:disc:time}, \eqref{item:disc:space}, \eqref{item:disc:gamma} and \eqref{ass:pot} hold.  Given $u\h\no \in \UhO$, let $\trkla{u\h\nn,\muOh\nn,\muGh\nn}$ ${\in \UhO\times\UhO\times\UhG}$ be a solution to \eqref{eq:fe:final} for $n=1,...,N$.  Then, the following estimates hold
\begin{align*}
\tau\sum_{k=0}^{N-l}\norm{u\h^{k+l}-u\h^k}_{L^2\trkla{\Omega}}^2 + \tau\sum_{k=0}^{N-l}\norm{u\h^{k+l}-u\h^k}_{L^2\trkla{\Gamma}}^2\leq C\tau l,
\end{align*}
for $l\in\tgkla{1,...,N}$ with $C>0$ independent of $l$, $L$, $h$, and $\tau$.
\end{lemma}
\begin{proof}
For $0\leq k\leq N-l$, we test \eqref{eq:fe:final:1} by $w\h=\trkla{u\h^{k+l}-u\h^k}$, sum from $n=k+1$ to $k+l$ and employ \eqref{disc:unif:bdd} which yields
\begin{align*}
& \iOmega\Ih{\abs{u\h^{k+l}-u\h^k}^2} \dx +\beta^{-1}\iGamma\IhG{\abs{u\h^{k+l}-u\h^k}^2} \dG\\
& \quad \leq \abs{\tau\sum_{n=k+1}^{k+l} \iOmega\nabla\muOh\nn\cdot\nabla\trkla{u\h^{k+l}-u\h^k} \dx}  +\beta^{-1}\abs{\tau\sum_{n=k+1}^{k+l}\iGamma\gradg\muGh\nn\cdot\gradg\trkla{u\h^{k+l}-u\h^k} \dG}\\
& \quad \leq\tau\sum_{n=k+1}^{k+l}\norm{\nabla\muOh\nn}_{L^2\trkla{\Omega}}\norm{\nabla u\h^{k+l}-\nabla u\h^k}_{L^2\trkla{\Omega}} \\
& \qquad + \beta^{-1}\tau\sum_{n=k+1}^{k+l}\norm{\gradg\muGh\nn}_{L^2\trkla{\Gamma}}\norm{\gradg u\h^{k+l}-\gradg u\h^k}_{L^2\trkla{\Gamma}}\,.
\end{align*}
Multiplying by $\tau$ and summing from $k=0$ to $N-l$, we infer that with the help of \eqref{disc:unif:bdd} that
\begin{align*}
& \tau\sum_{k=0}^{N-l}\norm{u\h^{k+l}-u\h^k}_{L^2\trkla{\Omega}}^2 +\tau\sum_{k=0}^{N-l}\norm{u\h^{k+l}-u\h^k}_{L^2\trkla{\Gamma}}^2\\
& \quad \leq C\tau^2\sum_{m=1}^l\rkla{\sum_{k=0}^{N-l}\norm{\nabla u\h^{k+l}-\nabla u\h^k}_{L^2\trkla{\Omega}}^2}^{1/2}\rkla{\sum_{k=0}^{N-l}\norm{\nabla\muOh^{k+m}}_{L^2\trkla{\Omega}}^2}^{1/2}\\
& \qquad +C\tau^2\sum_{m=1}^l\rkla{\sum_{k=0}^{N-l}\norm{\gradg u\h^{k+l}-\gradg u\h^k}_{L^2\trkla{\Gamma}}^2}^{1/2}\rkla{\sum_{k=0}^{N-l}\norm{\gradg\muGh^{k+m}}_{L^2\trkla{\Gamma}}^2}^{1/2}\\
& \quad \leq C\tau l \rkla{\tau N \max_{n=1,...,N}\norm{\nabla u\h\nn}_{L^2\trkla{\Omega}}^2}^{1/2}\rkla{\tau\sum_{n=1}^N\norm{\nabla\muOh\nn}_{L^2\trkla{\Omega}}}^{1/2}\\
& \qquad +C\tau l \rkla{\tau N \max_{n=1,...,N}\norm{\gradg u\h\nn}_{L^2\trkla{\Gamma}}^2}^{1/2}\rkla{\tau\sum_{n=1}^N\norm{\gradg\muGh\nn}_{L^2\trkla{\Gamma}}}^{1/2}\\
& \quad \leq C\tau l\,.
\end{align*}
Thus, the proof is complete.
\end{proof}

\subsection{Passing to the limit}
Let $L_0 \in [0,\infty]$ be arbitrary.  To pass to the limit $\trkla{h,\tau,L}\rightarrow\trkla{0,0,L_0}$, we define three interpolation functions for a collection of time-discrete functions $\{a^n\}_{n=0}^{N}$ as follows:
\begin{subequations}
\begin{align}
a\tl(\cdot,t)&:=\tfrac{t-t\no}{\tau} a^n\trkla{\cdot} + \tfrac{t\nn-t}{\tau}a\no\trkla{\cdot},&&t\in\tekla{t\no,t\nn},\,n\geq1\,,\\
a\tp\trkla{\cdot,t}&:=a^n\trkla{\cdot}\,~a\tm\trkla{\cdot,t}:=a\no\trkla{\cdot},&&t\in(t\no,t\nn],\,n\geq1\,.
\end{align}
\end{subequations}
If a statement is valid for $a\tl$, $a\tp$, and $a\tm$, we will use $a\tpm$.
Using this notation, we are able to write the uniform bounds established in the last section as
\begin{subequations}\label{eq:disc:bounds}
\begin{alignat}{2}
\notag & \norm{u\h\tpm}_{L^\infty\trkla{0,T;H^1\trkla{\Omega}}}^2 + \norm{u\h\tpm}_{L^\infty\trkla{0,T;H^1\trkla{\Gamma}}}^2 \\
\notag & \quad + \tau^{-1}\norm{\nabla u\h\tp- \nabla u\h\tm}_{L^2\trkla{0,T;L^2\trkla{\Omega}}}^2 + \tau^{-1}\norm{\gradg u\h\tp- \gradg u\h\tm}_{L^2\trkla{0,T;L^2\trkla{\Gamma}}}^2 \\
& \quad +\norm{\muOh\tp}_{L^2\trkla{0,T;H^1\trkla{\Omega}}}^2+\norm{\muGh\tp}_{L^2\trkla{0,T;H^1\trkla{\Gamma}}}^2\leq C\,, \label{eq:bounds:1}\\[1ex]
& \norm{\beta\muGh\tp-\muOh\tp}_{L^2\trkla{0,T;L^2\trkla{\Gamma}}}^2 \leq CL\,,\label{eq:bounds:1.5}\\
&\norm{u\h\tpm\trkla{\cdot+l\tau}-u\h\tpm\trkla{\cdot}}_{L^2\trkla{0,T;L^2\trkla{\Omega}}}^2\leq C\tau l\,,\label{eq:bounds:2}\\
&\norm{u\h\tpm\trkla{\cdot+l\tau}-u\h\tpm\trkla{\cdot}}_{L^2\trkla{0,T;L^2\trkla{\Gamma}}}^2\leq C\tau l\label{eq:bounds:3}\,.
\end{alignat}
\end{subequations}

\begin{lemma}\label{lem:disc:conv}
 Under \eqref{item:disc:time}, \eqref{item:disc:space}, \eqref{item:disc:gamma}, and \eqref{ass:pot},  there exist functions $\trkla{u,u_\Gamma,\muO,\muG}$ and a subsequence, denoted by $\tgkla{u\h\tpm,\muOh\tpm,\muGh\tpm}\Sub[-3pt]{h,\tau,L}$, satisfying
\begin{align}
\label{REG:LIM} 
\left\{
\begin{aligned}
	&u\in L^\infty\trkla{0,T;H^1\trkla{\Omega} \cap L^4(\Om)}\,,\quad &&u_\Gamma \in L^\infty\trkla{0,T;H^1\trkla{\Gamma} \cap L^4(\Ga)}\,,\\
	&\muO\in L^2\trkla{0,T;H^1\trkla{\Omega}}\,,\quad &&\muG\in L^2\trkla{0,T;H^1\trkla{\Omega}}
\end{aligned}
\right. 
\end{align}
such that $u \vert_{\Sigma_T}=u_\Gamma$ a.e.~on $\Sigma_T$, and as $\trkla{h,\tau,L}\rightarrow\trkla{0,0,L_0}$,
\begin{subequations}
\begin{alignat}{4}
u\h\tpm & \to u && \text{ weakly* in } L^\infty\trkla{0,T;H^1\trkla{\Omega}}\,,\label{eq:conv:weakstar:u}\\
&&& \quad \text{ strongly in } L^r\trkla{0,T;L^s\trkla{\Omega}}\quad \text{with~}r<\infty,~s\in[1,\tfrac{2d}{d-2})\,,\label{eq:conv:strong:u} \\
u\h\tpm \vert_{\Sigma_T} &\to u_\Gamma && \text{ weakly* in } L^\infty\trkla{0,T;H^1\trkla{\Gamma}}\,,\label{eq:conv:weakstar:traceu}\\
&&& \quad \text{ strongly in } L^r\trkla{0,T;L^s\trkla{\Gamma}}\quad\text{with~}r,s<\infty\,,\label{eq:conv:strong:traceu}\\
\muOh\tp &\to \muO&& \text{ weakly in }L^2\trkla{0,T;H^1\trkla{\Omega}}\,,\label{eq:conv:weak:muO}\\
\muOh\tp \vert_{\Sigma_T} &\to \muO\vert_{\Sigma_T} &&\text{ weakly in } L^2\trkla{0,T;H^{1/2}\trkla{\Gamma}}\,,\label{eq:conv:weak:tracemuO} \\
\muGh\tp & \to \muG&&\text{ weakly in } L^2\trkla{0,T;H^1\trkla{\Gamma}}\label{eq:conv:weak:muG}\,.
\end{alignat}
\end{subequations}
If $L_0=0$, we additionally obtain
\begin{align}\label{eq:conv:mudiff}
\beta\muGh\tp- \muOh\tp \vert_{\Sigma_T} \to 0\text{ strongly in }L^2\trkla{0,T;L^2\trkla{\Gamma}}\,.
\end{align}
\end{lemma}
\begin{proof}
The convergences expressed in \eqref{eq:conv:weakstar:u}, \eqref{eq:conv:weakstar:traceu}, \eqref{eq:conv:weak:muO}, and \eqref{eq:conv:weak:muG} are direct consequences of bounds established in \eqref{eq:bounds:1}.
To obtain the strong convergence in \eqref{eq:conv:strong:u} and \eqref{eq:conv:strong:traceu}, we have combined \eqref{eq:bounds:1} with \eqref{eq:bounds:2}, \eqref{eq:bounds:3}, and apply a compactness result \cite[s.~8, Thm.~5]{simon}.

As $\muOh\tp$ is uniformly bounded in $L^2\trkla{0,T;H^1{\trkla{\Omega}}}$, the trace theorem provides an additional uniform bound in $L^2\trkla{0,T;H^{1/2}\trkla{\Gamma}}$.
Consequently, there is a subsequence of $\{\muOh\tp \vert_{\Sigma_T}\}$ converging weakly towards some limit function $\nu\in L^2\trkla{0,T;H^{1/2}\trkla{\Gamma}}$. The identification of $\nu$ with $\muO \vert_{\Sigma_T}$ follows from similar arguments as in \cite{Metzger2019}, while the remaining convergence property stated in \eqref{eq:conv:mudiff} follows from \eqref{eq:bounds:1.5}.
\end{proof}
In \eqref{eq:fe:final}, we now consider test functions $w\h,\,\eta\h\in L^2\trkla{0,T;\UhO}$ and $z\h\in L^2\trkla{0,T;\UhG}$.  Then, summing over $n = 0, \dots, N-1$, we see that the time-interpolation functions satisfy
\begin{subequations}\label{eq:timecont}
\begin{alignat}{2}
\label{eq:timecont:1}
& \int_{Q_T} \Ih{\partial_tu\h\tl w\h} \dx \dt +\beta^{-1}\int_{\Sigma_T} \IhG{\partial_tu\h\tl w\h} \dG \dt \\
\notag & \quad  +\int_{Q_T} \nabla \muOh\tp\cdot\nabla w\h \dx \dt +\beta^{-1}\int_{\Sigma_T} \gradg\muGh\tp\cdot\gradg w\h \dG \dt=0\,, \\
\label{eq:timecont:2}
& \tfrac{L}{L+1}\int_{\Sigma_T} \IhG{\partial_t u\h\tl z\h} \dG \dt +\tfrac{L}{L+1}\int_{\Sigma_T} \gradg\muGh\tp\cdot\gradg z\h \dG \dt \\
\notag & \quad  +\beta \tfrac1{L+1}\int_{\Sigma_T} \IhG{\trkla{\beta\muGh\tp-\muOh\tp}z\h} \dG \dt =0\,, \\
\label{eq:timecont:3}
& \int_{Q_T} \Ih{\muOh\tp\eta\h} \dx \dt +\int_{\Sigma_T} \IhG{\muGh\tp\eta\h}  \dG \dt\\
\notag & \quad = \int_{Q_T} \nabla u\h\tp\cdot\nabla \eta\h \dx \dt + \int_{Q_T} \Ih{\trkla{F^\prime_1\trkla{u\h\tp}+F^\prime_2\trkla{u\h\tm}}\eta\h} \dx \dt \\
\notag & \qquad +\kappa \int_{\Sigma_T} \gradg u\h\tp\cdot\gradg\eta\h \dG \dt +\int_{\Sigma_T} \IhG{\trkla{G^\prime_1\trkla{u\h\tp}+G^\prime_2\trkla{u\h\tm}}\eta\h} \dG \dt.
\end{alignat}
\end{subequations}
We aim to pass to the limit $(h,\tau,L) \to (0,0,L_0)$ to deduce the convergence of our numerical solutions.

\begin{thm}[Convergence of numerical solutions]\label{thm:limit}
	 Under \eqref{item:disc:time}, \eqref{item:disc:space}, \eqref{item:disc:gamma}, \eqref{ass:pot},
	\eqref{item:potentialsbounds}, 
	and \eqref{item:htau},  the limit triplet $\trkla{u,\muO,\muG}$ obtained from Lemma \ref{lem:disc:conv} by passing to the limit $\trkla{h,\tau,L}\rightarrow\trkla{0,0,L_0}$ solves \eqref{CH:INT} in the following weak sense:
	\begin{subequations}
		\begin{alignat}{2}
		\label{eq:limit1:1} 
		& \beta \int_{Q_T} \trkla{u_0-u}\partial_t w  + \nabla \muO\cdot\nabla w \dx \dt  + \int_{\Sigma_T} (u_0-u)  \partial_t w  + \gradg\muG\cdot\gradg w \dG \dt =0\,,
		\\[1ex]
		\label{eq:limit1:2}
		&\left\{
		\begin{aligned}
		&\int_{\Sigma_T} L_0 (u_0 -u )\partial_t z + L_0 \gradg\muG\cdot\gradg z + \beta \trkla{\beta\muG-\muO  }z \dG \dt =0\, 
		&& \text{ if } L_0 \in [0,\infty),  \\
		& \int_{\Sigma_T} (u_0  -u ) \partial_t z + \gradg\muG\cdot\gradg z \dG \dt = 0 \, 
		&& \text{ if } L_0 = \infty, 
		\end{aligned} 
		\right.
		\\[1ex]
		\label{eq:limit1:3} 
		&\begin{aligned}
		& \int_{Q_T} \mu \eta \dx \dt + \int_{\Sigma_T} \muG \eta \dG \dt \\
		& \quad = \int_{Q_T}  \nabla u \cdot \nabla \eta +  F^\prime(u) \eta \dx \dt + \int_{\Sigma_T} \kappa  \gradg u \cdot \gradg \eta +  G^\prime(u) \eta \dG \dt
		\end{aligned}
		\end{alignat} 
	\end{subequations}
	for all $w\in H^1\trkla{0,T;\V}$ satisfying $w\trkla{\cdot,T}\equiv0$, $z\in H^1\trkla{0,T;H^1\trkla{\Gamma}}$ satisfying $z\trkla{\cdot,T}\equiv0$, and $\eta\in L^2\trkla{0,T;\V}$.  
\end{thm}

\begin{proof}
For an arbitrary $w\in C^1\trkla{\tekla{0,T};C^\infty\trkla{\overline{\Omega}}}$ with $w(\cdot,T) \equiv 0$, we denote its interpolation as $w\h:=\Ih{w}$, which allows us to interchange the interpolation and the trace operator, i.e., $(\Ih{w}) \vert_{\Sigma_T} =\IhG{ w \vert_{\Sigma_T}}$.  For the first term in \eqref{eq:timecont:1}, we obtain
\begin{align*}
\int_{Q_T} \Ih{\partial_t u\h\tl w\h } \dx \dt =\int_{Q_T} \partial_t u\h\tl w\h \dx \dt -\int_{Q_T}\trkla{I-\Ihop}\tgkla{\partial_t u\h\tl w\h} \dx \dt=:A_1+A_2\,.
\end{align*}
Integrating by parts, applying the fact $w\h \to w$ in $L^2(0,T;H^1(\Omega))$ (see \cite{BrennerScott}) and \eqref{eq:initialcond:conv}, we obtain
\begin{align*}
A_1=-\int_{Q_T} u\h\tl\partial_t w\h \dx \dt - \iOmega u\h^0w\h\trkla{\cdot ,0} \dx & \rightarrow -\int_{Q_T} u\partial_t w \dx \dt-\iOmega u_0 w\trkla{\cdot,0} \dx \\
& =\int_{Q_T} \trkla{u_0-u}\partial_t w \dx \dt \,.
\end{align*}
Concerning $A_2$, we employ Lemma \ref{lem:ihfe} with assumption \eqref{item:htau} to obtain
\begin{align}
\tabs{A_2}\leq C\tfrac{h^2}{\sqrt{\tau}} \tfrac{1}{\sqrt{\tau}}\norm{\nabla u\h\tp-\nabla u\h\tm}_{L^2\trkla{0,T;L^2\trkla{\Omega}}}\norm{\nabla w\h}_{L^2\trkla{0,T;L^2\trkla{\Omega}}}\rightarrow0\,,
\end{align}
due to the uniform estimates in \eqref{eq:bounds:1}.
Similar arguments provide the convergence of the second term in \eqref{eq:timecont:1}, and the convergence of the remaining terms in \eqref{eq:timecont:1} follows from \eqref{eq:conv:weak:muO}, \eqref{eq:conv:weak:muG}, and the strong convergences $\Ih{w}\rightarrow w$ in $L^2(0,T;H^1(\Omega))$ and $\IhG{w\vert_{\Sigma_T}}\rightarrow w\vert_{\Sigma_T}$ in $L^2(0,T;H^1(\Gamma))$.  Hence, we recover \eqref{eq:limit1:1} in the limit $(h, \tau, L) \to (0,0,L_0)$.

We refer the reader to \cite[Proof of Thm.~4.4]{Metzger2019} for the arguments to pass to the limit in \eqref{eq:timecont:3} to recover \eqref{eq:limit1:3}, as the equation treated there is identical to \eqref{eq:timecont:3}.  Hence, to finish the proof, it remains to pass to the limit in \eqref{eq:timecont:2}.  For arbitrary $z\in C^1\trkla{\tekla{0,T};C^\infty\trkla{\Gamma}}$ satisfying $z\trkla{\cdot,T}\equiv0$, we consider the interpolation function $z_h:=\IhG{z}$.  Then, for the case $L_0 \in (0,\infty)$ the first two terms in \eqref{eq:timecont:2} can be treated with analogous arguments used above.  Meanwhile for the third term in \eqref{eq:timecont:2}, we see that 
\begin{align}\label{eq:conv:tmp:1}
& \tfrac1{L+1}\int_{\Sigma_T} \beta \IhG{\trkla{\beta\muGh\tp- \muOh\tp\vert_{\Sigma_T} }z\h} \dG \dt \\
\notag & \quad  =  \tfrac1{L+1}\int_{\Sigma_T} \beta \trkla{\beta\muGh\tp-\muOh\tp\vert_{\Sigma_T} }z\h \dG \dt - \tfrac1{L+1}\int_{\Sigma_T} \beta \trkla{I-\IhGop}\tgkla{\trkla{\beta\muGh\tp-\muOh\tp\vert_{\Sigma_T} }z\h} \dG \dt\,.
\end{align}
For the first term on the right-hand side, thanks to \eqref{eq:conv:weak:tracemuO} and \eqref{eq:conv:weak:muG} we find that  
\begin{align*}
\tfrac{1}{L+1} \int_{\Sigma_T} \beta \trkla{\beta\muGh\tp-\muOh\tp\vert_{\Sigma_T} }z\h \dG \dt \rightarrow \tfrac{1}{L_0+1} \int_{\Sigma_T} \beta \trkla{\beta\muG-\muO \vert_{\Sigma_T} }z \dG \dt,
\end{align*}
while for the second term on the right-hand side, by \eqref{eq:bounds:1.5}, Lemma \ref{lem:ihfe} and a standard inverse estimate  $\norm{\gradg a\h}_{L^2\trkla{\Gamma}} \leq C h^{-1} \norm{a \h}_{L^2\trkla{\Gamma}}$  (see e.g., \cite[Thm.~4.5.11]{BrennerScott}), we have
\begin{align*}
& \tfrac{1}{L+1}\abs{\int_{\Sigma_T} \beta \trkla{I-\IhGop}\tgkla{\trkla{\beta\muGh\tp-\muOh\tp \vert_{\Sigma_T}}z\h} \dG \dt}\\
& \quad \leq C\tfrac{L}{L+1}  h L^{-1}\norm{\beta\muGh\tp-\muOh\tp\vert_{\Sigma_T}}_{L^2\trkla{0,T;L^2\trkla{\Gamma}}}\norm{\gradg z\h}_{L^2\trkla{0,T;L^2\trkla{\Gamma}}}\leq Ch \to 0.
\end{align*}
Therefore, the second term on the right-hand side of \eqref{eq:conv:tmp:1} vanishes for all $L\geq0$ as $h\searrow0$.
Hence, for the case $L_0 \in (0,\infty)$, passing to the limit $(h,\tau, L) \to (0,0,L_0)$ in \eqref{eq:timecont:2} yields \eqref{eq:limit1:2}.
For the case $L_0 = 0$, the uniform estimate \eqref{eq:disc:bounds} imply the first two terms in \eqref{eq:timecont:2} converge to zero in the limit, and thus we obtain from passing to the limit in \eqref{eq:conv:tmp:1} the identity
\begin{align*} 
\int_{\Sigma_T} (\beta \muG - \mu \vert_{\Sigma_T}) z \dG \dt = 0,
\end{align*}
which is \eqref{eq:limit1:2} with $L_0 = 0$.  For the case $L_0 = \infty$, we multiply \eqref{eq:timecont:2} with $\frac{L+1}{L}$, leading to 
\begin{align*}
\int_{\Sigma_T} \IhG{\partial_t u\h\tl z\h} + \gradg\muGh\tp\cdot\gradg z\h \dG \dt  + \beta \tfrac1{L}\int_{\Sigma_T} \IhG{\trkla{\beta\muGh\tp-\muOh\tp}z\h} \dG \dt =0,
\end{align*}
and passing to the limit $(h,\tau,L) \to (0,0,\infty)$ yields \eqref{eq:limit1:2}.  Hence, \eqref{eq:limit1:1}-\eqref{eq:limit1:2} hold for all $w\in C^1\trkla{\tekla{0,T};C^\infty\trkla{\overline{\Omega}}}$ satisfying $w\trkla{\cdot,T}\equiv0$, $z\in C^1\trkla{\tekla{0,T};C^\infty\trkla{\Gamma}}$ satisfying $z\trkla{0,T}\equiv0$, and the proof is complete after employing the density of $C^1\trkla{\tekla{0,T};C^\infty\trkla{\overline{\Omega}}}$ in $H^1\trkla{0,T;\V}$ and the density of $C^1\trkla{\tekla{0,T};C^\infty\trkla{\Gamma}}$ in $H^1\trkla{0,T;H^1(\Gamma)}$. 
\end{proof}

\begin{cor}\label{cor:limit}
	 Suppose that \eqref{item:disc:time}, \eqref{item:disc:space}, \eqref{item:disc:gamma}, \eqref{ass:pot},
	\eqref{item:potentialsbounds} 
	and \eqref{item:htau} hold,  and let $\trkla{u,\muO,\muG}$ denote the triplet obtained by Lemma \ref{lem:disc:conv}. 
	\begin{enumerate}[label=$(\alph*)$, ref = $(\alph*)$]
		\item If $L_0=0$, then $\trkla{u,\muO,\muG}$ is a weak solution of the GMS model \eqref{CH:GMS} in the sense of Proposition \ref{prop:GMS}.
		\item If $0<L_0<\infty$, then $\trkla{u,\muO,\muG}$ is a weak solution of the  reaction rate dependent model  \eqref{CH:INT} in the sense of Theorem \ref{THM:WP} (with $L=L_0$).
		\item If $L_0=\infty$, then $\trkla{u,\muO,\muG}$ is a weak solution of the LW model \eqref{CH:LW} in the sense of Proposition \ref{prop:lw}.
	\end{enumerate}	
\end{cor}

\begin{proof} As the proof is rather straightforward, we merely sketch the most important steps.	
	\subparagraph{The case $L_0=0$.} 
	By the definition of a weak derivative, we infer that $\delt u$ exists and belongs to $L^2(0,T;\V')$. In particular, testing \eqref{eq:limit1:1} with $\tilde w = w \varphi$ where $w\in \V$ and $\varphi\in H^1((0,T))$ are arbitrary, we can use the fundamental theorem of calculus of variations to conclude that
	\begin{align*}
		\inn{\delt u}{w}_{\V,\beta} +  \beta  \intO \nabla \muO\cdot\nabla w \dx + \intG \gradg\muG\cdot\gradg w \dG = 0
	\end{align*} 
	for all $w\in\V$. This verifies \eqref{WF:GMS:1} whereas \eqref{WF:GMS:2} follows immediately from \eqref{eq:limit1:3}. Furthermore, we obtain from \eqref{eq:limit1:2} that $\mu\vert_{\Sigma_T} =\beta\theta$ a.e.~on $\Sigma_T$. From $\delt u \in L^2(0,T;\V')$ and \eqref{REG:LIM},
	we deduce that \eqref{REG:GMS} holds where the conditions $u\in C([0,T];L^2(\Om))$ and $u\vert_\Ga \in C([0,T];L^2(\Ga))$ can be obtained a posteriori by the Aubin--Lions lemma. The energy inequality \eqref{WF:GMS:DISS} can be verified similarly to Step 6 of the proof of Theorem \ref{THM:WP}. This implies that $\trkla{u,\muO,\muG}$ is indeed a weak solution to the system~\eqref{CH:GMS}.
	
	\subparagraph{The case $0<L_0<\infty$.} 
	We proceed similarly as in the case $L_0=0$. Here we use both \eqref{eq:limit1:1} and \eqref{eq:limit1:2} to infer that $\delt u \in L^2(0,T;H^1(\Om)')$ and $\delt u\vert_{\Sigma_T} \in L^2(0,T;H^1(\Ga)')$ with 
	\begin{align*}
	\inn{\delt u}{ w}_{H^1(\Omega)} &= - \intO \grad\mu \cdot \grad w \dx 
	+  \intG \tfrac 1 {L_0}(\beta\theta-\mu)  w \dG, \\
	\inn{\delt u}{ z}_{H^1(\Gamma)} &= - \intG \gradg\theta \cdot \gradg z \dG 
	- \intG \tfrac 1 {L_0} (\beta\theta-\mu) \beta z \dG
	\end{align*}
	for all test functions $w\in\V$ and $z\in H^1(\Ga)$. By a density argument, the first line remains valid for all $w\in H^1(\Om)$. Along with \eqref{eq:limit1:3}, we conclude that the weak formulation \eqref{WF:INT} is satisfied. The regularity condition \eqref{REG:INT} follows from $\delt u \in L^2(0,T;H^1(\Om)')$, $\delt u\vert_{\Sigma_T} \in L^2(0,T;H^1(\Ga)')$, and \eqref{REG:LIM}. 
	We point out that the H\"older regularities can be obtained a posteriori by proceeding as in Step 4 of the proof of Theorem \ref{THM:WP}. The energy inequality \eqref{WF:INT:DISS} can be verified by following the line of argument in Step 6 of the proof of Theorem \ref{THM:WP}. This proves that $\trkla{u,\muO,\muG}$ is a weak solution to the system \eqref{CH:INT}.
	
	\subparagraph{The case $L_0=\infty$.} 
	The assertion can be established similarly to the approach in the case $0<L_0<\infty$. Therefore, we do not present the details.
	
	\medskip
	
	\noindent Thus, the proof is complete.
\end{proof}

\begin{remark}
The accuracy of discrete scheme discussed in this section can be improved by approximating the derivative of the polynomial double-well potential by a difference quotient and replacing $\iOmega \nabla u\h\nn\cdot\nabla\eta\h\,dx$ and $\iGamma\gradg u\h\nn\cdot\gradg\eta\h\,d\Gamma$ in \eqref{eq:fe:3} by
\begin{align*}
\alpha\iOmega \nabla u\h\nn\cdot\nabla\eta\h\,dx &+\trkla{1-\alpha}\iOmega \nabla u\h\no\cdot\nabla\eta\h\,dx\\
\text{ and } \hat{\alpha}\iGamma\gradg u\h\nn\cdot\gradg\eta\h\,d\Gamma&+\trkla{1-\hat{\alpha}}\iGamma\gradg u\h\no\cdot\gradg\eta\h\,d\Gamma\,,
\end{align*}
with $\alpha,\hat{\alpha}\in(0.5,1]$. For more details we refer the reader to \cite{GrunGuillenMetzger2016,Metzger2018}. 
\end{remark}

\FloatBarrier
\section{Simulations}
In this section, we investigate the convergence of discrete solutions for $L\nearrow\infty$ and $L\searrow0$ numerically.
The discrete scheme proposed in the last section is implemented in the \texttt{C++} framework EconDrop (cf. \cite{GrunKlingbeil,Campillo2012,GrunGuillenMetzger2016, Metzger2018,Metzger_2018b}).
In principle, this framework allows for adaptivity in space and time using the ideas presented in \cite{GrunKlingbeil}, i.e., we are able to use meshes with a high resolution in the evolving interfacial area and a lower resolution in the bulk phases where $u\approx\pm1$. 
Similarly, the time increments can be varied such that they are small, when the solution changes rapidly and larger when the solution is almost stationary.
However, as we are interested in the dependence on $L$ and therefore want to omit any additional effects which might be introduced by adaptivity, we choose to use a fixed time increment and mesh.

We consider the domain $\Omega:=\rkla{0,1}^2\subset\R^2$ and place an elliptical shaped droplet with with barycenter at $\trkla{0.1,0.5}$, a maximal horizontal elongation of $0.6814$, and a maximal vertical elongation of $0.367$ (see Figure~\ref{fig:initialcondition}).
The domain $\Omega$ is discretised using a triangulation $\Th$ with $h=\sqrt{2}\cdot 2^{-8}$, which provides a partition of $\Gamma$ into elements of length $2^{-8}$.
This corresponds to $\dim\UhO=66049$ and $\dim\UhG=1024$.
Choosing $F$ and $G$ of the form \eqref{eq:penalisedW} with $\tfrac1{\delta^\prime}=250$ and the remaining parameters as specified in Table~\ref{tab:params}, we simulate the behaviour of the droplet from $t=0$ to $t=T=0.05$ using a fixed time increment $\tau=6\times 10^{-7}$.
 The discretization parameters $h$ and $\tau$ used for the presented simulations are chosen very small, as we are interested in the convergence with respect to the parameter $L$ and therefore want to reduce the impact of the spatial and temporal approximations.

The evolution of the droplet is visualised in Figure~\ref{fig:evolution} for different values of $L$.
The corresponding evolution of $\iOmega u\,dx$ and $\iGamma u \,d\Gamma$ are plotted in Figure~\ref{fig:conservation}.
In the case $L=\infty$, the integral of $u$ is conserved in $\Omega$ and on $\Gamma$ individually (cf.~the red, continuous line in Figure~\ref{fig:conservation}). 
Therefore, the contact area in this case can not change.
However, the elliptical droplet still tries to attain circular shape with constant mean curvature (cf.~Figure~\ref{fig:evolution}(a)).
For $L<\infty$, the individual conservation is relaxed (see Figure~\ref{fig:conservation}) to $\beta\iOmega u\,dx+\iGamma u \,d\Gamma$, which allows the contact area to grow ($\iGamma u \,d\Gamma$ is increasing in Figure~\ref{fig:conservation}(b)), while the droplet's bulk volume is decreasing (cf.~Figure~\ref{fig:conservation}(a)).
The effect intensifies for decreasing $L$ (cf.~Figures \ref{fig:evolution}(b)-\ref{fig:evolution}(e)), i.e., for larger reaction rates.
However, we want to emphasise that in this scenario our implementation allows for a perfect conservation of $\beta\iOmega u\,dx+\iGamma u \,d\Gamma$.

According to \eqref{INT:NRG}, the total free energy $E=E_\text{bulk}+E_\text{surf}$ is non-increasing over time.
In the two-dimensional scenario discussed in this section, the boundary $\Gamma$ is only one-dimensional and the interface given as the zero level set of $u$ cuts $\Gamma$ always in two points.
Therefore the surface free energy $E_\text{surf}$ depends mainly on the profile of $u\vert_{\Sigma_T}$ in the transition regions.
However, as the optimal $u$-profile in the transition region is given by a hyperbolic tangent which attains the values $\pm1$ only infinitely far away from the zero level set of the phase-field variable, the length of the section covered by the droplet might still have a small influence on the surface free energy.
As we start with an interface profile which is close to the stationary one, we expect only little changes in $E_\text{surf}$.
The bulk free energy $E_\text{bulk}$, however, depends mainly on the droplet's surface area.
As the initial surface area is not minimal, we can expect a significant decrease in $E_\text{bulk}$. 
The evolution of the energy is plotted in Figure~\ref{fig:energy} for several values of $L$.
As expected, the bulk free energy and the total free energy depicted in Figure~\ref{fig:energy}(a) and \ref{fig:energy}(c) decrease over time.
Comparing the evolution of $E_\text{bulk}$ for different values of $L$, we notice that after an initial drop which occurs for all $L$, the further evolution of the energy depends significantly on $L$.
In the case $L=\infty$, the rate of energy decrease diminishes and $E_\text{bulk}$ attains a stationary value, the energy decrease continues for $L<\infty$.

As the initial shape of the droplet is elliptical, the right tip of the droplet exhibits high curvature and therefore vanishes quickly when the droplet optimises its shape, thus causing the initial energy drop.
As $\iGamma u\,d\Gamma$ is conserved for $L=\infty$, the droplet is not able to decrease its overall surface by increasing the contact area.
Consequently, $E_\text{bulk}$ stagnates in this case.
On the other hand, for $L<\infty$, only $\beta\iOmega u\,dx+\iGamma u\,d\Gamma$ is conserved  and the droplet's surface area can be further decreased by increasing the contact area which results in a further decrease of $E_\text{bulk}$.
As expected, the rate of energy reduction increases with decreasing $L$, while the total energy decrease is bounded by the energetically optimal droplet shape.

While the bulk free energy is decreasing, the surface free energy $E_\text{surf}$ which is depicted in Figure~\ref{fig:energy}(b) increases.
To explain the initial rapid increase in $E_\text{surf}$, we want to point out that our discrete initial condition does not exhibit the optimal $u$-profile in the transition region and that the parameters in this scenario are chosen in a way that the optimal transition profiles in $\Omega$ and on $\Gamma$ differ.
Therefore, optimizing the profile in $\Omega$ to reduce $E_\text{bulk}$ leads to a slight increase in $E_\text{surf}$.
After this initial incline, evolution of $E_\text{surf}$ depends on $L$, as the contact angle determines how the $u$-profile in $\Omega$ influences the profile on $\Gamma$.
It is also worth mentioning that the transition profiles are almost identical.
Figure~\ref{fig:levelsets} shows an overlay of level sets of $u$ for $L=0$ (red) and $L=\infty$ (blue) at $t=T=0.05$.
It is striking that the distances between the depicted level sets for $u=\tgkla{\pm0.9,\pm0.8,\pm0.7,\pm0.6,\pm0.5,\pm0.4,\pm0.3,\pm0.2,\pm0.1,0}$ are completely identical.
However, as $E_\text{surf}$ does not attain the same value for all $L$ at $t=0.05$, the profiles have to differ for $\tabs{u}>0.9$.
This might be a result of the small size of the wetted section of $\Gamma$ for $L=\infty$ and $L=100$ and the resulting interactions between the transition regions.

In order to deduce an experimental order of convergence (EOC) for the phase-field $u$ for $L\searrow0$, we compare discrete solutions $u_{L_i}\in\UhO$ for a decreasing sequence $\tgkla{L_i}$ with the discrete solution $u_*\in\UhO$ obtained for $L=0$ and define the corresponding error as
\begin{align}
\text{err}_{L_i}:=\norm{u_{L_i}-u_*}_{L^2\trkla{0,T;L^2\trkla{\Omega}}}\,.
\end{align}
Here, the time integral is approximated using the trapezoidal rule with time increment $\tilde{\tau}= 1.02 \times 10^{-5}$. 
The experimental order is then defined as
\begin{align*}
\text{EOC}\trkla{L_i}:=\frac{\log\rkla{\tfrac{\text{err}_{L_{i+1}}}{\text{err}_{L_i}}}}{\log\rkla{\tfrac{L_{i+1}}{L_i}}}\,.
\end{align*}
For $L\nearrow\infty$ and the convergence of $u\vert_{\Sigma_T}$, $\muO$, and $\theta$, we proceed analogously.
The results for the convergence of $u$ on $Q_T$ which are collected Table~\ref{tab:convergence:u} indicate that for $L\leq 1\times 10^{-3}$ the convergence rate is almost 1.
A similar pattern emerges for the EOC of $u\vert_{\Sigma_T}$ which is displayed in Table~\ref{tab:convergence:ugamma}, the EOC of $\muO$ displayed in Table~\ref{tab:convergence:muO}, and the EOC of $\muG$ that can be found in Table~\ref{tab:convergence:muG}.

As a last test case, we investigate the behaviour of $\muG$ and $\muO\vert_{\Sigma_T}$ for $L\searrow0$.
According to the theoretical results, $\norm{\beta\muG-\muO\vert_{\Sigma_T}}_{L^2\trkla{\Sigma_T}}\rightarrow0$ with a rate of at least $\sqrt{L}$.
As shown in Table~\ref{tab:potentials}, the numerical errors we obtain in the case $L=0$ for $\norm{\beta\muG-\muO\vert_{\Sigma_T}}_{L^2\trkla{0,T;L^2\trkla{\Gamma}}}$ are only of order $10^{-8}$ and of order $10^{-5}$ if we use the $L^\infty\trkla{0,T;L^\infty\trkla{\Gamma}}$-norm.
Similar to the results described above, our simulations yield an experimental order of convergence rate of $1$ for small values of $L$, but still reach the expected rate of $0.5$ for $L=10$.

\begin{table}
\center
\begin{tabular}{c|c|c|c|c|c}
$\varepsilon$& $\delta$& $\kappa$ & $\mo$ &$\mg$ & $\beta$\\
\hline
0.01&0.02&0.25&1&0.4&4
\end{tabular}
\caption{Parameters used for the simulations presented in this section.}
\label{tab:params}
\end{table}

\begin{figure}
\center
\includegraphics[width=0.3\textwidth]{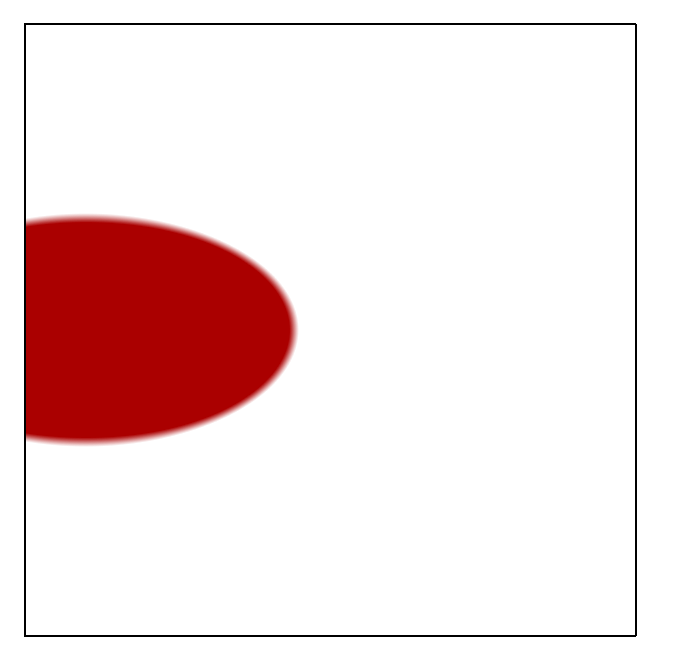}
\caption{Initial datum.}
\label{fig:initialcondition}
\end{figure}

\begin{figure}[p]
\newcommand{\scale}{.23\textwidth}
\subfloat[][$L=\infty$]{
\includegraphics[width=\scale]{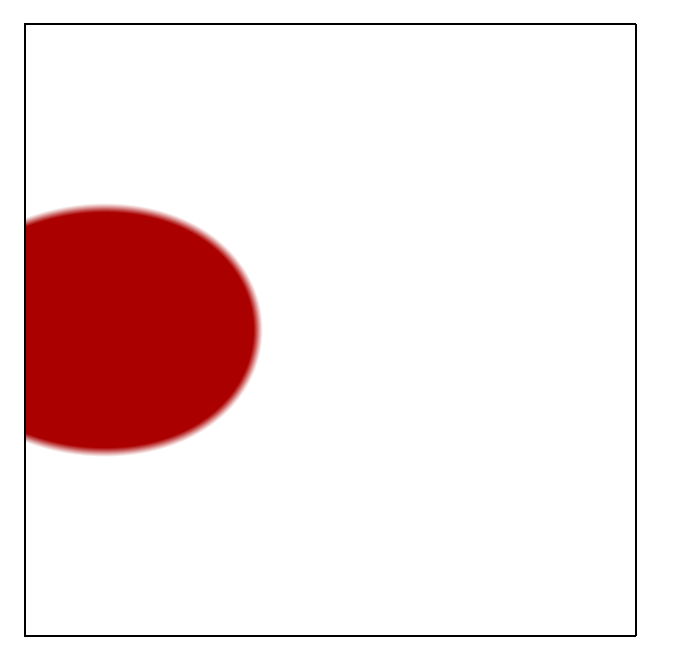}

\includegraphics[width=\scale]{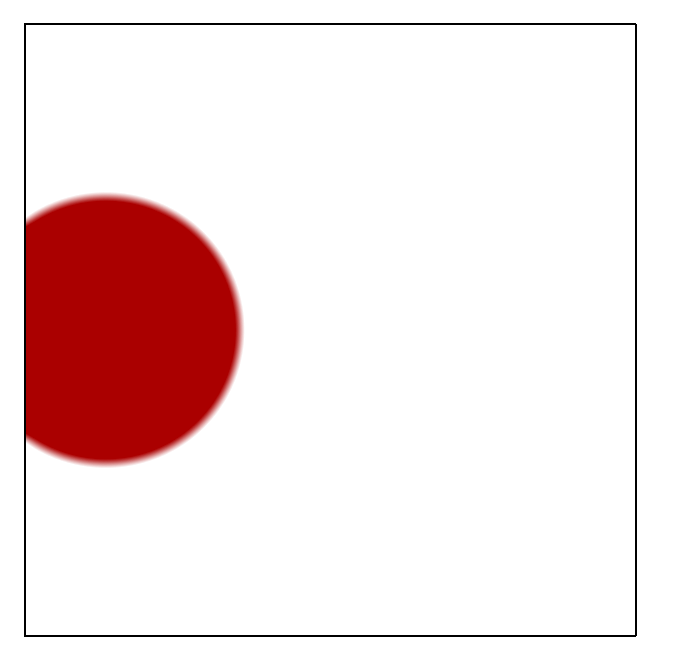}

\includegraphics[width=\scale]{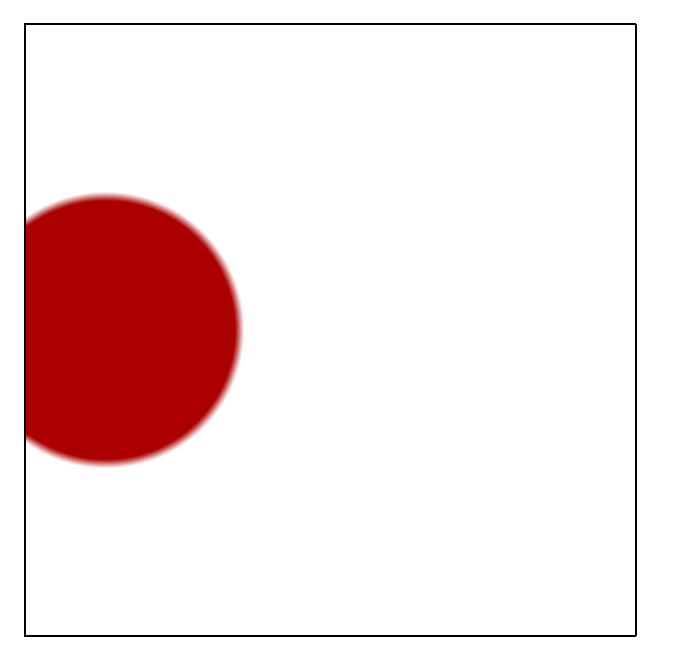}

\includegraphics[width=\scale]{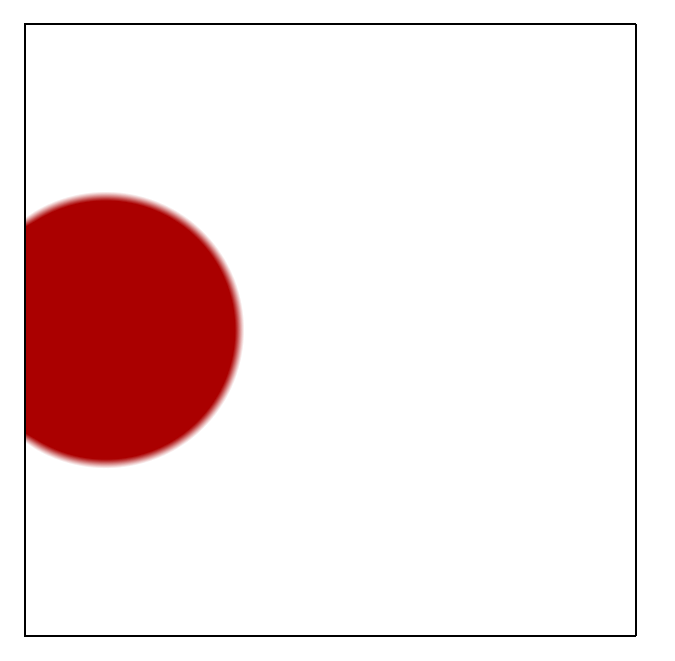}
}\\

\subfloat[][$L=10$]{
\includegraphics[width=\scale]{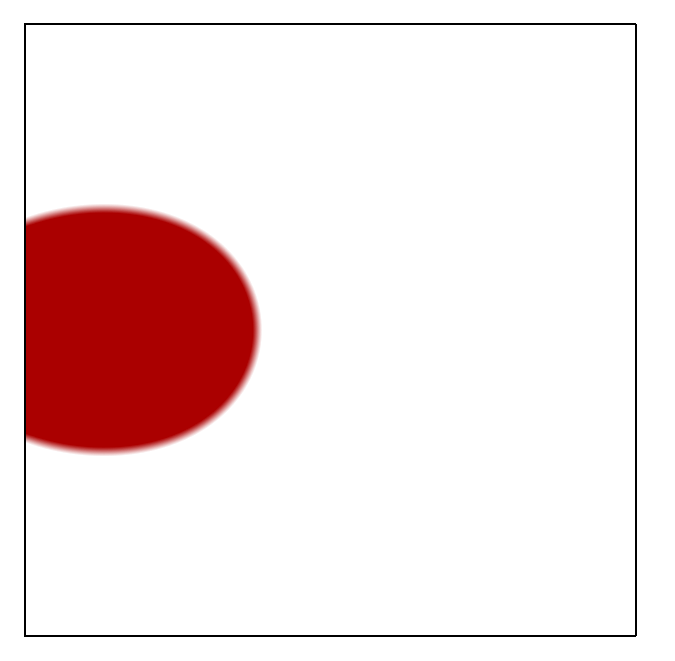}

\includegraphics[width=\scale]{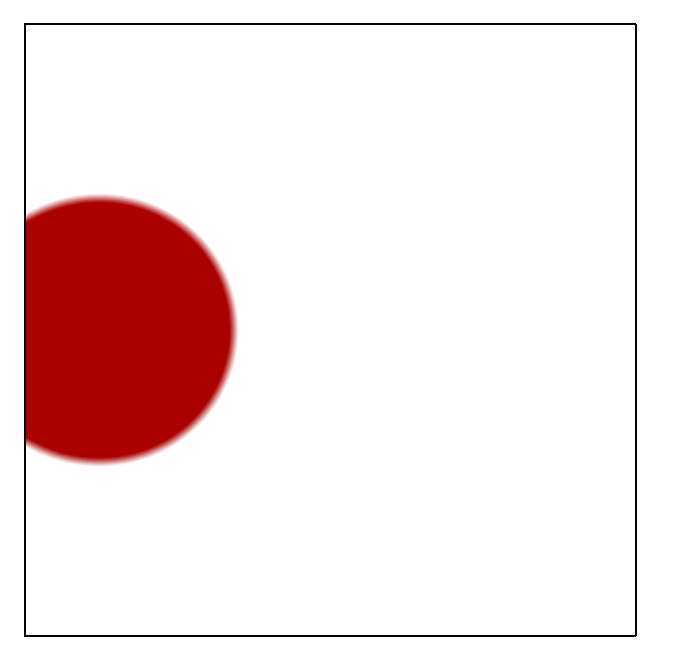}

\includegraphics[width=\scale]{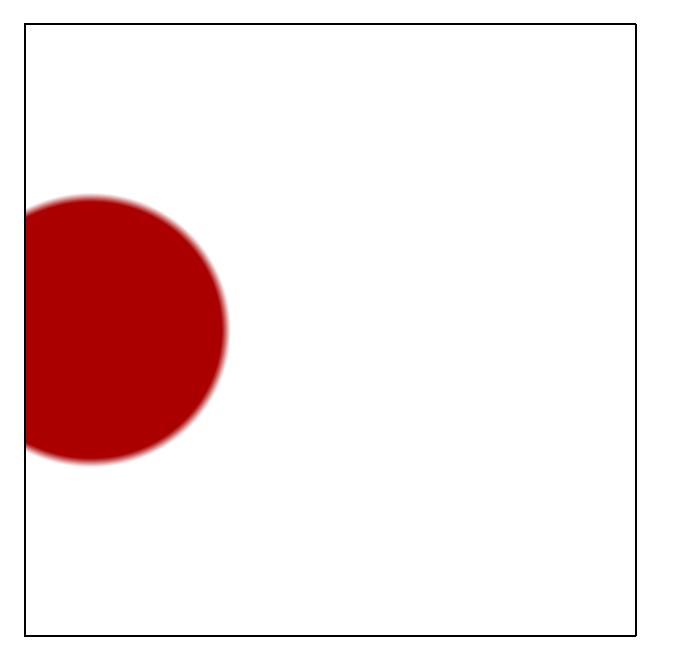}

\includegraphics[width=\scale]{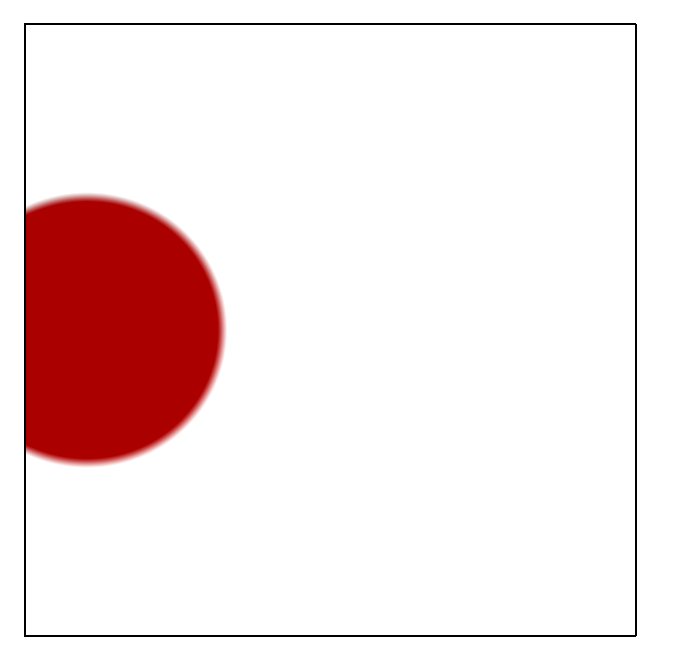}
}\\
\subfloat[][$L=1$]{
\includegraphics[width=\scale]{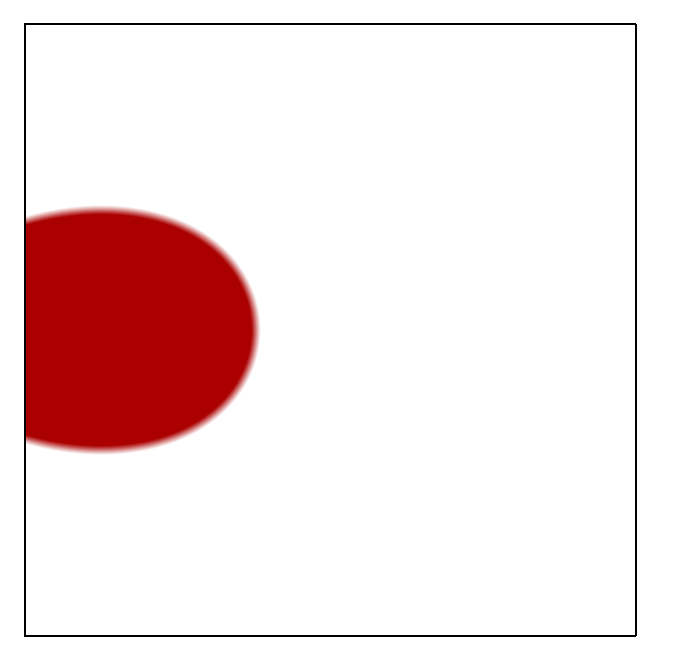}

\includegraphics[width=\scale]{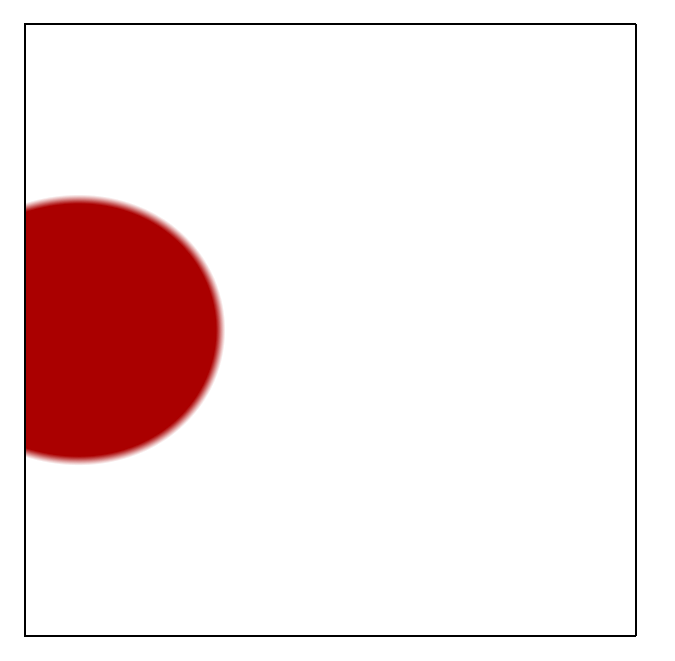}

\includegraphics[width=\scale]{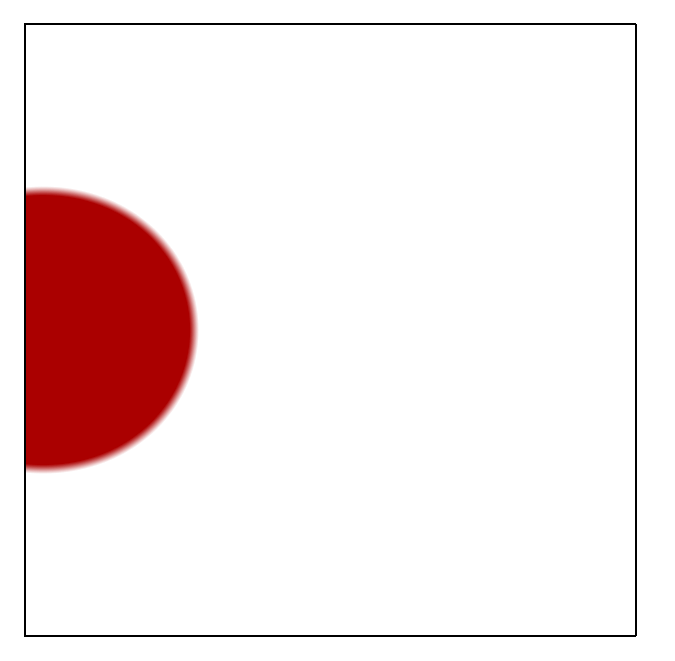}

\includegraphics[width=\scale]{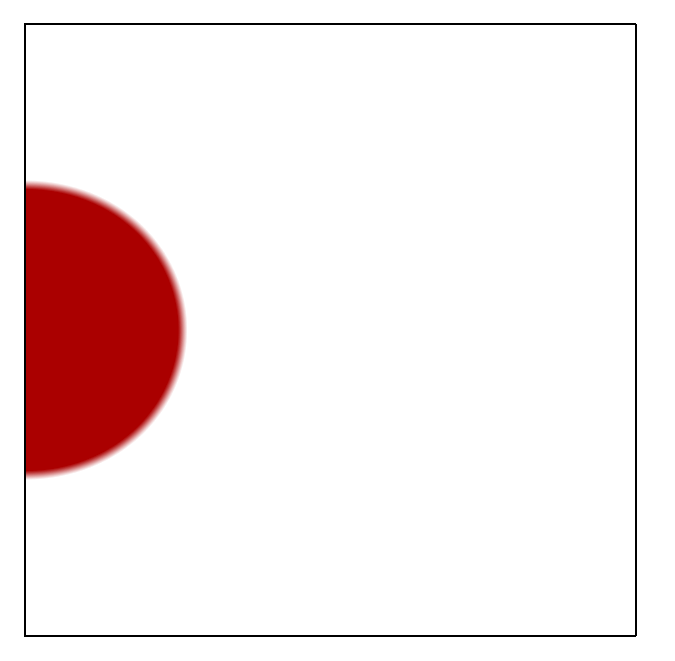}
}\\
\subfloat[][$L=0.1$]{
\includegraphics[width=\scale]{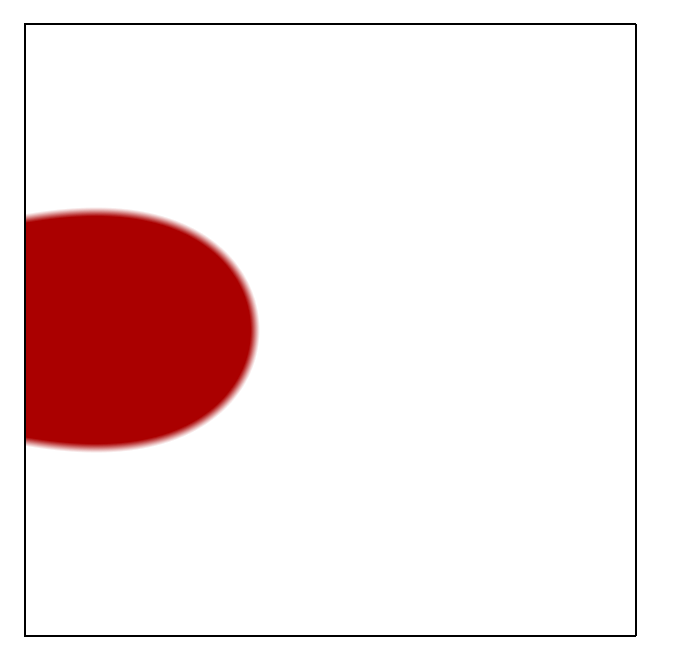}

\includegraphics[width=\scale]{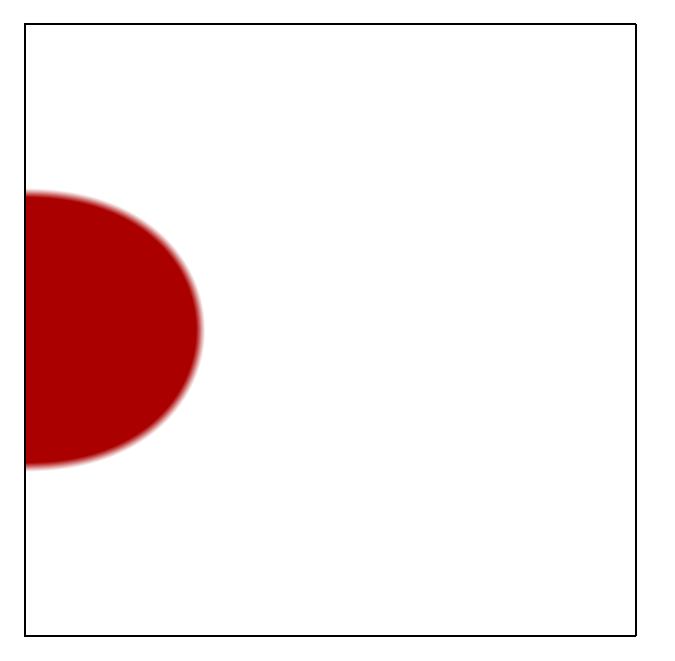}

\includegraphics[width=\scale]{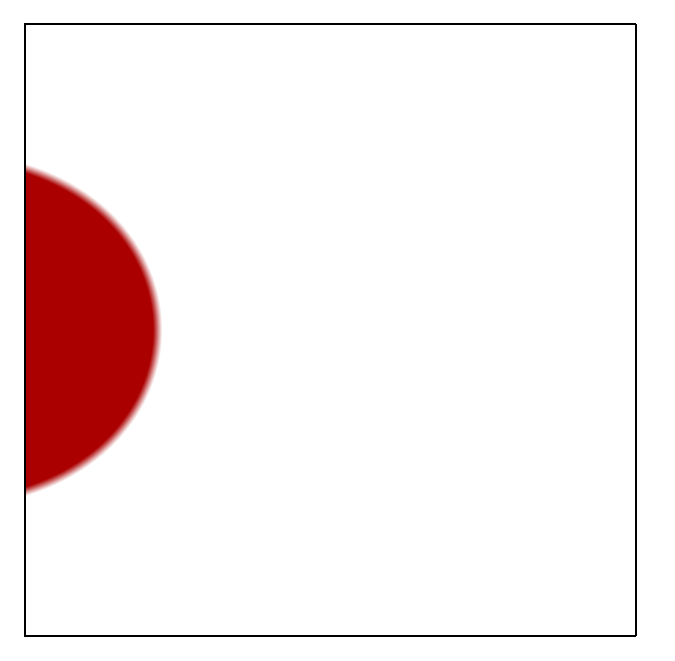}

\includegraphics[width=\scale]{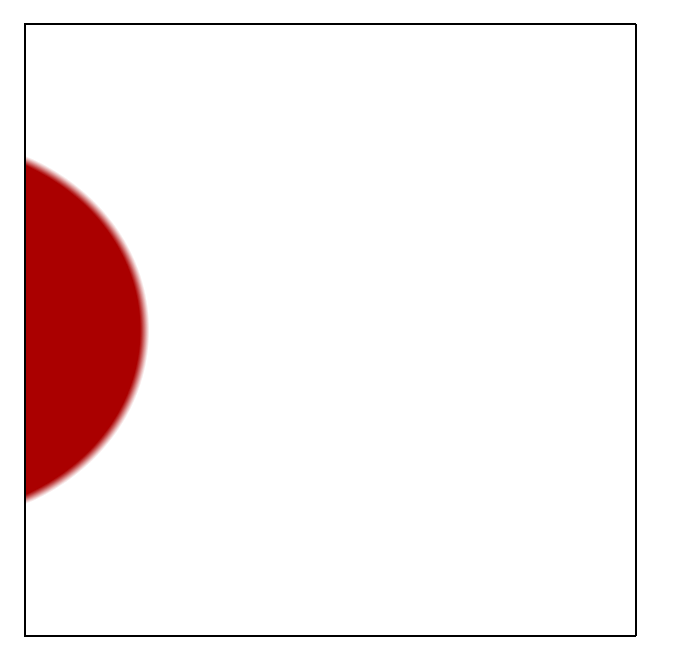}
}\\
\subfloat[][$L=0$]{
\includegraphics[width=\scale]{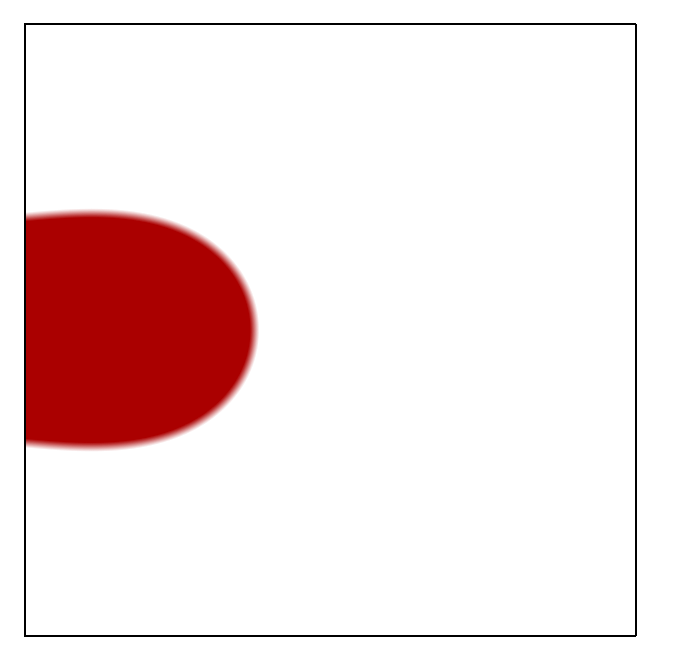}

\includegraphics[width=\scale]{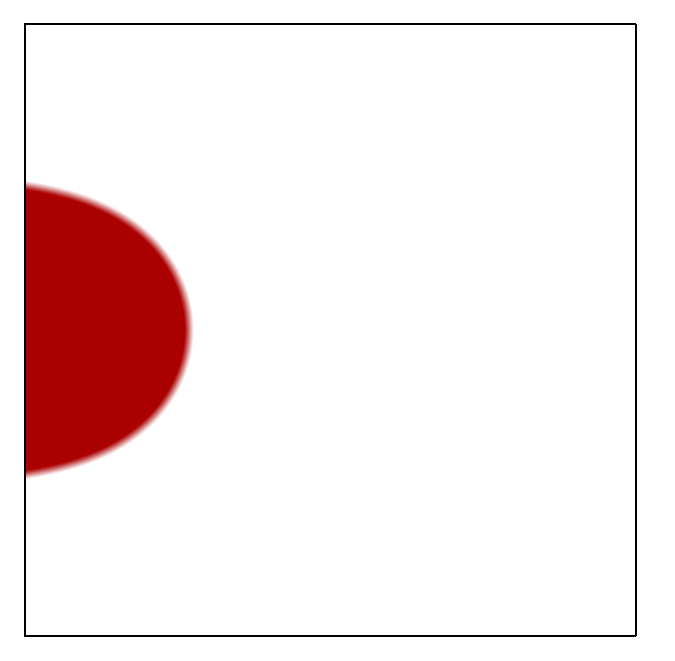}

\includegraphics[width=\scale]{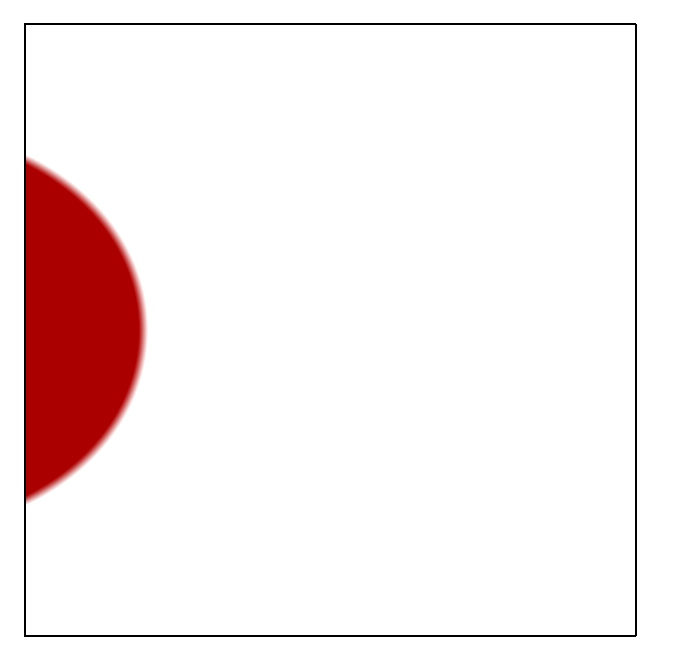}

\includegraphics[width=\scale]{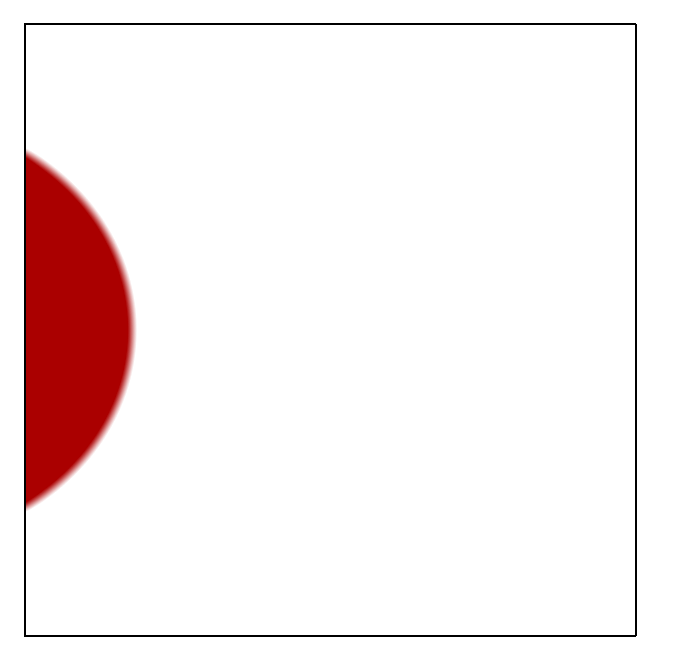}
}
\caption{Phase-field at $t=0.004$, $t=0.02$, $t=0.04$, and $t=0.05$.}
\label{fig:evolution}
\end{figure}

\begin{figure}
\center
\subfloat[][Evolution of $\displaystyle\int_\Omega u\,dx$.]{
\includegraphics[width=0.47\textwidth]{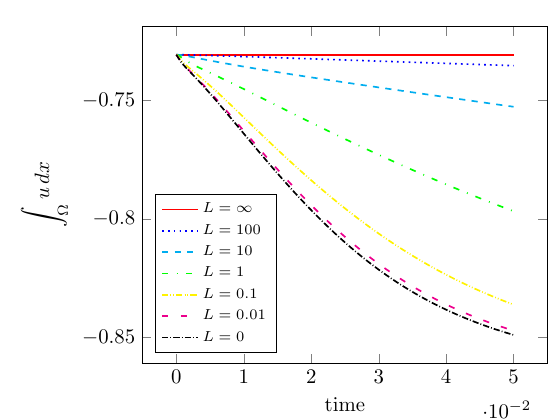}
}
\quad
\subfloat[][Evolution of $\displaystyle\int_\Gamma u\,d\Gamma$.]{
\includegraphics[width=0.47\textwidth]{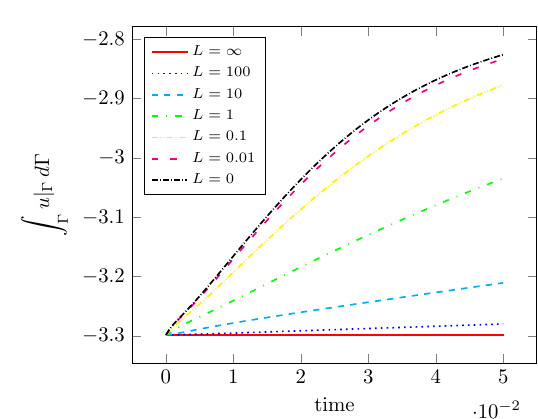}
}
\caption{Time evolution of the bulk mass and the surface mass of $u$.}
\label{fig:conservation}
\end{figure}

\begin{figure}
\center
\subfloat[][Evolution of $\displaystyle E_\text{bulk}\trkla{u}$.]{
\includegraphics[width=0.47\textwidth]{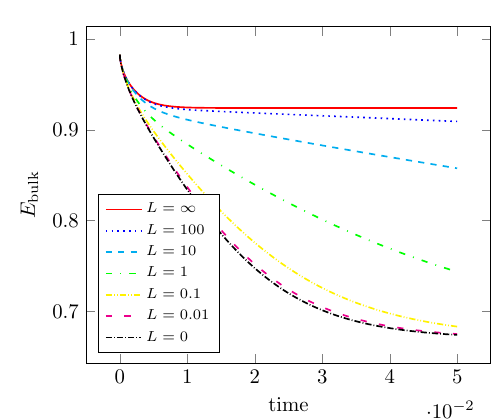}
}
\subfloat[][Evolution of $\displaystyle E_\text{surf}\trkla{u}$.]{
\includegraphics[width=0.47\textwidth]{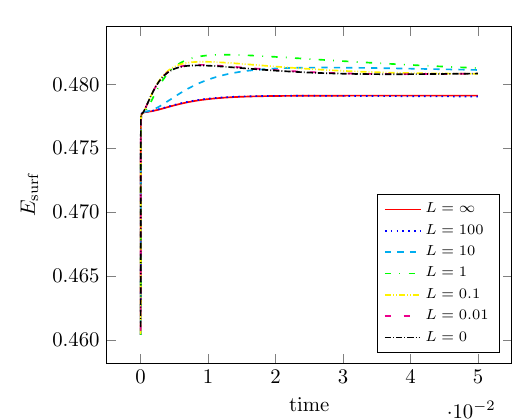}
}
\\
\subfloat[][Evolution of $\displaystyle E(u) = E_\text{bulk}\trkla{u}+E_\text{surf}\trkla{u}$.]{
\includegraphics[width=0.47\textwidth]{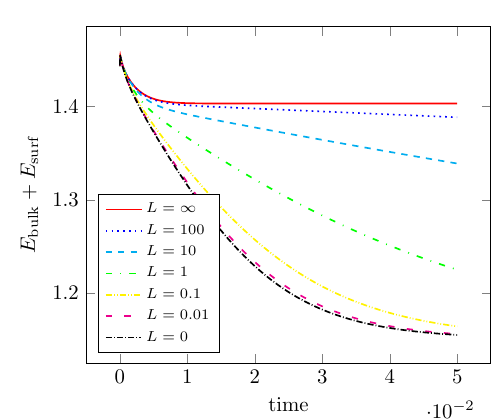}
}
\caption{Evolution of the energy.}
\label{fig:energy}
\end{figure}

\begin{figure}
\center
\includegraphics[width=0.32\textwidth]{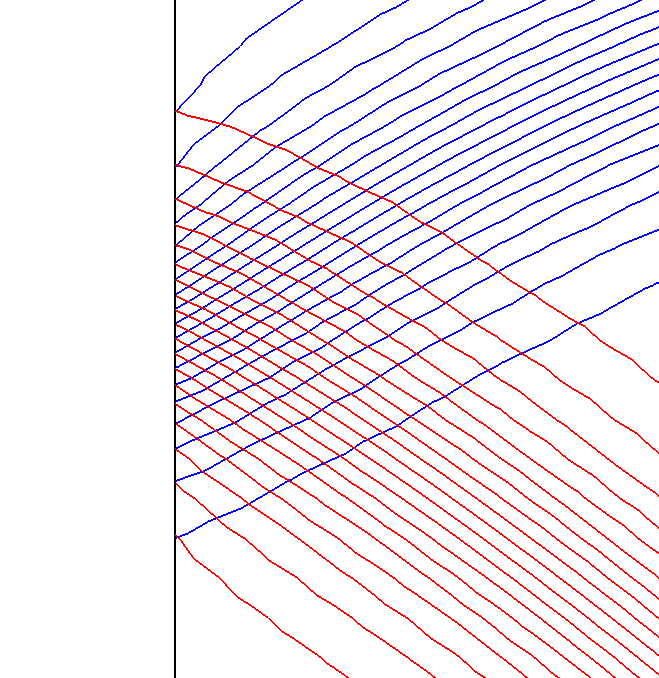}
\caption{Overlay of level sets of $u$ for $L=0$ (red) and $L=\infty$ (blue).}
\label{fig:levelsets}
\end{figure}


\begin{table}[h!]
	\footnotesize
\begin{center}
	\hfill
\begin{tabular}{l|cc}
 \diagbox{$L^{~}$}{$u$}& $\norm{\cdot}_{L^2\trkla{0,T;L^2\trkla{\Omega}}}$ & EOC \\
 \hline  \\[-7pt]
$0.0001$ & 4.01E-05 &  -\\ 
$0.0002$ & 8.02E-05 & 1.00 \\
$0.0003$ & 1.20E-04 & 1.00 \\
$0.0004$ & 1.60E-04 & 1.00 \\
$0.0005$ & 2.00E-04 & 0.99 \\
$0.00075$ & 2.98E-04 & 0.99 \\
$0.001$ & 3.96E-04 & 0.99 \\
$0.01$ & 3.58E-03 & 0.96\\
$0.1$ & 2.16E-02 & 0.78\\
$1$ & 5.66E-02 & 0.42 \\
\end{tabular}
\hfill
\begin{tabular}{l|cc}
 \diagbox{$L^{-1}$}{$u$}& $\norm{\cdot}_{L^2\trkla{0,T;L^2\trkla{\Omega}}}$ & EOC \\
 \hline \\[-7pt]
$0.0001$ & 6.12E-05 & - \\ 
$0.0002$ & 1.22E-04 & 0.99 \\
$0.0003$ & 1.82E-04 & 0.99 \\
$0.0004$ & 2.42E-04 & 0.99 \\
$0.0005$ & 3.01E-04 & 0.98 \\
$0.00075$ & 4.47E-04 & 0.97 \\
$0.001$ & 5.90E-04 & 0.97 \\
$0.01$ & 4.46E-03 & 0.88\\
$0.1$ & 2.06E-02 & 0.66\\
$1$ & 5.32E-02 & 0.41 \\
\end{tabular}
	\hfill $\,$
	\vspace{-15pt}	
\end{center}
\caption{Comparison of the phase-field parameters for different $L$ with the solution for $L=0$ (left) and $L=\infty$ (right).}
\label{tab:convergence:u}
\end{table}

\begin{table}[h!]
	\footnotesize
\begin{center}
	\hfill
\begin{tabular}{l|cc}
 \diagbox{$L^{~}$}{$u\vert_{\Sigma_T}$}& $\norm{\cdot}_{L^2\trkla{0,T;L^2\trkla{\Gamma}}}$ & EOC \\
 \hline
$0.0001$ & 6.78E-05 & - \\ 
$0.0002$ & 1.35E-04 & 1.00 \\
$0.0003$ & 2.03E-04 & 1.00 \\
$0.0004$ & 2.70E-04 & 1.00 \\
$0.0005$ & 3.37E-04 & 0.99 \\
$0.00075$ & 5.04E-04 & 0.99 \\
$0.001$ & 6.70E-04 & 0.99 \\
$0.01$ & 6.05E-03 & 0.96\\
$0.1$ & 3.71E-02 & 0.79\\
$1$ & 1.00E-01 & 0.43 \\
\end{tabular}
\hfill
\begin{tabular}{l|cc}
 \diagbox{$L^{-1}$}{$u\vert_{\Sigma_T}$}& $\norm{\cdot}_{L^2\trkla{0,T;L^2\trkla{\Gamma}}}$ & EOC \\
 \hline
$0.0001$ & 1.07E-04 & - \\ 
$0.0002$ & 2.13E-04 & 0.99 \\
$0.0003$ & 3.18E-04 & 0.99 \\
$0.0004$ & 4.22E-04 & 0.99 \\
$0.0005$ & 5.25E-04 & 0.98 \\
$0.00075$ & 7.79E-04 & 0.97 \\
$0.001$ & 1.03E-03& 0.96 \\
$0.01$ & 7.41E-03 & 0.86\\
$0.1$ & 3.12E-02 & 0.62\\
$1$ & 8.86E-02 & 0.45 \\
\end{tabular}
	\hfill $\,$
	\vspace{-15pt}
\end{center}
\caption{Comparison of the phase-field parameters for different $L$ with the solution for $L=0$ (left) and $L=\infty$ (right).}
\label{tab:convergence:ugamma}
\end{table}

\begin{table}[h!]
\footnotesize
\begin{center}
	\hfill
\begin{tabular}{l|cc}
 \diagbox{$L^{~}$}{$\muO$}& $\norm{\cdot}_{L^2\trkla{0,T;L^2\trkla{\Omega}}}$ & EOC \\
 \hline\\[-7pt]
$0.0001$ & 5.54E-05 & - \\ 
$0.0002$ & 1.11E-04 & 1.00 \\
$0.0003$ & 1.66E-04 & 1.00 \\
$0.0004$ & 2.21E-04 & 0.99 \\
$0.0005$ & 2.75E-04 & 0.99 \\
$0.00075$ & 4.12E-04 & 0.99 \\
$0.001$ & 5.47E-04 & 0.99 \\
$0.01$ & 4.92E-03 & 0.95\\
$0.1$ & 2.99E-02 & 0.78\\
$1$ & 8.41E-02 & 0.45 \\
\end{tabular}\hfill
\begin{tabular}{l|cc}
 \diagbox{$L^{-1}$}{$\muO$}& $\norm{\cdot}_{L^2\trkla{0,T;L^2\trkla{\Omega}}}$ & EOC \\
 \hline\\[-7pt]
$0.0001$ & 6.05E-05 & - \\ 
$0.0002$ & 1.20E-04 & 0.99 \\
$0.0003$ & 1.80E-04 & 0.99 \\
$0.0004$ & 2.38E-04 & 0.98 \\
$0.0005$ & 2.96E-04 & 0.98 \\
$0.00075$ & 4.39E-04 & 0.97 \\
$0.001$ & 5.78E-04 & 0.96 \\
$0.01$ & 4.00E-03 & 0.84\\
$0.1$ & 1.44E-02 & 0.56\\
$1$ & 4.45E-02 & 0.49 \\
\end{tabular}
	\hfill $\,$
	\vspace{-15pt}
\end{center}
\caption{Comparison of the chemical potentials for different $L$ with the solution for $L=0$ (left) and $L=\infty$ (right).}
\label{tab:convergence:muO}
\end{table}

\FloatBarrier

\begin{table}[h!]
	\footnotesize
\begin{center}
	\hfill
\begin{tabular}{l|cc}
 \diagbox{$L^{~}$}{$\muG$}& $\norm{\cdot}_{L^2\trkla{0,T;L^2\trkla{\Gamma}}}$ & EOC \\
 \hline\\[-7pt]
$0.0001$ & 2.72E-05 & - \\ 
$0.0002$ & 5.43E-05 & 0.99 \\
$0.0003$ & 8.13E-05 & 1.00 \\
$0.0004$ & 1.08E-04 & 0.99 \\
$0.0005$ & 1.35E-04 & 0.99 \\
$0.00075$ & 2.02E-04 & 0.99 \\
$0.001$ & 2.68E-04& 0.99 \\
$0.01$ & 2.41E-03 & 0.95\\
$0.1$ & 1.50E-02& 0.79\\
$1$ &  5.17E-02& 0.54 \\
\end{tabular}\hfill
\begin{tabular}{l|cc}
 \diagbox{$L^{-1}$}{$\muG$}& $\norm{\cdot}_{L^2\trkla{0,T;L^2\trkla{\Gamma}}}$ & EOC \\
 \hline\\[-7pt]
$0.0001$ & 1.66E-03 & - \\ 
$0.0002$ & 3.30E-03 & 0.99 \\
$0.0003$ & 4.93E-03 & 0.99 \\
$0.0004$ & 6.54E-03 & 0.98 \\
$0.0005$ & 8.13E-03 & 0.98 \\
$0.00075$ & 1.20E-02 & 0.97 \\
$0.001$ & 1.58E-02 & 0.95 \\
$0.01$ & 1.02E-01 & 0.81\\
$0.1$ & 2.43E-01 & 0.38\\
$1$ & 2.98E-01 & 0.09 \\
\end{tabular}
	\hfill $\,$
	\vspace{-15pt}
\end{center}
\caption{Comparison of the chemical potentials for different $L$ with the solution for $L=0$ (left) and $L=\infty$ (right).}
\label{tab:convergence:muG}
\end{table}

\begin{table}[h!]
\footnotesize
\begin{center}
\hfill
\begin{tabular}{l|cc|c}
 \diagbox{$L^{~}$}{$\beta\muG-\muO\vert_{\Sigma_T}$}& $\norm{\cdot}_{L^2\trkla{0,T;L^2\trkla{\Gamma}}}$ & EOC & $\norm{\cdot}_{L^\infty\trkla{0,T;L^\infty\trkla{\Gamma}}}$ \\
 \hline\\[-7pt]
 $0$ & 1.19E-08 & - &1.17E-05\\ 
 \hline\\[-7pt]
$0.0001$ & 5.93E-05 & - & 4.22E-02 \\ 
$0.0002$ & 1.19E-04 & 1.00 & 8.35E-02 \\
$0.0003$ & 1.78E-04 & 1.00 & 1.24 E-01\\
$0.0004$ & 2.37E-04 & 1.00 & 1.63E-01 \\
$0.0005$ & 2.96E-04 & 1.00 & 2.02E-01\\
$0.00075$ & 4.44E-04 & 1.00 & 2.95E-02\\
$0.001$ & 5.91E-04 & 1.00 & 3.84E-01\\
$0.01$ & 5.76E-03 & 0.99& 2.06E-00\\
$0.1$ & 4.98E-02 & 0.94& 3.77E-00\\
$1$ & 2.97E-01 & 0.78 & 4.12E-00\\
$10$& 9.88E-01 & 0.52 & 4.16E-00\\
$100$& 2.08E-00 & 0.32 & 4.17E-00
\end{tabular}
\hfill $\,$
	\vspace{-15pt}
\end{center}
\caption{Difference between $\beta\muG$ and $\muO\vert_{\Sigma_T}$ for different $L$.}
\label{tab:potentials}
\end{table}

\section{Appendix}

\begin{proof}[Proof of Lemma~\ref{LEM:INT}]
	We prove the assertion by contradiction. To this end, we assume that the estimate is false. This means we can find an $\alpha>0$ such that for any $k\in\N$ there exists a function $u_k \in \Wo$ with
	\begin{align}
	\label{A1}
		\norm{u_k}_{L^2(\Om)}^2 + \norm{u_k}_{L^2(\Ga)}^2  
			> \alpha \norm{\grad u_k}_{L^2(\Omega)}^2 + k \norm{u_k}_{L,\beta,*}^2
	\end{align}
	Now we define a sequence $(\tilde u_k)_{k\in\N}\subset\Vk$ by 
	\begin{align*}
		\tilde u_k := \frac{u_k}{\big(\norm{u_k}_{L^2(\Om)}^2 + \norm{u_k}_{L^2(\Ga)}^2\big)^{1/2}} 
	\end{align*}
	for all $k\in\N$. By this construction, it holds that 
	\begin{align}
		\label{A2}
		\norm{\tilde u_k}_{L^2(\Om)}^2 + \norm{\tilde u_k}_{L^2(\Ga)}^2 = 1
	\end{align}  
	as well as $\tilde u_k \in \Wo$ for all $k\in\N$.
	Moreover, it follows from \eqref{A1} that  
	\begin{align}
	\label{A3}
		\alpha\norm{\grad \tilde u_k}_{L^2(\Omega)}^2 + k\,\norm{\tilde u_k}_{L,\beta,*}^2
		< 1
		\quad\text{for all}\; k\in\N.
	\end{align}
	Consequently, the sequence $(\tilde u_k)$ is bounded in $H^1(\Omega)$. Hence, according to the Banach--Alaoglu theorem, there exists $u\in H^1(\Omega)$ such that $\tilde u_k \wto u$ in $H^1(\Omega)$ along a non-relabelled subsequence. We now deduce from the compact embeddings $H^1(\Omega)\emb L^2(\Omega)$ and $H^1(\Omega)\emb L^2(\Gamma)$ that $\tilde u_k \to u$ in $L^2(\Omega)$ and $\tilde u_k \to u$ in $L^2(\Gamma)$ after another subsequence extraction. In particular, this implies $u\in\Wo \subset \Wod$ and $\norm{u}_{L^2(\Om)}^2 + \norm{u}_{L^2(\Ga)}^2 = 1$. 
	It now follows from \eqref{A3} that
	\begin{align}
		\norm{\SS(\tilde u_k)}_{L,\beta}^2 = \norm{\tilde u_k}_{L,\beta,*}^2 < \frac 1 k \le 1
		\quad\text{for all}\; k\in\N.
	\end{align}
	Hence, the Banach--Alaoglu theorem yields the existence of a function $\SS^*\in \Ho$ such that $\SS(\tilde u_k)\wto \SS^*$ with respect to the inner product $\scp{\cdot}{\cdot}_{L,\beta}$ on $\Ho$ as $k\to \infty$. As $\SS(\tilde u_k)$ is the weak solution of the system \eqref{EQ:LIN} to the right-hand side $\tilde u_k$ we can pass to the limit in the weak formulation (see \eqref{WF:LIN}) to conclude that $\SS^* = \SS(u)$. Since the norm $\norm{\cdot}_{L,\beta}$ on $\Ho$ is weakly lower semicontinuous, we can use \eqref{A3} to obtain
	\begin{align}
		\label{A4}
		\norm{\SS(u)}_{L,\beta} 
		\;\le\; \underset{k\to\infty}{\lim\inf}\; \norm{\SS(\tilde u_k)}_{L,\beta}
		\;\le\; \underset{k\to\infty}{\lim\inf}\; \frac{1}{\sqrt{k}} 
		\;=\; 0
		\quad\text{for all}\; k\in\N.
	\end{align}
	This means that
	\begin{align*}
		\grad S_\Om( u) = 0 \;\;\text{a.e.~in}\;\Om
		\quad\text{and}\quad
		\gradg S_\Ga( u) = 0,\;\;
		\beta S_\Ga( u) - S_\Om( u) = 0 \;\;\text{a.e.~on}\;\Ga.
	\end{align*}
	Finally, as $\SS(u)$ is the weak solution of \eqref{EQ:LIN} (with $\phi=u$), this is enough to conclude that $u = 0$ a.e.~in $\Omega$ and also $u\vert_\Gamma = 0$ a.e.~on $\Gamma$. However, this is a contradiction to $\norm{u}_{L^2(\Om)} + \norm{u}_{L^2(\Ga)} = 1$. This completes the proof.
\end{proof}

\section*{Acknowledgement}
Patrik Knopf and Stefan Metzger were partially supported by the RTG 2339 ``Interfaces, Complex Structures, and Singular Limits''
of the German Science Foundation (DFG). 
Kei Fong Lam is partially supported by a Direct Grant (project 4053288) of the Chinese University of Hong Kong.
Chun Liu is partially supported by the NSF-DMS 1759536 ``Energetic Variational Approaches in Complex Fluids and Electrophysiology''
of the National Science Foundation (NSF).
The support is gratefully acknowledged. 


\footnotesize
\bibliographystyle{plain}
\bibliography{kllm}

\end{document}